\documentclass[11pt,reqno,a4paper]{amsart}
\usepackage{amsmath,amsfonts,amssymb,amsthm,amsxtra,latexsym,amscd,enumerate,verbatim}
\usepackage[margin=2.0cm]{geometry}
\usepackage[abbrev,lite,nobysame]{amsrefs}
\usepackage{graphics,graphicx}
\usepackage{subcaption}
\usepackage{tabularx}
\usepackage[usenames,dvipsnames]{color}
\usepackage{times}
\usepackage{bbm}
\usepackage[colorlinks=true, pdfstartview=FitV, linkcolor=blue, citecolor=blue, urlcolor=blue]{hyperref}
\usepackage{tikz}
\usepackage{enumerate,enumitem}
\usepackage{mathtools}
\makeatletter
\usepackage[utf8]{inputenc}
\usepackage{amsmath,amsfonts,amssymb,amsxtra,latexsym,amscd,enumerate,amsthm,verbatim}
\usepackage[abbrev,lite,nobysame]{amsrefs}
\usepackage{float}
\usepackage{tikz}
\usepackage{pgfplots}
\usepackage{caption}
\usepackage[colorlinks=true]{hyperref}
\hypersetup{urlcolor=blue, citecolor=blue}
\usepackage{hyperref}
\numberwithin{equation}{section}
\def\dd{{\mathrm{d}}}

\newcommand{\R}{\mathbb{R}}

\newcommand{\C}{\mathbb{C}}
\newcommand{\Z}{\mathbb{Z}}
\newcommand{\defeq}{\mathrel{\mathop:}=}

\newcommand{\inn}{{\quad\hbox{in } }}
\newcommand{\ve}{\varepsilon}
\newcommand{\be}{\begin{equation}}
\newcommand{\ee}{\end{equation}}

\newtheorem{lemma}{Lemma}[section]
\newtheorem{prop}{Proposition}[section]
\newtheorem{theorem}{Theorem}
\newtheorem{corollary}{Corollary}[section]
\newtheorem{remark}{Remark}[section]

\pgfplotsset{compat=1.18} 

\title[Clustered vortex helices in the 3D Euler equations]{Clustered vortex helices with compactly supported cross-sectional vorticity in the 3D Euler equations}
\author[A.~Averkiou]{Averkios Averkiou}
\address{\noindent A.~Averkiou: Department of Mathematical Sciences, University of Bath, BA2 7AY,
United Kingdom.}
\email{aa4119@bath.ac.uk}

\author[M.~Musso]{Monica Musso}
\address{\noindent M.~Musso: Department of Mathematical Sciences, University of Bath, BA2 7AY,
United Kingdom.}
\email{m.musso@bath.ac.uk}
\author[F.~Yu]{Fang Yu}
\address{\noindent F.~Yu: School of Mathematical Sciences, Chongqing University of Technology,
Chongqing, P.R. China}
\email{yufang@cqut.edu.cn}

\begin{document}
\begin{abstract}
We consider the three-dimensional incompressible Euler equations for helical flows without swirl. By adapting gluing techniques, we construct classical multi-vortex solutions in the whole space $\mathbb R^3$ exhibiting a cluster of collapsing helical filaments, providing the first such solutions whose cross-sectional vorticity remains compactly supported in $\mathbb R^2$ for all time. Our result complements previous work on collapsing configurations in $\mathbb R^3$ with rapidly decaying vorticity cores, as well as related variational constructions in infinite cylindrical domains.
\end{abstract}
\maketitle
\section{Introduction}
The motion of a three-dimensional inviscid and incompressible fluid without external forcing is described by the Euler equations in $\R^3$. For a sufficiently smooth velocity field $\mathbf{u}: \R^3\times [0,T) \to \R^3$ with $\nabla\cdot \mathbf{u} =0,$ vorticity $\vec{\omega}\defeq \nabla\times{\mathbf{u}}: \R^3\times [0,T) \to \R^3,$ stream function $\vec{\psi}: \R^3 \times [0,T) \to \R^3$ and prescribed initial data $\mathbf{u}(x,0)=u_0(x)$, the Euler system in vorticity formulation reads
\begin{equation}
\label{euler}
\begin{aligned}
\partial_{t}\vec \omega  +
(\mathbf{u}\cdot \nabla){\vec \omega}
&=( \vec \omega \cdot \nabla)\mathbf {u}  &&  \inn \, \R^3\times (0,T), \\
\quad \mathbf{u}  = \nabla \times \vec {\psi},\ &
-\Delta \vec \psi =  \vec \omega  && \inn \, \R^3\times (0,T),\\
{\vec \omega}(\cdot,0)  &=  \nabla \times u_0 && \inn \, \R^3.
\end{aligned}
\end{equation}
We are interested in the rigorous construction of  solutions to \eqref{euler} with vortex filament structure, where the vorticity remains highly concentrated around a smooth curve evolving over time. Despite the growing interest in the vortex dynamics of such flows in recent years, their study dates back to the foundational works of Helmholtz and Kelvin \cites{MR1579057,thomson1868vi}. Formal asymptotic considerations suggest that when the vorticity is concentrated in a thin tubular neighbourhood of radius $\ve>0$ around a smooth curve $\mathcal{G}(t),$ the associated filament evolves to leading order according to the binormal flow with velocity  $\mathcal{O}\left(|\log\ve|\right)$. More precisely, for an arclength parametrisation $\gamma(s,t)$ and a constant $\bar c$ proportional to the circulation, the formal derivations in \cites{da1906sul, levi1908sull}  indicate that the evolution is governed by
\[\partial_{t}\gamma=\bar{c}|\log\ve|\left(\gamma_{s}\times\gamma_{ss}\right)\ \quad \text{as}\quad \ve\to 0,\]
which one can also express as
\begin{equation}\label{binormal}
\partial_{\tau}\gamma=\bar{c}\varkappa\mathbf{B}_{{\mathcal{G}(\tau)}}, \quad t = \frac{\tau}{|\log\ve|},
\end{equation}
with $\varkappa$ denoting the curvature and $\mathbf{B}_{\mathcal{G}(\tau)}$ the binormal unit vector. In this regime, the \emph{vortex filament conjecture} concerns the rigorous verification that concentrated vorticity structures remain localised and evolve at leading order along the trajectories determined by \eqref{binormal}. Beyond the conditional result in \cite{MR3609248}, significant progress has been achieved in the case of circles and helices, which constitute symmetric solutions of \eqref{binormal}.

For circular vortex filaments, the validity of the binormal law \eqref{binormal} was first established in \cite{fraenkel1970steady} through the construction of \emph{vortex rings} propagating with constant speed $\mathcal{O}\left(|\log\ve|\right)$, with subsequent  contributions including \cites{MR3101789,MR422916,MR1738349,MR4404610, guo2026global}. In the helical framework, the vortex filament conjecture was rigorously verified in \cite{MR4417384}, where the authors constructed \emph{vortex helices}; see also \cites{qin2024concentrated,MR4555251,MR4534710,MR5036920} and the references therein. Of particular relevance to the present work are the helical multi-vortex solutions in \cites{MR4899797,cao2023clustered,MR5030554,cao2025co,MR5111731}, covering among others nearly parallel and collapsing configurations, while results on global well-posedness and long-time dynamics appear in  \cites{MR4809294,MR5107439,MR3320529}. Finally, helical structures have also been investigated in related models, such as the Ginzburg–Landau  \cites{MR4493621,DUAN2026114325}, the Gross–Pitaevskii  \cites{MR2181460,MR4540752}, and the Schrödinger map \cite{MR3449250} equations. 
\subsection{Statement of the main result}
In this paper, we rigorously construct classical global-in-time solutions $\vec{\omega}_{\ve}(x,t)$ to \eqref{euler} in the whole space $\R^3$, concentrating near multiple collapsing helices undergoing a common rigid rotation as $\ve\to0$. In particular, we provide the first such construction in which the associated cross-sectional vorticity remains compactly supported in $\R^2$ for all times. We refer to such solutions as \emph{clusters of vortex helices with compactly supported cross-sectional vorticity}.

In this regard, given an integer $N\geq1$ and $h>0$, we take planar points $P_i\defeq(a_i,b_i)$, $i=1,\dots,N$, and consider the evolving circular helices $\mathcal G_i^{\ve}(\tau)$ of radius $R_i=\sqrt{a_i^2+b_i^2}$ and common pitch parameter $h>0$, parametrised by arclength as
\begin{equation}\label{parhelix}
\gamma_i^{\ve}(s,\tau)=\begin{pmatrix}
a_i\cos\left(\frac{s}{\sqrt{h^2+R_i^2}}-\alpha\tau\right)-b_i\sin\left(\frac{s}{\sqrt{h^2+R_i^2}}-\alpha\tau\right)\\
a_i\sin\left(\frac{s}{\sqrt{h^2+R_i^2}}-\alpha\tau\right)+b_i\cos\left(\frac{s}{\sqrt{h^2+R_i^2}}-\alpha\tau\right)\\
\frac{hs}{\sqrt{h^2+R_i^2}}
\end{pmatrix},\qquad s\in\R,
\end{equation}
where $\alpha>0$ is given by \eqref{uniformspeed} below and $\tau=|\log\ve|\,t$. One can readily verify that each $\gamma_i^{\ve}$ has curvature $\frac{R_i}{R_i^2+h^2}$ and torsion $\frac{h}{R_i^2+h^2}$, and that all helices rotate with the same angular velocity $-\alpha$ in the rescaled time $\tau$. Equivalently,
\begin{equation}\label{helicesdefshorter}
\mathcal G_i^{\ve}(\tau)=\left\{\left(R_{\frac{x_3}{h}-\alpha\tau}P_i,x_3\right):x_3\in\R\right\},\quad \mbox{with}\quad  R_\theta=\begin{pmatrix}\cos\theta&-\sin\theta\\\sin\theta&\cos\theta\end{pmatrix}.
\end{equation}
Hereafter, we identify the base point $(a_i,b_i,0)$ of each helix at $\tau=0$ with the planar point $(a_i,b_i)$.

In view of our interest in clustered helical filaments, for a fixed $r_0>0$ we consider points $P_i=(a_i,b_i)$ in \eqref{parhelix} of the specific form
\begin{equation}\label{formpoints}
P_i=(r_0+\tilde{r},0)+ \frac{\widehat{P}_i}{|\log\ve|}, \quad |\tilde{r}|\leq C\frac{\log|\log\ve|}{|\log\ve|}, \quad |\widehat P_i|<\delta^{-1}, \quad i=1,\dots,N,
\end{equation}
where $\delta>0$ is sufficiently small and $C>0$, both independent of $\ve$. Since $P_i=(a_i,b_i)\to(r_0,0)$ as $\ve\to0$ for all $i\in\{1,\dots,N\}$, we infer that the helices $\mathcal G_i^{\ve}(\tau)$ in \eqref{helicesdefshorter} collapse onto the same limiting helix 
\begin{equation}\label{limitinghelix}
\mathcal G_0(\tau)=\left\{\left(R_{\frac{x_3}{h}-\alpha\tau}(r_0,0),x_3\right):x_3\in\R\right\},
\end{equation}
for every fixed $\tau\in\R$.

Furthermore, we assume that the points $\widehat{P}_i$ in \eqref{formpoints} are uniformly separated, namely there exists a small constant $d>0$ independent of $\ve$ such that 
\begin{equation}\label{unifdistance}
|\widehat P_j-\widehat P_i|\geq d,\qquad i\ne j,
\end{equation}
and we also impose a reflection symmetry on the set $\{\widehat{P}_1,\dots,\widehat{P}_N\}$ with respect to the first coordinate axis, i.e.
\begin{equation}\label{symmetryforpoints}
\mathcal{P}=(\mathcal{P}_{1},\mathcal{P}_{2}) \in \{\widehat{P}_1,\dots,\widehat{P}_N\} \iff (\mathcal{P}_{1},-\mathcal{P}_{2}) \in \{\widehat{P}_1,\dots,\widehat{P}_N\}.\end{equation} 
Finally, for nonzero amplitudes $\kappa_1,\dots,\kappa_N\in\R$ such that $\sum\limits_{i=1}^{N}\kappa_i>0$, the points $\widehat{P}_i$ admit the decomposition
\begin{equation}\label{points1}
\widehat{P}_i=P_i^{b}+q_i, \quad i=1,\dots,N,
\end{equation}
so that if we write in anisotropic coordinates 
\begin{equation}\label{rescaledmainpoints}
{P}^{b}_i=\left(\frac{h\,\widetilde{P}^{b}_{i,1}}{\sqrt{h^2+r_0^2}},\widetilde{P}^{b}_{i,2}\right), \quad \widetilde{P}^{b}_i=\left(\widetilde{P}^{b}_{i,1},\widetilde{P}^{b}_{i,2}\right),
\end{equation}
the configuration $\left(\widetilde{P}^{b}_{1},\dots,\widetilde{P}^{b}_{N}\right)$ solves the system of balancing conditions
\begin{equation}\label{balancingcond}
\sum_{j\ne i}\kappa_j\frac{\widetilde P_i^b-\widetilde P_j^b}{|\widetilde P_i^b-\widetilde P_j^b|^2}
=\left(\frac{\kappa_ihr_0}{2(h^2+r_0^2)^{\frac32}}+\frac{\alpha hr_0}{\nu'(1)\sqrt{h^2+r_0^2}}\right)\vec e_1, \qquad \vec e_1=(1,0), \qquad i=1,\dots,N,
\end{equation}
where $\nu'(1)<0$ is determined by the profile in \eqref{defGamma} and $\alpha>0$ is the common rotation parameter. Here and throughout this work, the system \eqref{balancingcond} uniquely determines $\alpha>0$ as  
\begin{equation}\label{uniformspeed}
\alpha=\frac{-\nu'(1)}{2(h^2+r_0^2)}\frac{\sum\limits_{i=1}^{N} \kappa_i^2}{\sum\limits_{i=1}^{N}\kappa_{i}},
\end{equation}
while the terms ${q}_i$ in \eqref{points1} are lower order perturbations, with $|q_i|\to 0$ as $\ve \to 0$. 

Notably, admissible configurations of points $P^{b}=\left(P^{b}_1,\dots,P^{b}_N\right)$ satisfying \eqref{unifdistance}, \eqref{symmetryforpoints}, \eqref{rescaledmainpoints} and \eqref{balancingcond} are exhibited in \cites{MR4404610,MR4350893}. In this context, 
assuming that $\kappa_j=\kappa_i$ whenever $\widetilde P_j^b=(\widetilde P_{i,1}^b,-\widetilde P_{i,2}^b)$ so that reflected points carry equal amplitudes, we call a solution $P^{b}$ as in \eqref{rescaledmainpoints} \emph{nondegenerate} if the linearisation of \eqref{balancingcond} at $(\widetilde P_1^b,\dots,\widetilde P_N^b)$, restricted to the reflection-symmetric subspace, has a one-dimensional real kernel consisting of common translations in the first coordinate. A more detailed analysis of this condition and its consequences is deferred to Section \ref{reduced}, where it plays a crucial role in locating the points $\widehat{P}_i$ in \eqref{points1} as small perturbations of $P_i^{b}$,\, $i=1,\dots,N,$ thereby determining the adjusted vortex centers of the helices associated with an exact classical solution of \eqref{euler}.

In addition, for fixed $\gamma>3,$ we let $\nu$ be the unique positive radial solution of
\[\Delta\nu+\nu^\gamma=0\quad\mbox{in }B_1(0)\subset\R^2,\qquad
\nu=0\quad\mbox{on }\partial B_1(0).\]
Since $\nu'(1)<0$ due to the Hopf lemma, the radial profile
\begin{equation}\label{defGamma}
\Gamma(z)=
\begin{cases}
\nu(|z|),&|z|\le1,\\
\nu'(1)\log|z|,&|z|>1,
\end{cases}
\end{equation}
satisfies the semilinear elliptic equation
\begin{equation}\label{modelsemil}
\Delta_z\Gamma+\Gamma_+^\gamma=0\quad\mbox{in }\R^2,\qquad
\{\Gamma>0\}=B_1(0),
\end{equation}
where $s_+\defeq\max\{0,s\}$. We also define the positive constant
\begin{equation}\label{circulationconstant}
c_*\defeq\frac{-2\pi h\nu'(1)}{\sqrt{h^2+r_0^2}},
\end{equation}
which depends only on the fixed parameters $\gamma>3$ and $r_0,\,h>0$.

The following Theorem is the main result of this work.
\begin{theorem}\label{maintheorem}
Fix an integer $N\ge1$, constants $r_0>0$, $h>0$, and nonzero amplitudes $\kappa_1,\dots,\kappa_N\in\R$ satisfying $\sum_{i=1}^N\kappa_i>0$. In addition, fix $\gamma>3,$ and let $\alpha$ and $c_{*}$ be given by \eqref{uniformspeed} and \eqref{circulationconstant}, respectively. Assume that the configuration $P^b=(P_1^b,\dots,P_N^b)$ is uniformly separated and reflection-symmetric as in \eqref{unifdistance} and \eqref{symmetryforpoints}, and that the associated configuration $\widetilde{P}^b$ defined by \eqref{rescaledmainpoints} is a nondegenerate solution of \eqref{balancingcond}. Then, for all sufficiently small $\ve>0$, there exist $\tilde r_*\in\R$, $q_1,\dots,q_N\in\R^2$, points 
\begin{equation}\label{pointstheorem}
P_i=(r_0+\tilde r_*,0)+\frac{1}{|\log\ve|}\widehat P_i,\qquad
\widehat P_i=P_i^b+q_i,\qquad
|\tilde r_*|+|q_i|\lesssim\frac{\log|\log\ve|}{|\log\ve|}, \qquad i=1,\dots,N,
\end{equation}
and a classical global-in-time solution $\vec\omega_\ve(x,t)$ of \eqref{euler}, such that for every $\tau\in\R$ it holds
\begin{equation*}
\vec\omega_\ve\left(x,|\log\ve|^{-1}\,\tau\right)-c_*\sum_{i=1}^N\kappa_i\delta_{\mathcal G_i^\ve(\tau)}\mathbf T_{\mathcal G_i^\ve(\tau)}\rightharpoonup0\qquad\mbox{as }\ve\to0
\end{equation*}
in the sense of measures, where $\mathcal G_i^\ve(\tau)$ is the helix in \eqref{helicesdefshorter} determined by $P_i$ in \eqref{pointstheorem}, while $\delta_{\mathcal G_i^\ve(\tau)}$ and $\mathbf T_{\mathcal G_i^\ve(\tau)}$ denote its arclength measure and unit tangent oriented by increasing $x_3$.

Moreover, for every fixed $\zeta\in(0,1)$ and all sufficiently small $\ve>0$, the cross-sectional vorticity satisfies
\begin{equation}\label{remark1prop}
\operatorname{supp}\vec\omega_\ve\left(\cdot,x_3,|\log\ve|^{-1}\,\tau\right)
\subset\bigcup_{i=1}^N B_{\ve^{1-\zeta}}
\left(R_{\frac{x_3}{h}-\alpha\tau}P_i\right),
\qquad (x_3,\tau)\in\R^2,
\end{equation}
where the balls are pairwise disjoint. The individual signed core circulations satisfy
\begin{equation}\label{circulationfull}
\int_{B_{\ve^{1-\zeta}}(P_i)}\omega_{\ve,3}(x',0,0)\,\mathrm dx'
=c_*\kappa_i+\mathcal{O}\Big(\frac{\log|\log\ve|}{|\log\ve|}\Big),\qquad i=1,\dots,N,
\end{equation}
where $x'=(x_1,x_2)$ and $\omega_{\ve,3}$ is the third component of $\vec\omega_\ve$.

Finally, for every $\tau\in\R$, we have
\[
\vec\omega_\ve\left(x,|\log\ve|^{-1}\,\tau\right)
\rightharpoonup
c_*\left(\sum_{i=1}^N\kappa_i\right)
\delta_{\mathcal G_0(\tau)}\mathbf T_{\mathcal G_0(\tau)}
\qquad\mbox{as }\ve\to0\]
in the sense of measures, where $\mathcal G_0(\tau)$ is the limiting helix in \eqref{limitinghelix}, with the same conventions for its arclength measure and orientation.
\end{theorem}
\begin{remark}
A related multi-filament helical solution of \eqref{euler} in domains $\Omega\subsetneq\R^3$ with bounded cross-sections is obtained in \cite{MR5111731}, where the authors consider vorticity initially concentrated near helices of distinct radii and show that the concentration persists for $\tau\in[0,T]$, with the cross-sectional vorticity supported in a thin annulus. Although their result applies to a general class of concentrated initial data, the filaments remain uniformly separated at distances $\mathcal{O}(1)$, thus mutual interactions are asymptotically negligible compared to self-induced motion. The polygonal solution of \cite{MR5036920} belongs to the same separation regime. In contrast, our filament separation scale is $\mathcal{O}(|\log\ve|^{-1})$, hence mutual interactions and self-induced motion contribute at the same leading order, with our rigidly rotating configurations remaining concentrated for all $\tau\in\R$. A similar balance is studied in \cites{MR4854830, cao2025dynamics}, where the authors derived leapfrogging dynamics for vortex rings and vortex helices, respectively. Finally, \cite{cao2023clustered} uses variational methods to construct clustered helical solutions in infinite cylinders. However, the analysis is crucially based on a $C^1$ expansion of the Green’s function of a uniformly elliptic operator in bounded planar domains, thus it does not appear to extend directly to our setting in $\mathbb R^2$, where uniform ellipticity is lost at infinity.
\end{remark}
\subsection{Helical symmetry with no swirl reduction}
Our construction is carried out within the class of helically symmetric flows with no helical swirl \cites{MR2505860,MR1717127}, where the Euler dynamics in \eqref{euler} reduces to a transport-elliptic system in $\R^2$. More precisely, writing $x=(x',x_3)\in\R^3$, $x'=(x_1,x_2)\in\R^2$ and  $t=|\log\ve|^{-1}\tau$, we consider the pair $\left(w(x',\tau),\psi(x',\tau)\right)$ solving
\begin{equation}\label{PB}
\begin{aligned}
\left\{\begin{aligned}
|\log\ve|\partial_{\tau}w + \nabla^\perp \psi \cdot \nabla w &=0 && {\mbox {in}} \quad \R^2 \times (-\infty, +\infty),
\\
- \nabla\cdot (K \nabla \psi) &= w && {\mbox {in}} \quad \R^2 \times (-\infty,+\infty),
\end{aligned}
\right.
\end{aligned}
\end{equation}
with $(a,b)^\perp = (b,-a)$ and $K (x_1,x_2) $ denoting the coefficient matrix
\[
K(x_1 , x_2 ) = \frac{1}{h^2+x_1^2+x_2^2}
\left(
\begin{matrix}
h^2+x_2^2 & -x_1  x_2\\
-x_1 x_2 & h^2+x_1^2
\end{matrix}
\right).\]
If the pair $\big(w(x',\tau),\psi(x',\tau)\big)$ solves \eqref{PB}, then the vorticity vector
\begin{align}
\label{omegaHelical}
\vec \omega(x,t)  =  \frac{1}{h} w\left( R_{-\frac {x_3}h} x',|\log\ve|\, t\right)  \, \begin{pmatrix}-x_2\\x_1\\h\end{pmatrix}, \quad R_{\theta}=\begin{pmatrix} \cos\theta& -\sin\theta\\ \sin\theta& \cos\theta \end{pmatrix},
\end{align}
is helically symmetric and satisfies \eqref{euler}. Moreover, seeking rotating solutions of the form
\begin{equation}
\label{rotansatz}
w (x', \tau) = W \left( R_{ \alpha \tau } x' \right), \quad \psi (x',\tau) = \Psi \left( R_{ \alpha \tau } x' \right), \quad x'=(x_1,x_2)\in\R^2,
\end{equation}
reduces \eqref{PB} to the semilinear elliptic equation
\begin{equation}\label{semilin}
\nabla_{\tilde{x}}\cdot\left(K(\tilde x)\nabla_{\tilde{x}}\Psi\right)+F\left(\Psi-\frac{\alpha}{2}|\log\ve||\tilde{x}|^2\right)=0 \quad \mbox{in} \quad \R^2, \quad \tilde{x}=R_{\alpha\tau}x'.
\end{equation}
Our aim is to find a suitable nonlinear function $F$ and a stream function $\Psi$ for \eqref{semilin}, so that the 
vorticity $W(\tilde x)\defeq F\left(\Psi-\frac{\alpha}
{2}|\log\ve||\tilde x|^2\right)$
satisfies
\begin{equation}\label{propvorticity}
W(\tilde x)\sim c_{*}\sum_{i=1}^{N}\kappa_i\delta_{P_i},\quad \operatorname{supp}W(\tilde x) \subset \bigcup_{i=1}^{N} B_{\tilde{\rho}_{\ve}}(P_i),
\end{equation}
with $c_{*}$ given by \eqref{circulationconstant}, $P_i$ as in \eqref{formpoints}, and $\tilde\rho_\ve=o\left(|\log\ve|^{-1}\right)$ as $\ve\to0$. In this way, the cross-sections are pairwise disjoint since $|P_i-P_j|\geq\frac d{|\log\ve|}$ for $i\ne j$, and the profile $w(x',\tau)=W(R_{\alpha\tau}x')$ in \eqref{rotansatz} is rigidly rotating, while remaining compactly supported for all times in the rotated union of planar vortex cores. Consequently, \eqref{omegaHelical} and \eqref{helicesdefshorter} yield a classical $3$D multi-filament solution $\vec\omega(x,t)$, which is confined in thin non-overlapping helical tubes globally in time.
\begin{remark}\label{decoupling}
Theorem \ref{maintheorem} is proved via an Inner–Outer gluing scheme, in the spirit of \cite{MR4899797}. A key novelty of the present work is the choice of nonlinearity $F$ in \eqref{semilin}, motivated by the equation $\Delta u + u^{\gamma}_{+}=0 \,$ in $\,\R^2$, where $u_{+}\defeq\max(0,u)$ and $\gamma>3$. This setting differs significantly from that in \cite{MR4899797}, where $F(s)\sim e^{s}$. While the latter yields decaying cross-sectional vorticity, the current choice $F(s)\sim s_{+}^{\gamma}$ gives rise to a compactly supported one, as in \eqref{remark1prop}. Moreover, since our construction involves the operator  $\Delta+\gamma\Gamma_{+}^{\gamma-1}$ in $\R^2$, with $\Gamma$ given in \eqref{defGamma}, the presence of the radial kernel element $Z_{0}=\mathcal{O}(\log(2+|y|))$ in \eqref{elemkernel} poses a major technical difficulty, since the quantity $
\int_{\R^{2}} \gamma \Gamma_{+}^{\gamma-1} Z_{0} \neq 0$ induces a strong coupling in the Inner-Outer system (see Proposition \ref{propsec8} and specifically \eqref{finalestHiZ_0}). In contrast, the Liouville operator $\Delta+e^{u}$ employed in \cite{MR4899797}, with $\Delta u + e^{u}=0$ in $\R^2$, has a bounded radial kernel element $\widetilde {Z}_{0}$ with $\int_{\R^{2}} e^{u}\widetilde {Z}_{0}=0$, hence the corresponding Inner–Outer system is sufficiently decoupled. To resolve this, we introduce perturbations \(\mu_j^*\) in the scaling parameters of Proposition \ref{propscaling} to cancel the constant outer contribution in the radial projection. However, since both the leading scales $\mu_j^{0}$ and the outer solution depend on the full scales $\mu_j$, the matching conditions are implicit. Proposition \ref{propsec8} establishes simultaneous solvability through a contraction in the logarithms of the full scales for each trial inner correction, followed by a contraction for the projected Inner problem. The estimates for the outer solution constructed in Section \ref{subsecouter} verify a posteriori that the resulting scaling perturbations satisfy the smallness bounds required by the gluing scheme.
\end{remark}
Following the preceding discussion, in the sequel we write $x$ instead of $\tilde x$ in \eqref{semilin} for ease of notation, and dedicate the rest of the paper to solving the problem
\begin{equation}\label{eqtosolve}
\nabla_x\cdot(K\nabla_x\Psi)+F\left(\Psi-\frac{\alpha}{2}|\log\ve|\,|x|^2\right)=0\quad\mbox{in }\R^2.
\end{equation}
As explained earlier, to establish Theorem \ref{maintheorem} it suffices to select a suitable nonlinearity $F$ and construct a stream function $\Psi$ for \eqref{eqtosolve}, so that the induced vorticity $W(x)\defeq F\left(\Psi-\frac{\alpha}{2}|\log\ve|\,|x|^2\right)$ has the properties in \eqref{propvorticity}. 
\subsection{Structure of the paper}
The organisation of this work is as follows. In Section \ref{Section 2}, we construct a local regularised approximation of the Green's function of $\nabla_x\cdot(K\nabla_x\cdot)$ in $\R^2$, based on a radial profile $\Gamma$ solving a semilinear elliptic equation and on a suitable change of variables. Section \ref{section3} extends this construction to a global multi-vortex approximation, while Section \ref{Section 4} is devoted to the properties of the associated scaling parameters. In Section \ref{Section 5}, we choose a nonlinearity $F$ producing compactly supported vorticity and derive the corresponding error estimates. Section \ref{Section 6} introduces the \emph{Inner-Outer gluing scheme} and describes the strategy for perturbing the approximation into an exact solution, relying on the linear theories of Section \ref{Section 7}. Finally, in Section \ref{Section projected} we solve a projected problem, and in Section \ref{reduced} we suitably adjust the vortex centers to obtain a genuine solution.
\section{Local approximate stream function}\label{Section 2} 
Let $h>0 $ and consider a point $P=(a,b)\in\R^2\setminus\{0\}$. This section is concerned with the construction of an approximate  stream function $\Psi$ with the property 
\begin{equation}\label{ori}
-\nabla_{x}\cdot (K\nabla_{x}\Psi) \sim c\,\delta_P
\end{equation}
locally around $P$, with
\begin{equation}\label{operL}
L_x=\nabla_{x}\cdot (K\nabla_{x}\cdot)
\quad \text{and}\ \quad
K=\frac{1}{h^2+x_1^2+x_2^2}
\begin{pmatrix}
h^2+x_2^2 & -x_1x_2 \\
-x_1x_2 & h^2+x_1^2
\end{pmatrix}.
\end{equation}
As a preliminary step, we analyse the behaviour of $L_x$ in a small neighbourhood of the point $P$, which will enable us to derive a local approximation of the Green’s function for \eqref{operL}. 

Following \cite{MR4899797}, we introduce the linear change of variables
\begin{equation}\label{var}
x-P=A[P]z,\quad
A[P]=
\begin{pmatrix}
\frac{ah}{R\sqrt{h^2+R^2}}
& -\frac{b}{R}\\
\frac{bh}{R\sqrt{h^2+R^2}}
& \frac{a}{R}
\end{pmatrix}, \quad\mbox{with } R=|P|,
\end{equation}
which normalises the coefficient matrix $K$ in \eqref{operL} at the point $P=(a,b)$.

The following Proposition holds.
\begin{prop}\label{expansionofL}
Let $z$ be the variable defined in \eqref{var}. Then, the operator $L_x$ in \eqref{operL} can be expressed as  
\[L_{x}=\Delta_{z}+B_0,\]
where
\begin{align*}
B_{0}
&= \left(\frac{h^2(R^2-r^2)+z_2^2(h^2+R^2)}{h^2(h^2+r^2)}\right)\partial_{z_1z_1}
+\frac1{(h^2+r^2)}\bigg(\Big(z_1\frac h{\sqrt{h^2+R^2}}+R\Big)^2-r^2\bigg)\partial_{z_2z_2} \nonumber\\
 &\hspace{5mm}- 2\frac{\sqrt{h^2+R^2}}{h(h^2+r^2)}z_2\left(z_1\frac{h}{\sqrt{h^2+R^2}}+R\right)
\partial_{z_1z_2}
 \nonumber\\
&\hspace{5mm}-\frac{z_1h^2+Rh\sqrt{h^2+R^2}}{h^2(h^2+r^2)}\left(1+\frac{2h^2}{h^2+r^2}\right)
\partial_{z_1}
 - \frac{z_2}{h^2+r^2}\left(1+\frac{2h^2}{h^2+r^2}\right)\partial_{z_2},
\end{align*}
and
\begin{equation*}
r^2=|x|^2=R^2+2R\frac h{\sqrt{h^2+R^2}}z_1 + \frac{h^2}{h^2+R^2}z_1^2+z_2^2.
\end{equation*}
\end{prop}

\begin{proof}
See \cite{MR4899797}*{Proposition 2.1}.
\end{proof}

Given any small $\delta>0$, Proposition \ref{expansionofL} shows that in the bounded region $|z|<\delta$, the operator $L_x$ in \eqref{operL} may be regarded as a variable-coefficient perturbation of the Laplacian whose second-order coefficients vanish at the origin. More precisely, we have
\begin{equation}\label{oper}
L_x=\Delta_z+B_{0},
\end{equation}
where $B_{0}$ admits the expansion 
\begin{equation*}
\begin{aligned}
B_0
&= \left(-2\frac{Rh}{(h^2+R^2)^{\frac 3 2}}z_1+\mathcal O(|z|^2)\right)\partial_{z_1z_1}
+\mathcal{O}(|z|^2)\partial_{z_2z_2} -\left(2\frac {R}{h\sqrt{h^2+R^2}}z_2+\mathcal{O}(|z|^2)\right)\partial_{z_1z_2}  \nonumber \\
&-\left(\frac{R}{h\sqrt{h^2+R^2}}\Big(1+\frac{2h^2}{h^2+R^2}\Big)+\mathcal{O}(|z|)\right)
\partial_{z_1} -\left(\frac{z_2}{h^2+R^2}\Big(1+\frac{2h^2}{h^2+R^2}\Big)
+\mathcal{O}\left(|z|^2\right)\right)
\partial_{z_2}.
\end{aligned}
\end{equation*}
Furthermore, \eqref{ori} is now formulated as
\begin{equation}\label{oriz}
-(\Delta_z+B_0)\psi \sim \frac{c}{\det{A[P]}}\,\delta_{0}, 
\qquad 
\psi(z)=\Psi\bigl(P+A[P]z\bigr),
\end{equation}
where $\delta_{0}$ represents a Dirac delta at the origin and $A[P]$ is the matrix employed in the change of variables \eqref{var}. Consequently, an approximate regularisation of the Green’s function for $L_x$ can be obtained through a regularisation of the Green's function of $\Delta_z$. 

Motivated by this observation, recalling the radial profile $\Gamma$ defined in \eqref{defGamma}, for $\ve,\mu>0$ we introduce the rescaled profile
\begin{equation}\label{erescgamma}
\Gamma_{\ve\mu}(z)=\Gamma\left(\frac{z}{\ve\mu}\right),
\end{equation}
so that \eqref{modelsemil} yields
\begin{equation}\label{rescaledequation123}
    \Delta_z\Gamma_{\ve\mu}(z)+\frac{1}{\ve^{2}\mu^2}\left(\Gamma_{\ve\mu}(z)\right)_+^\gamma=0
\quad  \mbox{in}\,\, \mathbb{R}^2.
\end{equation}
In the rest of this section, we focus on the region $|z|<\delta$ for a small $\delta>0$, and define
\begin{equation}\label{defofgammahat}
\widehat{\Gamma}_{\ve\mu}(z)=\Gamma\left(\frac{z}{\ve\mu}\right)-\nu'(1)|\log(\ve\mu)|,
\end{equation}
which satisfies \[-\Delta_z\widehat{\Gamma}_{\ve\mu}(z) \rightharpoonup \hat{c}\,\delta_0  \quad  \mbox{as} \quad \ve\mu \to 0, \quad \mbox{with} \quad \hat{c}=\int_{\R^2} \Gamma^{\gamma}_{+}=-2\pi\nu'(1).\]
Based on the model problem \eqref{rescaledequation123}, we construct an approximate solution of \eqref{oriz} via elliptic singular perturbation methods, with \eqref{defofgammahat} as the reference profile.

We obtain the following Proposition.

\begin{prop}\label{propapprox}
Given $h>0$ and $P=(a,b)\in\R^2\setminus\{0\}$, let $R=\sqrt{a^2+b^2}$. For any $\mu>0$ and the variable $z$ in \eqref{var}, consider the profile $\widehat{\Gamma}_{\ve\mu}(z)$ in \eqref{defofgammahat} and define the approximate solution of \eqref{oriz} as
\[\psi_{\mu}(z)=\widehat{\Gamma}_{\ve\mu}(z)\left(1+c_1z_1+c_2|z|^2\right)+\frac{R^3}{2h(h^2+R^2)^{\frac32}}H_1(z),\]
where
\[c_1=\frac{Rh}{2(h^2+R^2)^{\frac32}},\qquad c_2=\frac{3h^2R^2+R^4}{8(h^2+R^2)^3},\]
and $H_1$  satisfies
\[\Delta_zH_1+\frac{\operatorname{Re}(z^3)}{\ve^2\mu^2|z|^2}
\left(\Gamma''\left(\frac{|z|}{\ve\mu}\right)-\frac{\Gamma'\left(\frac{|z|}{\ve\mu}\right)}{\frac{|z|}{\ve\mu}}\right)=0.\]
In the rescaled variable $y=\frac{z}{\ve\mu}$ and for any small $\delta>0$, in the region $|y|<\frac{\delta}{\ve\mu}$ we obtain
\[
\ve^2\mu^2L_x(\psi_\mu)=\Delta_y\Gamma(y)+\frac{3Rh^2+R^3}{2h(h^2+R^2)^{\frac32}}\ve\mu y_1\left(\Gamma(y)\right)_+^\gamma+\ve^2\mu^2E_*,\]
where $E_*$ is of class $C^2$ in $z=\ve\mu y$ for each fixed $\ve\mu>0$, with $\sup\limits_{|z|<\delta}|E_*(z)|\le C$ independently of $\ve\mu$. 

In terms of the $x$ variable, the approximate stream function is given by
\begin{equation}\label{approxinx}
\Psi_{\mu,P}(x)=\psi_\mu\left(A[P]^{-1}(x-P)\right),
\end{equation}
where $A[P]$ is defined in \eqref{var}.
\end{prop}

\begin{proof}
We first consider $\delta>0$ small and introduce the rescaled variable $y=\frac{z}{\ve\mu}$, so that for $|y|<\frac{\delta}{\ve\mu}$ we use \eqref{oper} to write
\begin{equation}\label{Bfirsterror}
\begin{aligned}
\ve^2\mu^2B_0(\ve\mu y)[\widehat\Gamma_{\ve\mu}]
&=-\frac{2Rh}{(h^2+R^2)^{\frac32}}\ve\mu y_1\partial_{y_1y_1}\widehat\Gamma_{\ve\mu}
-\frac{2R}{h\sqrt{h^2+R^2}}\ve\mu y_2\partial_{y_1y_2}\widehat\Gamma_{\ve\mu}\\
&\quad-\frac{R}{h\sqrt{h^2+R^2}}\left(1+\frac{2h^2}{h^2+R^2}\right)\ve\mu\partial_{y_1}\widehat\Gamma_{\ve\mu}
+\ve^2\mu^2E_1,
\end{aligned}
\end{equation}
where $E_1$ is of class $C^2$ in $z=\ve\mu y$ for each fixed $\ve\mu>0$ and uniformly bounded as $\ve\mu\to0$.

In addition, direct calculations give
\begin{equation}\label{radialders}
\partial_{y_1}\widehat\Gamma_{\ve\mu}=\Gamma'\frac{y_1}{|y|},\quad
\partial_{y_1y_1}\widehat\Gamma_{\ve\mu}=\left(\Gamma''-\frac{\Gamma'}{|y|}\right)\frac{y_1^2}{|y|^2}+\frac{\Gamma'}{|y|},\quad
\partial_{y_1y_2}\widehat\Gamma_{\ve\mu}=\left(\Gamma''-\frac{\Gamma'}{|y|}\right)\frac{y_1y_2}{|y|^2},
\end{equation}
where $\Gamma'$ and $\Gamma''$ denote radial derivatives of $\Gamma(|y|)=\Gamma_{\ve\mu}(z)$ in \eqref{erescgamma}.

We then turn to complex variables by identifying $y=(y_1,y_2)\in\R^2$ with $y=y_1+\mathrm{i}y_2\in\C$. Substituting \eqref{radialders} into \eqref{Bfirsterror} and using the identities
\[y_1y_2^2=\frac{y_1|y|^2}{4}-\frac{\operatorname{Re}(y^3)}4,\qquad
y_1^3=\frac{3y_1|y|^2}{4}+\frac{\operatorname{Re}(y^3)}4,\]
we arrive at
\begin{equation}\label{bgamma2}
\begin{aligned}
\ve^2\mu^2B_0(\ve\mu y)[\widehat\Gamma_{\ve\mu}]
=&-\frac{Rh}{(h^2+R^2)^{\frac32}}\frac{\Gamma'\ve\mu y_1}{|y|}
-\frac{R^3+4Rh^2}{2h(h^2+R^2)^{\frac32}}\left(\Gamma''+\frac{\Gamma'}{|y|}\right)\ve\mu y_1\\
&+\frac{\ve\mu R^3}{2h(h^2+R^2)^{\frac32}}\left(\Gamma''-\frac{\Gamma'}{|y|}\right)\frac{\operatorname{Re}(y^3)}{|y|^2}
+\ve^2\mu^2E_1,
\end{aligned}
\end{equation}
with $E_1$ as in \eqref{Bfirsterror}.

To eliminate the first error term in \eqref{bgamma2}, we modify the initial approximation $\widehat\Gamma_{\ve\mu}$ by setting
\begin{equation}\label{psi1}
\psi_1(\ve\mu y)=(1+c_1\ve\mu y_1)\widehat\Gamma_{\ve\mu},\qquad c_1=\frac12\frac{Rh}{(h^2+R^2)^{\frac32}},
\end{equation}
where we obtain
\begin{align*}
\ve^2\mu^2L_x(\psi_1)
&=\Delta_y\Gamma-\frac{3Rh^2+R^3}{2h(h^2+R^2)^{\frac32}}\left(\Gamma''+\frac{\Gamma'}{|y|}\right)\ve\mu y_1
+\frac{\ve\mu R^3}{2h(h^2+R^2)^{\frac32}}\left(\Gamma''-\frac{\Gamma'}{|y|}\right)\frac{\operatorname{Re}(y^3)}{|y|^2}\\
&\quad+\ve^2\mu^2B_0(\ve\mu y)\big[c_1\ve\mu y_1\widehat\Gamma_{\ve\mu}\big]+\ve^2\mu^2E_1.
\end{align*}
In addition, due to the expression
\[\ve^2\mu^2B_0(\ve\mu y)[c_1\ve\mu y_1\widehat\Gamma_{\ve\mu}]
=-\frac{c_1\ve^2\mu^2R}{h\sqrt{h^2+R^2}}\Big(1+\frac{2h^2}{h^2+R^2}\Big)\widehat\Gamma_{\ve\mu}+\ve^2\mu^2E_2,\]
where $E_2$ is another $C^2$ function of $z=\ve\mu y$ which is uniformly bounded as $\ve\mu\to0$, we infer that the profile in \eqref{psi1} satisfies
\begin{equation}\label{jae}
\begin{aligned}
\ve^2\mu^2L_x(\psi_1)
&=\Delta_y\Gamma-\frac{3Rh^2+R^3}{2h(h^2+R^2)^{\frac32}}\left(\Gamma''+\frac{\Gamma'}{|y|}\right)\ve\mu y_1
-\frac{\ve^2\mu^2(3h^2R^2+R^4)}{2(h^2+R^2)^3}\widehat\Gamma_{\ve\mu}\\
&\quad+\frac{\ve\mu R^3}{2h(h^2+R^2)^{\frac32}}\left(\Gamma''-\frac{\Gamma'}{|y|}\right)\frac{\operatorname{Re}(y^3)}{|y|^2}
+\ve^2\mu^2(E_1+E_2).
\end{aligned}
\end{equation}
To remove the error term proportional to $\widehat\Gamma_{\ve\mu}$ in \eqref{jae}, we first note that
\[
\Delta_y\Big(c_2\ve^2\mu^2|y|^2\widehat\Gamma_{\ve\mu}\Big)
=c_2\ve^2\mu^2\Big(4\widehat\Gamma_{\ve\mu}+4y\cdot\nabla_y\Gamma+|y|^2\Delta_y\Gamma\Big),\]
and then adjust the approximation to
\[\psi_2(\ve\mu y)=\left(1+c_1\ve\mu y_1+c_2\ve^2\mu^2|y|^2\right)\widehat\Gamma_{\ve\mu},\qquad
c_2=\frac{3h^2R^2+R^4}{8(h^2+R^2)^3},\]
for which it then holds
\begin{equation}\label{errorsecondapprox}
\begin{aligned}
\ve^2\mu^2L_x(\psi_2)
&=\Delta_y\Gamma-\frac{3Rh^2+R^3}{2h(h^2+R^2)^{\frac32}}\left(\Gamma''+\frac{\Gamma'}{|y|}\right)\ve\mu y_1
+\frac{\ve\mu R^3}{2h(h^2+R^2)^{\frac32}}\left(\Gamma''-\frac{\Gamma'}{|y|}\right)\frac{\operatorname{Re}(y^3)}{|y|^2}\\
&\quad+\ve^2\mu^2(E_1+E_2+E_3),
\end{aligned}
\end{equation}
with $E_3$ having the same properties as $E_1$ and $E_2$ above.

To cancel the mode-three term in \eqref{errorsecondapprox}, we now introduce the radial correction
\[h_1(s)=s^3\int_s^1\frac{\dd r}{r^7}\int_0^r\frac{t^5}{\ve^2\mu^2}
\Bigg(\Gamma''\bigg(\frac{|t|}{\ve\mu}\bigg)-\frac{\Gamma'\left(\frac{|t|}{\ve\mu}\right)}{\frac{|t|}{\ve\mu}}\Bigg)\dd t,\]
solving the equation
\[h_1''+\frac1s h_1'-\frac9{s^2}h_1+\frac{s}{\ve^2\mu^2}
\Bigg(\Gamma''\left(\frac{|s|}{\ve\mu}\right)-\frac{\Gamma'\left(\frac{|s|}{\ve\mu}\right)}{\frac{|s|}{\ve\mu}}\Bigg)=0.\]
Due to the radial equation for $\Gamma$ in \eqref{modelsemil} and its smoothness near the origin, we have
\[\Gamma''(u)-\frac{\Gamma'(u)}u=
\begin{cases}
\mathcal{O}(u^2),&\mbox{as }u\to0,\\[1mm]
\mathcal{O}(u^{-2}),&\mbox{for }u\ge1.
\end{cases}\]
Changing variables in the definition of $h_1$, we obtain
\begin{equation}\label{Rescaledh1}
h_1(\ve\mu s)=\ve\mu s^3\int_s^{\frac{1}{\ve\mu}}\frac{\dd u}{u^7}
\int_0^u v^5\left(\Gamma''(v)-\frac{\Gamma'(v)}v\right)\dd v,
\end{equation}
thus estimating this formula and its derivatives gives
\[|h_1^{(k)}(s)|\le C\frac{s^{3-k}}{(\ve\mu+s)^2},\qquad 0<s<\delta,\quad k=0,1,2.\]
In particular, $h_1$ is uniformly bounded on $[0,\delta]$ and satisfies $|h_1(\ve\mu s)|\le C\ve\mu s^3$ for $0\le s\le1$.

Writing $y=|y|e^{i\theta}$, we set $H_1(z)=h_1(|z|)\cos(3\theta)$ and extend it continuously by $H_1(0)=0$ to obtain a solution of
\[\Delta_y H_1(\ve\mu y)+\ve\mu\Big(\Gamma''-\frac{\Gamma'}{|y|}\Big)\frac{\operatorname{Re}(y^3)}{|y|^2}=0.\]
This leads us to the final approximation
\begin{equation}\label{psi3}
\psi_\mu(\ve\mu y)=\left(1+c_1\ve\mu y_1+c_2\ve^2\mu^2|y|^2\right)\widehat\Gamma_{\ve\mu}
+\frac{R^3}{2h(h^2+R^2)^{\frac32}}H_1(\ve\mu y),
\end{equation}
where the preceding estimates also imply
\[|D_z^kH_1(z)|\le C\frac{|z|^{3-k}}{(\ve\mu+|z|)^2},\qquad 0<|z|<\delta,\quad k=0,1,2.\]
Using \eqref{oper}, we obtain
\[|B_0[H_1](z)|\le C\left(|z||D_z^2H_1(z)|+|D_zH_1(z)|\right)\le C,\]
so that \eqref{modelsemil} yields
\[
\ve^2\mu^2L_x(\psi_\mu)=\Delta_y\Gamma+\frac{3Rh^2+R^3}{2h(h^2+R^2)^{\frac32}}\ve\mu y_1\Gamma_+^\gamma+\ve^2\mu^2E_*,\]
with
\[E_*=E_1+E_2+E_3+\frac{R^3}{2h(h^2+R^2)^{\frac 32}}B_0[H_1].\]
To this end, we use \eqref{var} to express the approximate stream function \eqref{psi3} around $P=(a,b)$ in the original coordinates $x=(x_1,x_2)$ as
\[\Psi_{\mu,P}(x)=\psi_\mu\left(A[P]^{-1}(x-P)\right).\]
\end{proof}
\section{Construction of a global multi-vortex approximation}\label{section3}

Given an integer $N\geq1$, in this section we construct an approximate stream function for an $N$-vortex solution of \eqref{eqtosolve}. The approximation is designed so that its associated vorticity resembles a superposition of Dirac masses concentrated at $N$ distinct points in the plane, whose pairwise distance shrinks as $\ve\to0$. As will be shown below, it is obtained by combining suitably rescaled copies of the building block introduced in \eqref{approxinx}, centered at each vortex point.

More specifically, we let $P_j=(a_j,b_j)$, $j=1,\dots,N$, be as in \eqref{formpoints}, consider positive scaling parameters $\mu_j>0$ to be determined, and employ \eqref{approxinx} to define
\begin{equation}\label{profj}
\Psi_j(x)=\Psi_{\mu_j,P_j}(x).
\end{equation}
Hereafter, we assume that $\max\limits_{1\leq j\leq N}\ve\mu_j\to0$ as $\ve\to0$.
Using Proposition \ref{propapprox}, for any $\delta>0$ small and the variable $z=A[P_j]^{-1}(x-P_j)=\ve\mu_j y$ in \eqref{var}, for each $j\in\{1,\dots,N\}$ we infer that
\begin{equation}\label{erroraroundeachj}
\ve^2\mu_j^2L_x(\Psi_j)=\Delta_y\Gamma+\frac{3R_jh^2+R_j^3}{2h(h^2+R_j^2)^{\frac32}}\ve\mu_j y_1\Gamma_+^\gamma+\ve^2\mu_j^2E_{*,j},\qquad |y|<\frac{\delta}{\ve\mu_j},
\end{equation}
where $R_j=\sqrt{a_j^2+b_j^2}$ and $E_{*,j}$ is of class $C^2$ in $z=\ve\mu_j y$ for each fixed $\ve\mu_j>0$, and is uniformly bounded as $\ve\mu_j\to0$. 

We then multiply each $\Psi_j$ by a smooth cutoff function $\eta_0:\R^2\to[0,1]$ to obtain the profile
\[\eta_0(x)\sum_{j=1}^N\kappa_j\Psi_j(x),\]
where
\begin{equation}\label{etazero}
\eta_0(x)=\eta\left(|x-(r_0,0)|\right),
\end{equation}
with $\eta:\R\to[0,1]$ denoting a smooth cutoff satisfying
\begin{equation}\label{defeta}
\eta(s)=1\quad\mbox{for }s\leq\frac12,\qquad \eta(s)=0\quad\mbox{for }s\geq1.
\end{equation}
In this way, the approximate stream function is localised in a fixed bounded neighbourhood of $(r_0,0)$, while setting $z_j=([z_j]_1,[z_j]_2)=A[P_j]^{-1}(x-P_j)$ and using \eqref{erroraroundeachj}, a direct computation yields
\[
L_x\Big(\eta_0\sum_{j=1}^N\kappa_j\Psi_j\Big)
=\eta_0\sum_{j=1}^N\kappa_j\bigg(\Delta_{z_j}\Gamma_{\ve\mu_j}+\frac{3R_jh^2+R_j^3}{2h(h^2+R_j^2)^{\frac32}}\frac{[z_j]_1}{\ve^2\mu_j^2}(\Gamma_{\ve\mu_j})_+^\gamma\bigg)+g(x),
\]
where
\begin{equation}\label{Definitionofg}
g(x)=\eta_0\sum_{j=1}^N\kappa_jE_{*,j}+\sum_{j=1}^N\kappa_j\left[L_x(\eta_0\Psi_j)-\eta_0L_x(\Psi_j)\right].
\end{equation}
One can verify that $g\in C^2(\R^2)$ with compact support, satisfying $\|g\|_{L^\infty(\R^2)}\leq C$ for some $C>0$ independent of $\ve$ and the admissible scales.

To proceed, we refine the preliminary ansatz $\eta_0\sum\limits_{j=1}^N\kappa_j\Psi_j$ by adding a global term that eliminates the bounded error term $g(x)$. In particular, we let $H_{2\ve}(x)$ be a solution of
\[L_x(H_{2\ve})+g(x)=0\quad\mbox{in }\R^2,\qquad H_{2\ve}((r_0,0))=0,\]
whose existence follows from the linear theory of Proposition \ref{propouter}, and which satisfies the growth estimate
\begin{equation}\label{boundh2e}
|H_{2\ve}(x)|\leq C(1+|x|^2)\quad\forall x\in\R^2.
\end{equation}
More precisely, we fix the linear solution operator $\mathcal T^o$ of Proposition \ref{propouter} and take
\begin{equation}\label{preciseH2epsilon}
H_{2\ve}(x)=\mathcal T^o[g](x)-\mathcal T^o[g]((r_0,0)),
\end{equation}
specifying the choice of $H_{2\ve}(x)$ as the points and scaling parameters vary. The uniform bound for $g(x)$ together with its fixed compact support yields \eqref{boundh2e}.
In terms of the variable $z_j=A[P_j]^{-1}(x-P_j)$, $z_j=([z_j]_1,[z_j]_2)$, we define the final approximate stream function as
\begin{equation}\label{globalapprox}
\Psi_0(x)=\eta_0(x)\sum_{j=1}^N\kappa_j\bigg[\widehat\Gamma_{\ve\mu_j}(z_j)\left(1+c_{1,j}[z_j]_1+c_{2,j}|z_j|^2\right)+\frac{R_j^3}{2h(h^2+R_j^2)^{\frac32}}H_1(z_j)\bigg]+H_{2\ve}(x),
\end{equation}
where $\widehat\Gamma_{\ve\mu_j}$ is given by \eqref{defofgammahat}, while the function $H_1$ and the constants $c_{1,j}$, $c_{2,j}$, $j=1,\dots,N$, are as in Proposition \ref{propapprox}, with $\mu$ and $P=(a,b)$ replaced by $\mu_j$ and $P_j=(a_j,b_j)$, respectively. 

Collecting the above, we conclude that the function $\Psi_0$ in \eqref{globalapprox} satisfies
\begin{equation}\label{lxofglobal}
L_x(\Psi_0)=\eta_0\sum_{j=1}^N\frac{\kappa_j}{\ve^2\mu_j^2}\Big(-1+\frac{3R_jh^2+R_j^3}{2h(h^2+R_j^2)^{\frac32}}[z_j]_1\Big)(\Gamma_{\ve\mu_j})_+^\gamma,\qquad z_j=A[P_j]^{-1}(x-P_j),
\end{equation}
which further yields the property
\begin{equation}\label{Lxsupport}
\operatorname{supp}L_x(\Psi_0)\subset\bigcup_{j=1}^N\big\{x\in\R^2:|A[P_j]^{-1}(x-P_j)|\leq\ve\mu_j\big\}.
\end{equation}
We note that the profile $\Psi_0$ in \eqref{globalapprox} is of class $C^2$ and globally defined in $\R^2$, with its dependence on the points $P_1,\dots,P_N$ and the positive scaling parameters $\mu_1,\dots,\mu_N$ determined by \eqref{profj} and the fixed choice of the outer solution operator in \eqref{preciseH2epsilon}.
\section{Properties of the scaling parameters}\label{Section 4}
The aim of this section is to specify the relations determining the scaling parameters $\mu_j,\,j=1,\dots,N$ introduced in \eqref{profj}, and to describe how they depend on the vortex points $P_j$ in \eqref{formpoints}. We use the decomposition
\begin{equation}\label{scalingansatz}
\mu_j=\mu_j^0+\mu_j^*,\qquad j=1,\dots,N,
\end{equation}
where $\mu_j^0$ designates the leading order scaling parameter around $P_j$ and $\mu_j^*$ is a small correction.
For each $j\in\{1,\dots,N\}$, the parameter $\mu_j^0$ will be chosen so that
\[\log(\mu_j^0)=\mathcal O\left(\log|\log\ve|\right)\quad\text{as }\ve\to0,\]
while for the perturbation $\mu_j^*$ we assume a priori that it satisfies
\begin{equation}\label{aprioriscalingperp}
\log\bigg(1+\frac{\mu_j^*}{\mu_j^0}\bigg)=\mathcal O\left(\ve^{1+\tilde\sigma}\right),\qquad\mbox{for a small }\tilde\sigma>0.
\end{equation}
We remark that the choice of the leading-order scaling parameters $\mu_j^{0}$ will depend on all the full scaling parameters in \eqref{scalingansatz}, thus the matching relations will be implicit. However, their solvability and the choice of $\mu_j^*$ will be established jointly in Proposition \ref{propsec8}, where \eqref{aprioriscalingperp} will be verified for a suitable $\tilde\sigma>0$. The role of these scaling corrections in the construction is discussed in Remark \ref{decoupling}.

Currently, our objective is to explain in more detail the role of the parameters $\mu_j^0$. They are chosen to cancel the constant terms contributed by the other local stream functions and the global correction $H_{2\ve}$ in the expansion of $\Psi_0-\frac{\alpha}{2}|\log\ve|\,|x|^2$ near each vortex point $P_j$, which are collectively of size $\mathcal O(\log|\log\ve|)$.

More concretely, for $d>0$ as in \eqref{unifdistance}, we fix $\delta>0$ sufficiently small so that
\begin{equation}\label{cond1delta}
\delta<\frac{d}{4},
\end{equation}
and consider the region
\begin{equation}\label{firstdelta}
|A_i^{-1}(x-P_i)|<\frac{\delta}{|\log\ve|},
\end{equation}
where $\eta_0(x)=1$ for all sufficiently small $\ve>0$ due to its definition in \eqref{etazero}. 

We have the following Proposition.
\begin{prop}\label{propscaling}
For each $j=1,\dots,N$, let $A_j=A[P_j]$ as in \eqref{var}, evaluated at the point $P_j$ in \eqref{formpoints}. In addition, for every $i\in\{1,\dots,N\}$, assume that the decomposition $\mu_i=\mu_i^0+\mu_i^*$ satisfies
\begin{equation}\label{scalingmain}
\begin{aligned}
-\kappa_i\nu'(1)\log(\mu_i^0)
&=\sum_{j\ne i}\kappa_j\nu'(1)\log\left(|A_j^{-1}(P_i-P_j)|\right)\bigg[1+c_{1,j}[A_j^{-1}(P_i-P_j)]_1+c_{2,j}|A_j^{-1}(P_i-P_j)|^2\bigg]\\
&\quad+\sum_{j\ne i}\kappa_j\frac{R_j^3}{2h(h^2+R_j^2)^{\frac32}}H_1\big(A_j^{-1}(P_i-P_j)\big)+H_{2\ve}(P_i),
\end{aligned}
\end{equation}
and
\begin{equation}\label{scalingpert}
-\kappa_i\nu'(1)\log\left(1+\frac{\mu_i^*}{\mu_i^0}\right)=\mathcal O\left(\ve^{1+\tilde\sigma}\right)\quad\mbox{for some }\tilde\sigma>0,
\end{equation}
with
\begin{equation}\label{constantsj}
c_{1,j}=\frac{R_jh}{2(h^2+R_j^2)^{\frac32}},\qquad c_{2,j}=\frac{3h^2R_j^2+R_j^4}{8(h^2+R_j^2)^3},\qquad R_j=\sqrt{a_j^2+b_j^2},\qquad j=1,\dots,N.
\end{equation}
Then, for fixed $i\in\{1,\dots,N\}$, $\delta>0$ as in \eqref{cond1delta} and the variable $z=A_i^{-1}(x-P_i)=\ve\mu_i y$, in the region $|z|<\frac{\delta}{|\log\ve|}$ the approximate stream function $\Psi_0$ in \eqref{globalapprox} has the expansion
\begin{equation}\label{exppsiminusfinal}
\begin{aligned}
\frac1{\kappa_i}\left(\Psi_0-\frac{\alpha}{2}|\log\ve|\,|x|^2\right)
&=\left(1+c_{1,i}\ve\mu_i y_1+c_{2,i}\ve^2\mu_i^2|y|^2\right)\Gamma(y)-\frac{\alpha}{2\kappa_i}|\log\ve|\,|P_i|^2-\nu'(1)|\log\ve|\\
&\quad+\nu'(1)\log\left(1+\frac{\mu_i^*}{\mu_i^0}\right)+\frac{R_i^3}{2h(h^2+R_i^2)^{\frac32}}H_1(\ve\mu_i y)\\
&\quad+\ve\mu_i y_1\mathcal A_{1,i}+\ve\mu_i y_2\mathcal A_{2,i}+\mathcal O\left(\ve^2\mu_i^2|\log\ve|^2|y|^2\right)\quad\mbox{as }\ve\to0,
\end{aligned}
\end{equation}
where the functions $\mathcal A_{1,i}$ and $\mathcal A_{2,i}$ are defined by
\begin{equation}\label{mode1func}
\begin{aligned}
\mathcal A_{1,i}
&=|\log\ve|\bigg(-c_{1,i}\nu'(1)-\frac{\alpha R_i h}{\kappa_i\sqrt{h^2+R_i^2}}\bigg)+\nu'(1)c_{1,i}\log(\mu_i)\\
&\quad+\sum_{j\ne i}\frac{\kappa_j}{\kappa_i}\nu'(1)\frac{\left[(A_j^{-1}A_i)^T A_j^{-1}(P_i-P_j)\right]_1}{|A_j^{-1}(P_i-P_j)|^2}\left(1+c_{1,j}[A_j^{-1}(P_i-P_j)]_1+c_{2,j}|A_j^{-1}(P_i-P_j)|^2\right)\\
&\quad+\sum_{j\ne i}\frac{\kappa_j}{\kappa_i}\nu'(1)\log\left(|A_j^{-1}(P_i-P_j)|\right)\left(c_{1,j}[A_j^{-1}A_i]_{11}+2c_{2,j}\left[(A_j^{-1}A_i)^T A_j^{-1}(P_i-P_j)\right]_1\right)\\
&\quad+\sum_{j\ne i}\frac{\kappa_j}{\kappa_i}\frac{R_j^3}{2h(h^2+R_j^2)^{\frac32}}\left[(A_j^{-1}A_i)^T\nabla_z H_1\left(A_j^{-1}(P_i-P_j)\right)\right]_1+\frac1{\kappa_i}[A_i^T\nabla_x H_{2\ve}(P_i)]_1,\\
\mathcal A_{2,i}
&=\sum_{j\ne i}\frac{\kappa_j}{\kappa_i}\nu'(1)\frac{\left[(A_j^{-1}A_i)^T A_j^{-1}(P_i-P_j)\right]_2}{|A_j^{-1}(P_i-P_j)|^2}\left(1+c_{1,j}[A_j^{-1}(P_i-P_j)]_1+c_{2,j}|A_j^{-1}(P_i-P_j)|^2\right)\\
&\quad+\sum_{j\ne i}\frac{\kappa_j}{\kappa_i}\nu'(1)\log\left(|A_j^{-1}(P_i-P_j)|\right)\left(c_{1,j}[A_j^{-1}A_i]_{12}+2c_{2,j}\left[(A_j^{-1}A_i)^T A_j^{-1}(P_i-P_j)\right]_2\right)\\
&\quad+\sum_{j\ne i}\frac{\kappa_j}{\kappa_i}\frac{R_j^3}{2h(h^2+R_j^2)^{\frac32}}\left[(A_j^{-1}A_i)^T\nabla_z H_1\left(A_j^{-1}(P_i-P_j)\right)\right]_2+\frac1{\kappa_i}[A_i^T\nabla_x H_{2\ve}(P_i)]_2.
\end{aligned}
\end{equation}
The function $H_{2\ve}$ depends on all the full scaling parameters through its construction in \eqref{preciseH2epsilon} and the remainder in \eqref{exppsiminusfinal} is uniform in $|z|<\frac{\delta}{|\log\ve|}$. Finally, defining
\begin{equation}\label{logmode1fun}
\widehat{\mathcal A}_{j,i}=\frac{\mathcal A_{j,i}}{|\log\ve|},\qquad j=1,2,\quad i=1,\dots,N,
\end{equation}
we have $|\widehat{\mathcal A}_{1,i}|+|\widehat{\mathcal A}_{2,i}|\le C$ for some $C>0$ independent of $\ve$ and the admissible parameters.
\end{prop}

\begin{proof}
To begin with, we fix $i\in\{1,\dots,N\}$ and let $A_i=A[P_i]$ as in \eqref{var}. Since the points in \eqref{formpoints} are of the form
\[P_i=(r_0+\tilde r,0)+\frac1{|\log\ve|}\widehat P_i,\qquad P_i=(a_i,b_i),\qquad\widehat P_i=(\hat a_i,\hat b_i),\]
we obtain the matrix expansions
\begin{equation}\label{expinverse}
\begin{aligned}
A_i^{-1}
&=\begin{pmatrix}\dfrac{a_i\sqrt{h^2+R_i^2}}{R_i h}&\dfrac{b_i\sqrt{h^2+R_i^2}}{R_i h}\\[1mm]-\dfrac{b_i}{R_i}&\dfrac{a_i}{R_i}\end{pmatrix}
=\begin{pmatrix}\dfrac{\sqrt{h^2+r_0^2}}h&0\\0&1\end{pmatrix}+\frac{\log|\log\ve|}{|\log\ve|}\widehat A_i,\\
A_j^{-1}A_i&=\begin{pmatrix}1&0\\0&1\end{pmatrix}+\frac1{|\log\ve|}\widehat I_{ij},\qquad\mbox{for all }i\ne j,
\end{aligned}
\end{equation}
where $\widehat A_i$ and $\widehat I_{ij}$ are $2\times2$ matrices with smooth entries depending on $(\tilde r,\hat a_i,\hat b_i)$ and $(\tilde r,\hat a_i,\hat b_i,\hat a_j,\hat b_j)$, respectively, which are uniformly bounded as $\ve\to0$.

Due to \eqref{formpoints} and \eqref{unifdistance}, we have that $|P_i-P_j|=\mathcal{O}(|\log\ve|^{-1})$ for all $j\ne i$, thus Proposition \ref{propapprox} gives $H_1(A_j^{-1}(P_i-P_j))=\mathcal O(|\log\ve|^{-1})$, while \eqref{boundh2e} provides a uniform bound on $H_{2\ve}(P_i)$. Consequently, \eqref{scalingmain} and \eqref{scalingpert} imply that for all $j=1,\dots,N$ it holds
\begin{equation}\label{scalingboundsinprop}
|\log\mu_j^0|+|\log\mu_j|\le C\log|\log\ve|\quad \mbox{as } \ve\to0.
\end{equation}
Furthermore, Proposition \ref{propapprox} and \eqref{Definitionofg} give $\|g\|_{L^\infty(\R^2)}\le C$ and $\|\nabla_xg\|_{L^\infty(\R^2)}\le C(\ve\min\limits_{1\le j \le N}\mu_j)^{-1}$, thus using \eqref{boundh2e} one can further show that \[
\|D_x^2H_{2\ve}\|_{L^\infty(B_1((r_0,0)))}\le C\bigg(1+\int_{0}^{1}\frac{1}{r}\min\big\{1,r(\ve\min_{1\leq j \leq N} \mu_j)^{-1}\big\}\,\mathrm dr\bigg)\le C|\log\ve|.
\]
With these in mind, in the region $|y|<\frac{\delta}{\ve\mu_i|\log\ve|}$ we get the expansions
\begin{equation}\label{expanH}
\begin{aligned}
&H_{2\ve}(x)=H_{2\ve}(P_i)+\ve\mu_i(A_i y)\cdot\nabla H_{2\ve}(P_i)+\mathcal O\left(\ve^2\mu_i^2{|\log\ve|}|y|^2\right),\\
&H_1\big(A_j^{-1}A_i\ve\mu_i y+A_j^{-1}(P_i-P_j)\big)=H_1\big(A_j^{-1}(P_i-P_j)\big)+\ve\mu_i\nabla H_1\big(A_j^{-1}(P_i-P_j)\big)\cdot A_j^{-1}A_i y\\
&\hspace{55mm}\quad+\mathcal O\left(\ve^2\mu_i^2|\log\ve|\,|y|^2\right), \quad j\neq i.
\end{aligned}
\end{equation}
Moreover, using \eqref{defofgammahat} we obtain
\begin{equation}\label{expansionGamma}
\begin{aligned}
&\widehat\Gamma_{\ve\mu_j}\left(A_j^{-1}A_i z+A_j^{-1}(P_i-P_j)\right)
\\
&=\nu'(1)\log\left(|A_j^{-1}(P_i-P_j)|\right)+\nu'(1)\frac{A_j^{-1}(P_i-P_j)}{|A_j^{-1}(P_i-P_j)|^2}\cdot\ve\mu_i A_j^{-1}A_i y
+\mathcal O\bigg(\frac{|\ve\mu_i A_j^{-1}A_i y|^2}{|A_j^{-1}(P_i-P_j)|^2}\bigg).
\end{aligned}
\end{equation}
Using the definition of $\Psi_0$ in \eqref{globalapprox}, we write
\begin{equation}\label{psiminusquad}
\begin{aligned}
&\Psi_0-\frac{\alpha}{2}|\log\ve|\,|x|^2\\
&=\kappa_i(1+c_{1,i}\ve\mu_i y_1+c_{2,i}\ve^2\mu_i^2|y|^2)\Gamma(y)\\
&\quad+\kappa_i\left(-\nu'(1)|\log\ve|+\nu'(1)\log(\mu_i^0)+\nu'(1)\log\left(1+\frac{\mu_i^*}{\mu_i^0}\right)\right)(1+c_{1,i}\ve\mu_i y_1+c_{2,i}\ve^2\mu_i^2|y|^2)\\
&\quad+\kappa_i\frac{R_i^3}{2h(h^2+R_i^2)^{\frac32}}H_1(\ve\mu_i y)+H_{2\ve}(P_i+A_i\ve\mu_i y)\\
&\quad-\frac{\alpha}{2}|\log\ve|\,|P_i|^2-\frac{\alpha R_i h|\log\ve|}{\sqrt{h^2+R_i^2}}\ve\mu_i y_1-\frac{\alpha}{2}|\log\ve|\ve^2\mu_i^2|A_i y|^2\\
&\quad+\sum_{j\ne i}\kappa_j\Bigg(\nu'(1)\log\bigg(\frac{|A_j^{-1}A_i\ve\mu_i y+A_j^{-1}(P_i-P_j)|}{\ve\mu_j}\bigg)-\nu'(1)|\log(\ve\mu_j)|\Bigg)\\
&\qquad\times\left(1+c_{1,j}[A_j^{-1}A_i\ve\mu_i y+A_j^{-1}(P_i-P_j)]_1+c_{2,j}|A_j^{-1}A_i\ve\mu_i y+A_j^{-1}(P_i-P_j)|^2\right)\\
&\quad+\sum_{j\ne i}\kappa_j\frac{R_j^3}{2h(h^2+R_j^2)^{\frac32}}H_1\big(A_j^{-1}A_i\ve\mu_i y+A_j^{-1}(P_i-P_j)\big).
\end{aligned}
\end{equation}
Substituting \eqref{expanH} and \eqref{expansionGamma} into \eqref{psiminusquad}, we arrive at the expansion in \eqref{exppsiminusfinal}, where constant terms are removed by the choice of $\mu_i^0$ in \eqref{scalingmain} and the contribution of $\mu_i^*$ is controlled by the a priori bound \eqref{scalingpert}.
This concludes the proof.
\end{proof}

\begin{remark}
Combining \eqref{scalingboundsinprop} and \eqref{scalingpert}, the vortex core scales satisfy
\[\ve\mu_i\le C\ve|\log\ve|^M\le C\ve^{1-\zeta},\qquad i=1,\dots,N,\]
for some fixed $M>0$ and every fixed $\zeta\in(0,1)$, uniformly in the admissible parameters. Consequently, the vortex core scales are asymptotically much smaller than the pairwise distances between the vortex centers, which are comparable to $|\log\ve|^{-1}$ by \eqref{formpoints} and \eqref{unifdistance}.
\end{remark}

\section{Choice of the vorticity profile and approximation error estimates}\label{Section 5}
In the present section, the objective is to select a suitable nonlinearity $F$ in \eqref{eqtosolve} so that the corresponding vorticity profile
\[W(x)=F\left(\Psi-\frac{\alpha}{2}|\log\ve|\,|x|^2\right)\]
satisfies \eqref{propvorticity}, and to quantify the approximation error arising from the substitution of $\Psi_0(x)$ from \eqref{globalapprox} into \eqref{eqtosolve}.

In this regard, we introduce the error operator
\begin{equation}\label{definitionerroroper}
S(\Psi)(x)=L_x(\Psi)+F\left(\Psi-\frac{\alpha}{2}|\log\ve|\,|x|^2\right),\qquad x\in\R^2,
\end{equation}
so that $\Psi$ solves \eqref{eqtosolve} if and only if $S(\Psi)=0$.

For $i\in\{1,\dots,N\}$, we take $\delta>0$ as in \eqref{cond1delta} and consider the inner region \eqref{firstdelta} centered at $P_i$. In the variable $y=\frac{A_i^{-1}(x-P_i)}{\ve\mu_i}$ and the annulus
\begin{equation}\label{annuluscutoff}
\frac{\delta}{2\ve \mu_i|\log\ve|}\le |y|\le\frac{\delta}{\ve \mu_i|\log\ve|},
\end{equation}
we have that $\log|y|=|\log\ve|+\mathcal O(\log|\log\ve|)$, while Proposition \ref{propscaling} yields
\begin{equation}\label{exppreliminary}
\frac{1}{\kappa_i}\left(\Psi_0-\frac{\alpha}{2}|\log\ve|\,|x|^2\right)+\nu'(1)|\log\ve|+\frac{\alpha}{2\kappa_i}|\log\ve||P_i|^2=\nu'(1)|\log\ve|+ \mathcal O(\log|\log\ve|),
\end{equation}
uniformly in \eqref{annuluscutoff} and the admissible parameters.

To proceed, using the cutoff $\eta$ in \eqref{defeta}, for $i=1,\dots,N$ we define the spatial cutoffs
\begin{equation}\label{cutoffchi}
\chi_i(x)=\eta\left(\frac{|\log\ve|\,|A_i^{-1}(x-P_i)|}{\delta}\right),\qquad \mbox{so that }\,\,\chi_i(x)=
\begin{cases}
1,&\displaystyle |A_i^{-1}(x-P_i)|\le\frac{\delta}{2|\log\ve|}\\[1mm]
0,&\displaystyle |A_i^{-1}(x-P_i)|\ge\frac{\delta}{|\log\ve|}
\end{cases}
.
\end{equation}
Inspired by the structure of \eqref{rescaledequation123}, for $\gamma>3$ and $s_+=\max(s,0)$ we then introduce the nonlinearity
\begin{equation}\label{nonlinearity}
F(x,s)=\sum_{i=1}^N\frac{\kappa_i}{\ve^2\mu_i^2}\chi_i(x)F_i(s),\qquad
F_i(s)=\left(\frac{s}{\kappa_i}+\nu'(1)|\log\ve|+\frac{\alpha}{2\kappa_i}|\log\ve|\,|P_i|^2\right)_+^\gamma,
\end{equation}
noting that a similar spatially localised positive-part nonlinearity was also employed in \cite{MR3162476}.
\begin{remark}
The nonlinearity $F$ in \eqref{nonlinearity} admits an additional dependence on $x$ through the cutoffs $\chi_i(x)$, thus in what follows we interpret the expression in \eqref{definitionerroroper} as
\[W(x)=F\left(x,\Psi-\frac{\alpha}{2}|\log\ve|\,|x|^2\right),\qquad 
S(\Psi)(x)=L_x(\Psi)+W(x),\]
with the same convention applying to \eqref{eqtosolve}. We verify that this spatial dependence is compatible with the transport equation
\begin{equation}\label{Transportverify}
\nabla^{\perp}_x\left(\Psi-\frac{\alpha}{2}|\log\ve||x|^2\right)\cdot\nabla_x W(x)=0, \quad \forall x\in\R^2
\end{equation}
as follows.

For any stream function $\Psi\in C^2(\R^2)$ satisfying
\[\sup_{x\in\bigcup_{i=1}^N\operatorname{supp}\chi_i}|\Psi(x)-\Psi_0(x)|\le C
\]
for some $C>0$ independent of $\ve$ (cf. \eqref{gluingansatz} and Proposition \ref{propsec8}), \eqref{cutoffchi} shows that if  $x\in\operatorname{supp}\nabla_x\chi_i$ then $y=\frac{A_i^{-1}(x-P_i)}{\ve\mu_i}$ satisfies \eqref{annuluscutoff}. Since the amplitude parameters $\kappa_i$ are fixed and nonzero, the above bound gives $\frac{\Psi-\Psi_0}{\kappa_i}=\mathcal O(1)$, thus the expansion in \eqref{exppreliminary} remains valid with $\Psi_0$ replaced by $\Psi$, and the argument defining $F_i$ in \eqref{nonlinearity} can be estimated as $\nu'(1)|\log\ve|+\mathcal O(\log|\log\ve|)<0$
throughout \eqref{annuluscutoff}, for all sufficiently small $\ve$. Consequently, one finds 
\[\nabla_{x}^{\perp}\Big(\Psi-\frac{\alpha}{2}|\log\ve||x|^2\Big)\cdot F_i\left(\Psi-\frac{\alpha}{2}|\log\ve|\,|x|^2\right)\nabla_x\chi_i=0
\quad\mbox{in }\,\,\R^2,\quad i=1,\dots,N,\]
while a direct application of the chain rule proves \eqref{Transportverify}.
\end{remark}
The following Proposition holds.
\begin{prop}\label{errorprop1}
Consider $r_0>0$, $h>0$, let $\alpha$ be the fixed rotational speed in \eqref{uniformspeed}, and assume that the points $P_i$ and the scaling parameters $\mu_i$ satisfy the hypotheses of Proposition \ref{propscaling}. For the approximate stream function $\Psi_0$ in \eqref{globalapprox}, the nonlinear function $F(x,s)$ in \eqref{nonlinearity}, and $\delta>0$ sufficiently small satisfying \eqref{cond1delta}, there exists a constant $C>0$ such that for all $\ve>0$ small and each $i\in\{1,\dots,N\}$, the error function in \eqref{definitionerroroper} satisfies
\[\ve^2\mu_i^2|S(\Psi_0)(x)|\le C\ve\mu_i|\log\ve|\left(\Gamma\left(\frac y2\right)\right)_+^\gamma,\qquad
|A_i^{-1}(x-P_i)|<\frac{\delta}{|\log\ve|},\]
where $z=A_i^{-1}(x-P_i)$ and $y=\frac z{\ve\mu_i}$. Moreover, it holds
\[
\operatorname{supp}S(\Psi_0)\subset\bigcup_{i=1}^N\Big\{x\in\R^2:\ |A_i^{-1}(x-P_i)|\le\frac32\ve\mu_i\Big\}.\]
\end{prop}
\begin{proof}
For fixed $i\in\{1,\dots,N\}$, in order to estimate the approximation error
\begin{equation*}
\ve^2\mu_i^2S(\Psi_0)=\ve^2\mu_i^2L_x(\Psi_0)+\ve^2\mu_i^2\sum_{j=1}^N\frac{\kappa_j}{\ve^2\mu_j^2}\chi_j(x)F_j\left(\Psi_0-\frac{\alpha}{2}|\log\ve|\,|x|^2\right),
\end{equation*}
with $L_x(\Psi_0)$ given in \eqref{lxofglobal}, we consider three distinct regions in the variable
\[
y=\frac{A_i^{-1}(x-P_i)}{\ve\mu_i}.
\]
\medskip
\noindent\textbf{Case 1:} $|y|\le\frac32$.
\\
In this region, due to \eqref{cutoffchi} and \eqref{etazero} we have $\chi_i=1$, $\chi_j=0$ for all $j\ne i$ and $\eta_0=1$. Using \eqref{exppsiminusfinal}-\eqref{logmode1fun} from Proposition \ref{propscaling}, we find
\begin{equation*}
\left|\frac1{\kappa_i}\left(\Psi_0-\frac{\alpha}{2}|\log\ve|\,|x|^2\right)+\nu'(1)|\log\ve|+\frac{\alpha}{2\kappa_i}|\log\ve|\,|P_i|^2-\Gamma(y)\right|
\le C\ve\mu_i|\log\ve|,
\end{equation*}
hence combining \eqref{lxofglobal} with \eqref{nonlinearity} and using \eqref{rescaledequation123}, we obtain
\[
\begin{aligned}
&\ve^2\mu_i^2S(\Psi_0)\\
&=\kappa_i\frac{3R_ih^2+R_i^3}{2h(h^2+R_i^2)^{\frac32}}\ve\mu_i y_1\Gamma_+^\gamma(y)+\kappa_i\gamma\Gamma_+^{\gamma-1}(y)\bigg[
\nu'(1)\log\left(1+\frac{\mu_i^*}{\mu_i^0}\right)
+\left(c_{1,i}\ve\mu_i y_1+c_{2,i}\ve^2\mu_i^2|y|^2\right)\Gamma(y)\\
&\quad+\frac{R_i^3}{2h(h^2+R_i^2)^{\frac32}}H_1(\ve\mu_i y)
+\ve\mu_i y_1\mathcal A_{1,i}+\ve\mu_i y_2\mathcal A_{2,i}\bigg]
+\mathcal O\left(\ve^2\mu_i^2|\log\ve|^2\right),
\end{aligned}
\]
where we applied Taylor's formula to $s\mapsto s_+^\gamma$. Since $\inf\limits_{|y|\le3/2}\Gamma(\frac{y}{2})>0$, it then follows that
\[
\ve^2\mu_i^2|S(\Psi_0)|\le C\ve\mu_i|\log\ve|\left(\Gamma\left(\frac y2\right)\right)_+^\gamma.
\]
\medskip
\noindent\textbf{Case 2:} $\frac32<|y|<\frac{\delta}{\ve\mu_i|\log\ve|}$.
\\
In this region, $L_x(\Psi_0)=0$ by \eqref{Lxsupport}, while $0\leq\chi_i\leq1$ and $\chi_j=0$ for every $j\ne i$. We first observe that 
\[
1+c_{1,i}\ve\mu_i y_1+c_{2,i}\ve^2\mu_i^2|y|^2\ge\frac12,\qquad
|\ve\mu_i y_1\mathcal A_{1,i}+\ve\mu_i y_2\mathcal A_{2,i}|\le C\ve\mu_i|\log\ve|\,|y|,
\]
where we used \eqref{mode1func} and \eqref{logmode1fun}. 

Since $\Gamma(y)=\nu'(1)\log|y|$ and $\nu'(1)<0$, for $\delta$ sufficiently small (independently of $\ve$) we employ \eqref{exppsiminusfinal} and \eqref{scalingpert} to find
\begin{equation}\label{secondcaseprop}
\begin{aligned}
&\frac{1}{\kappa_i}\left(\Psi_0-\frac{\alpha}{2}|\log\ve|\,|x|^2\right)+\nu'(1)|\log\ve|+\frac{\alpha}{2\kappa_i}|\log\ve|\,|P_i|^2\\
&\qquad\le\frac{\nu'(1)}2\log|y|+C\ve\mu_i|\log\ve|\,|y|+C\ve^2\mu_i^2|\log\ve|^2|y|^2+C\ve^{1+\tilde\sigma}\\
&\qquad\le\frac{\nu'(1)}2\log|y|+C(\delta+\delta^2)+C\ve^{1+\tilde\sigma}
\le\frac{\nu'(1)}4\log\frac32<0,
\end{aligned}
\end{equation}
for all small $\ve>0$. As a result, we readily get
\[
F_i\left(\Psi_0-\frac{\alpha}{2}|\log\ve|\,|x|^2\right)=0,\qquad S(\Psi_0)=0.
\]
\medskip
\noindent\textbf{Case 3:} $|A_i^{-1}(x-P_i)|\ge\frac{\delta}{|\log\ve|}$ for every $i=1,\dots,N$.
\\
Since $\chi_i=0$ for all $i$ and $L_x(\Psi_0)=0$, we directly obtain $S(\Psi_0)=0$.
\end{proof}
\section{Inner-Outer Gluing Procedure}\label{Section 6}
In this section, we consider the approximate stream function $\Psi_0$ in \eqref{globalapprox}, satisfying the error estimates of Proposition \ref{errorprop1}, and seek an exact solution $\Psi$ of the equation
\begin{equation}\label{truesol}
S(\Psi)(x)=L_x(\Psi)+F\left(x,\Psi-\frac{\alpha}{2}|\log\ve|\,|x|^2\right)=0\quad\text{in }\R^2,
\end{equation}
where the nonlinearity $F$ is defined in \eqref{nonlinearity}.

We solve \eqref{truesol} by means of an \emph{Inner-Outer gluing scheme}, through which we aim to find a true solution $\Psi$ as a small perturbation of the approximation $\Psi_0$. We recall that $\Psi_0$ depends on the concentration points $P_1,\dots,P_N$ satisfying \eqref{formpoints}, and on the full scaling parameters $\mu_j=\mu_j^0+\mu_j^*$, $j=1,\dots,N$, as in Proposition \ref{propscaling}.

More specifically, we seek a solution $\Psi(x)$ to \eqref{truesol} of the form
\begin{equation}\label{gluingansatz}
\Psi(x)=\Psi_0(x)+\varphi(x),\qquad \varphi(x)=\sum_{i=1}^N\tilde{\eta}_i(x)\phi_{in,i}(y_i)+\phi_{out}(x),\qquad y_i=\frac{A_i^{-1}(x-P_i)}{\ve\mu_i},
\end{equation}
where for $\delta$ sufficiently small as in Proposition \ref{errorprop1} and $0<2\delta_1<\delta$, we use the cutoff $\eta$ from \eqref{defeta} to define
\begin{equation}\label{cutoffgluing}
\tilde{\eta}_i(x)=\eta\left(\frac{|\log\ve|^2|A_i^{-1}(x-P_i)|}{\delta_1}\right).
\end{equation}
In view of \eqref{truesol} and the decomposition \eqref{gluingansatz}, we write
\[\begin{aligned}
S(\Psi_0+\varphi)
&=\sum_{i=1}^N\tilde{\eta}_i\bigg[L_x[\phi_{in,i}]+\partial_sF\left(x,\Psi_0-\frac{\alpha}{2}|\log\ve|\,|x|^2\right)(\phi_{in,i}+\phi_{out})+S(\Psi_0)+N_0(\varphi)\bigg]\\
&\quad+L_x[\phi_{out}]+\bigg(1-\sum_{i=1}^N\tilde{\eta}_i\bigg)\bigg[\partial_sF\left(x,\Psi_0-\frac{\alpha}{2}|\log\ve|\,|x|^2\right)\phi_{out}+S(\Psi_0)+N_0(\varphi)\bigg]\\
&\quad+\sum_{i=1}^N\left(L_x[\tilde{\eta}_i\phi_{in,i}]-\tilde{\eta}_iL_x[\phi_{in,i}]\right),
\end{aligned}\]
where we have introduced the notation
\begin{equation}\label{nol}
\begin{aligned}
N_0(\varphi)(x)&=F\Big(x,\Psi_0-\frac{\alpha}{2}|\log\ve|\,|x|^2+\varphi\Big)-F\Big(x,\Psi_0-\frac{\alpha}{2}|\log\ve|\,|x|^2\Big)-\partial_sF\Big(x,\Psi_0-\frac{\alpha}{2}|\log\ve|\,|x|^2\Big)\varphi,\\
\partial_sF(x,s)&=\sum_{j=1}^N\frac{\gamma}{\ve^2\mu_j^2}\chi_j(x)\Big(\frac{s}{\kappa_j}+\nu'(1)|\log\ve|+\frac{\alpha}{2\kappa_j}|\log\ve|\,|P_j|^2\Big)_+^{\gamma-1}.
\end{aligned}
\end{equation}
Rearranging terms, we deduce that $\Psi$ in \eqref{gluingansatz} is an exact solution of \eqref{truesol}, provided that the components $\phi_{in}=(\phi_{in,1},\dots,\phi_{in,N})$ and $\phi_{out}$ satisfy the coupled system
\begin{equation}\label{inner1}
\begin{aligned}
L_x[\phi_{in,i}]+\partial_sF\Big(x,\Psi_0-\frac{\alpha}{2}|\log\ve|\,|x|^2\Big)(\phi_{in,i}+\phi_{out})+S(\Psi_0)+N_0\bigg(\sum_{j=1}^N\tilde{\eta}_j\phi_{in,j}+\phi_{out}\bigg)=0,&\\
\mbox{for}\,\,\,|A_i^{-1}(x-P_i)|<\frac{2\delta_1}{|\log\ve|^2},\qquad i=1,\dots,N,&
\end{aligned}
\end{equation}
and
\begin{equation}\label{outer1}
\begin{aligned}
L_x[\phi_{out}]
&+\bigg(1-\sum_{i=1}^N\tilde{\eta}_i\bigg)\bigg[\partial_sF\Big(x,\Psi_0-\frac{\alpha}{2}|\log\ve|\,|x|^2\Big)\phi_{out}+S(\Psi_0)+N_0\bigg(\sum_{j=1}^N\tilde{\eta}_j\phi_{in,j}+\phi_{out}\bigg)\bigg]\\
&+\sum_{i=1}^N\left(L_x[\tilde{\eta}_i\phi_{in,i}]-\tilde{\eta}_iL_x[\phi_{in,i}]\right)=0\quad\text{in }\R^2.
\end{aligned}
\end{equation}
In the sequel, we refer to \eqref{inner1} as the Inner problem and to \eqref{outer1} as the Outer problem. We remark that the estimates that we will establish later in Proposition \ref{propsec8} will ensure that the stream function $\Psi$ obtained through the gluing construction also satisfies \eqref{Transportverify}.

To analyse \eqref{inner1} near a vortex point $P_i$, we work in the rescaled variable $y=\frac{A_i^{-1}(x-P_i)}{\ve\mu_i}$, with $\mu_i=\mu_i^0+\mu_i^*$ satisfying \eqref{scalingmain} and \eqref{scalingpert}. Setting $\rho_i=\frac{2\delta_1}{\ve\mu_i|\log\ve|^2},$ we consider the region $|y|<\rho_i$ and use Proposition \ref{expansionofL} to obtain
\begin{equation}\label{rescaledLx}
\ve^2\mu_i^2L_x[\phi_{in,i}]=\Delta_y\phi_{in,i}+\widetilde B_0(y)[\phi_{in,i}], \qquad\mbox{with }\,\,\widetilde B_0(y)[\phi_{in,i}]=\ve^2\mu_i^2\left.B_0(z)\bigg[\phi_{in,i}\left(\frac{z}{\ve\mu_i}\right)\bigg]\right|_{z=\ve\mu_i y},
\end{equation}
see \eqref{oper}. Moreover, in the same region $|y|<\rho_i$, \eqref{nol} and a Taylor expansion give
\begin{equation}\label{defofbi}
\ve^2\mu_i^2\partial_sF\Big(x,\Psi_0(x)-\frac{\alpha}{2}|\log\ve|\,|x|^2\Big)=\gamma\Gamma_+^{\gamma-1}(y)+b_i(y),
\end{equation}
where 
\[|b_i(y)|\le C\ve\mu_i|\log\ve|\left(\Gamma\left(\frac{y}{2}\right)\right)_+^{\gamma-2}.\]
For the quadratic term, we define
\[\mathcal N_i(\varphi)(y)\defeq\ve^2\mu_i^2N_0(\varphi)(P_i+A_i\ve\mu_i y),\]
so that for perturbations satisfying $\|\varphi\|_{L^\infty(\bigcup_{i=1}^{N}\operatorname{supp}\chi_i)}\le c$, with $c>0$ fixed sufficiently small to preserve the negativity in \eqref{secondcaseprop}, we similarly find
\[|\mathcal N_i(\varphi)(y)|\le C\left(\Gamma\left(\frac{y}{2}\right)\right)_+^{\gamma-2}|\varphi(P_i+A_i\ve\mu_i y)|^2,\]
while Proposition \ref{errorprop1} implies that the function $\widetilde E_i(y)\defeq\ve^2\mu_i^2S(\Psi_0)(P_i+A_i\ve\mu_i y)$ satisfies 
\[
|\widetilde E_i(y)|\leq C\ve\mu_i|\log\ve|\left(\Gamma\left(\frac{y}{2}\right)\right)_+^\gamma.
\]
Consequently, after multiplication of \eqref{inner1} by $\ve^2\mu_i^2$, the Inner problem can be recast as
\begin{equation}\label{inprobexp}
\Delta_y\phi_{in,i}+\gamma\Gamma_+^{\gamma-1}(y)\phi_{in,i}+\bar B_i(y)[\phi_{in,i}]+H_i(\phi_{in},\phi_{out})(y)=0\quad\mbox{in }B_{\rho_i},\qquad i=1,\dots,N,
\end{equation}
where 
\begin{equation}\label{defbarBinner}
\bar B_i(y)[\phi_{in,i}]=\widetilde B_0(y)[\phi_{in,i}]+b_i(y)\phi_{in,i}(y)
\end{equation}
as in \eqref{rescaledLx}-\eqref{defofbi}, and
\begin{equation}\label{defHinner}
H_i(\phi_{in},\phi_{out})(y)=\mathcal N_i(\varphi)(y)+\widetilde E_i(y)+\big(\gamma\Gamma_+^{\gamma-1}(y)+b_i(y)\big)\phi_{out}(P_i+A_i\ve\mu_i y),
\end{equation}
with $\varphi$ given by \eqref{gluingansatz}.

Regarding the Outer problem in \eqref{outer1}, for notational simplicity we write
\begin{equation}\label{Poissonouter}
L_x[\phi_{out}]+G(\phi_{in},\phi_{out})=0\quad\mbox{in }\R^2,
\end{equation}
where, for $\phi_{in}=(\phi_{in,1},\dots,\phi_{in,N})$, we have introduced
\begin{equation}\label{Gout}
\begin{aligned}
G(\phi_{in},\phi_{out})&=\Big(1-\sum_{i=1}^N\tilde{\eta}_i\bigg)\bigg[\partial_sF\left(x,\Psi_0-\frac{\alpha}{2}|\log\ve|\,|x|^2\right)\phi_{out}+S(\Psi_0)+N_0(\varphi)\Big]\\
&\quad+\sum_{i=1}^N\left(L_x[\tilde{\eta}_i\phi_{in,i}]-\tilde{\eta}_iL_x[\phi_{in,i}]\right).
\end{aligned}
\end{equation}

\subsection{Solution of the Outer Problem}\label{subsecouter}
To initiate the resolution of the coupled system formed by \eqref{inprobexp} and \eqref{Poissonouter}, we first solve the Outer problem for a fixed admissible configuration $P=(P_1,\dots,P_N)$ satisfying \eqref{formpoints}, and full scales $\mu_i=\mu_i^0+\mu_i^*$, $i=1,\dots,N$ as in Proposition \ref{propscaling}. The resulting solution will then be substituted into the inner equations and the conditions determining the scaling corrections in Proposition \ref{propsec8}, which will be solved simultaneously.

For $\sigma>0$ and $\beta\in(0,1)$, we make the a priori assumptions
\begin{equation}\label{b100}
\phi_{out}\in C_{\mathrm{loc}}^{1,\beta}(\R^2),\qquad \|(1+|\cdot|^2)^{-1}\phi_{out}\|_{L^{\infty}(\R^2)}\leq100,
\end{equation}
and $\phi_{in,i}\in C^2(B_{\rho_i})$ satisfies
\begin{equation}\label{propertiesforphi1}
(1+|y|)^2|D_y^2\phi_{in,i}(y)|+(1+|y|)|D_y\phi_{in,i}(y)|+|\phi_{in,i}(y)| \leq\frac{C\ve\mu_i|\log\ve|}{1+|y|^{\frac{\sigma}{2}}},\quad i=1,\dots,N,
\end{equation}
where $\rho_i=\frac{2\delta_1}{\ve\mu_i|\log\ve|^2}$.

In this setting, the following lemma holds.
\begin{lemma}
Let $\phi_{out}$ and $\phi_{in,i}$, $i=1,\dots,N$, satisfy \eqref{b100} and \eqref{propertiesforphi1}. For $\delta>0$ sufficiently small and the cutoff functions $\chi_i$ and $\tilde\eta_i$ in \eqref{cutoffchi} and \eqref{cutoffgluing}, there exists $\ve_0>0$, independent of the admissible parameters, such that
\[
G(\phi_{in},\phi_{out})=\sum_{i=1}^N\left(L_x[\tilde\eta_i\phi_{in,i}]-\tilde\eta_iL_x[\phi_{in,i}]\right)\qquad\text{for }0<\ve<\ve_0,\]
where $G$ is defined in \eqref{Gout}.
\end{lemma}
\begin{proof}
Using \eqref{cutoffgluing} and the separation of the points in \eqref{unifdistance}, if $|A_i^{-1}(x-P_i)|\leq\frac{\delta_1}{2|\log\ve|^2}$ for some $i$, then $\tilde\eta_i=1$ and $\tilde\eta_j=0$ for $j\ne i$, thus $1-\sum\limits_{j=1}^N\tilde\eta_j=0$. Therefore, it suffices to focus on the complement of these inner regions, i.e.
\begin{equation}\label{regionoutsdeallinner}
|A_i^{-1}(x-P_i)|>\frac{\delta_1}{2|\log\ve|^2},\qquad \mbox{for all }i=1,\dots,N.
\end{equation}
By virtue of Proposition \ref{errorprop1} and \eqref{nol}, one has 
\begin{equation*}
\operatorname{supp}\partial_sF\Big(x,\Psi_0-\frac{\alpha}{2}|\log\ve|\,|x|^2\Big)\cup\operatorname{supp}S(\Psi_0) \subset\bigcup_{j=1}^N\Big\{x\in\R^2:\ |A_j^{-1}(x-P_j)|\leq\frac{3}{2}\ve\mu_j\Big\},
\end{equation*}
hence \eqref{Lxsupport} and \eqref{definitionerroroper} give \(L_x(\Psi_0)=0\) and $F\left(x,\Psi_0-\frac{\alpha}{2}|\log\varepsilon|\,|x|^2\right)=0,$ respectively. In order to establish the result, it remains to show that the term $ N_0(\varphi)$ in \eqref{Gout} vanishes identically in \eqref{regionoutsdeallinner}.

To do so, we split the analysis into the regions
\begin{equation*}
\begin{aligned}
U_{in,i}&\defeq\left\{x\in\R^2:\ \frac{\delta_1}{2|\log\ve|^2}<|A_i^{-1}(x-P_i)|<\frac{\delta}{|\log\ve|}\right\},\\ 
U_{out}&\defeq\bigcap_{j=1}^N\left\{x\in\R^2:\ |A_j^{-1}(x-P_j)|\geq\frac{\delta}{|\log\ve|}\right\}.
\end{aligned}
\end{equation*}
For $x\in U_{in,i}$, we observe that $\varphi(x)=\tilde\eta_i(x)\phi_{in,i}(y)+\phi_{out}(x)$, as \eqref{cutoffchi} yields $\chi_j(x)=0$ for all $j\neq i,$ while $\operatorname{supp}\tilde{\eta}_i\subset \{\chi_i=1\}.$ Uniformly in this annulus, we have $\log|y|=|\log\ve|+\mathcal O(\log|\log\ve|)$, hence \eqref{exppsiminusfinal}, \eqref{logmode1fun}, and \eqref{b100}-\eqref{propertiesforphi1} yield
\[
\frac{1}{\kappa_i}\left(\Psi_0-\frac{\alpha}{2}|\log\ve|\,|x|^2+\varphi\right)+\nu'(1)|\log\ve|+\frac{\alpha}{2\kappa_i}|\log\ve|\,|P_i|^2 =\nu'(1)|\log\ve|+\mathcal O(\log|\log\ve|)<0
\]
for all sufficiently small $\ve$, since $|\varphi|\leq C$ and $\nu'(1)<0$. 

Finally, for $x\in U_{out}$, \eqref{cutoffchi} readily gives $\chi_j=0$ for every $j\in\{1,\dots,N\}$ and $F(x,s)=0$ for all $s\in\R$, thus we conclude that $N_0(\varphi)=0$.
\end{proof}

\begin{corollary}
If $\phi_{out}$ and $\phi_{in,i}$, $i=1,\dots,N$, satisfy \eqref{b100} and \eqref{propertiesforphi1}, then the Outer problem \eqref{Poissonouter} reduces to
\begin{equation}\label{poisreduced}
L_x[\phi_{out}]=-\sum_{i=1}^N\left(L_x[\tilde\eta_i\phi_{in,i}]-\tilde\eta_iL_x[\phi_{in,i}]\right)\quad\mbox{in }\R^2.
\end{equation}
Consequently, we obtain a linear equation whose right-hand side is independent of $\phi_{out}$.
\end{corollary}
To estimate the right-hand side in \eqref{poisreduced}, we use \eqref{operL} to compute
\begin{equation}\label{outeroperatorphiin}
L_x[\tilde\eta_i\phi_{in,i}]-\tilde\eta_iL_x[\phi_{in,i}] =\phi_{in,i}(y_i)L_x\tilde\eta_i+\frac{2}{\ve\mu_i}K(x)\nabla_x\tilde\eta_i\cdot A_i^{-T}\nabla_y\phi_{in,i}(y_i),
\end{equation}
which is supported in the annulus
\[\frac{\delta_1}{2|\log\ve|^2}\leq|A_i^{-1}(x-P_i)|\leq\frac{\delta_1}{|\log\ve|^2}.\]
In this region, \eqref{cutoffgluing} gives
\[|\nabla_x\tilde\eta_i|\leq C|\log\ve|^2,\qquad |L_x\tilde\eta_i|\leq C|\log\ve|^4,\]
while using \eqref{propertiesforphi1} we deduce that
\[|\phi_{in,i}(y_i)|\leq C(\ve\mu_i)^{1+\frac{\sigma}{2}}|\log\ve|^{1+\sigma},\qquad |\nabla_x[\phi_{in,i}(y_i)]|\leq C(\ve\mu_i)^{1+\frac{\sigma}{2}}|\log\ve|^{3+\sigma}.\]
In this way, for any  $\bar\nu>2$ we find 
\begin{equation}\label{outerforcingbound}
\left|L_x[\tilde\eta_i\phi_{in,i}]-\tilde\eta_iL_x[\phi_{in,i}]\right| \leq\frac{C(\ve\mu_i)^{1+\frac{\sigma}{2}}|\log\ve|^{5+\sigma}}{1+|x|^{\bar\nu}} \leq\frac{C\ve^{1+\sigma^*}}{1+|x|^{\bar\nu}}\qquad \forall x\in\R^2,
\end{equation}
where $0<\sigma^*<\frac{\sigma}{2}$.

We now fix the linear solution operator $\mathcal T^o$ from Proposition \ref{propouter} and define the outer solution by
\begin{equation}\label{outermap}
\begin{aligned}
\phi_{out}[\phi_{in};P,\mu_1,\dots,\mu_N](x) &=\mathcal T^o\bigg[\sum_{i=1}^N\big(L_x[\tilde\eta_i\phi_{in,i}]-\tilde\eta_iL_x[\phi_{in,i}]\big)\bigg](x)\\
&-\mathcal T^o\bigg[\sum_{i=1}^N\big(L_x[\tilde\eta_i\phi_{in,i}]-\tilde\eta_iL_x[\phi_{in,i}]\big)\bigg]\big((r_0,0)\big),
\end{aligned}
\end{equation}
which specifies the choice of solution as the parameters vary. Then, the Outer linear theory in Proposition \ref{propouter},  \eqref{outerforcingbound} and local elliptic estimates yield
\begin{equation}\label{boundphiout}
\|(1+|\cdot|^2)^{-1}\phi_{out}\|_{L^{\infty}(\R^2)}+\|\phi_{out}\|_{C^{1,\beta}(B_1((r_0,0)))}\leq C\ve^{1+\sigma^*}.
\end{equation}
Together with \eqref{propertiesforphi1} and \eqref{globalapprox}, the estimate in \eqref{boundphiout} places $\Psi=\Psi_0+\varphi$  in the class considered in the verification of \eqref{Transportverify}.
\section{Inner and Outer Linear Theories}\label{Section 7}
This section collects the linear solvability results and a priori estimates for the Inner and Outer problems introduced in Section \ref{Section 6}. These constitute the main analytic tools for solving the coupled system in \eqref{inprobexp} and \eqref{Poissonouter}. For readability, the proofs of the inner solvability results can be found in the Appendix.
\subsection{Linear Inner Theory}
Let $h(y)$ be a function satisfying the decay condition 
\begin{equation*}
|h(y)|\leq \frac{C}{1+|y|^{2+\sigma}},
\end{equation*} 
for some constant $C>0$ and $\sigma\in(0,1)$. The current section considers the linear equation
\be\label{Linearisedinner}
\begin{aligned}
\Delta_{y} \phi + \gamma\Gamma^{\gamma-1}_{+}\phi +  h(y)= 0 \inn \, \R^2, 
\end{aligned}
\ee
where $\Gamma$ is given in \eqref{defGamma}.

For $\beta\in(0,1)$, we employ the weighted norms
\be\label{norma} \begin{aligned}
\|h\|_{2+\sigma}  = & \sup_{y\in \R^2} (1+|y|)^{2+\sigma}|h(y)|, \\
\|h\|_{2+\sigma,\beta}  = &\|h\|_{2+\sigma}  +   \sup_{y\in \R^2}(1+|y|)^{2+\sigma + \beta}[h]_{B_{\frac{1+|y|}{8}}(y),\beta},
\end{aligned}
\ee
where for any $A\subset \R^2,$ we use the standard H\"{o}lder seminorm notation
\[[h]_{A,\beta} = \sup_{\ell_1\neq \ell_2,\, \ell_1,\ell_2\in A}  \frac {|h(\ell_1) -h(\ell_2)| } {|\ell_1-\ell_2|^\beta}.\]
In \cite{MR2456896}, the authors established that the linear operator $\Delta_y +\gamma\Gamma^{\gamma-1}_{+}$ in \eqref{Linearisedinner} is nondegenerate in $L^{\infty}(\R^2)$, in the sense that its kernel is generated by the bounded translation modes $Z_i(y)=\partial_{y_i}\Gamma(y), \ i=1,2$. On the other hand, if one enlarges the class of admissible solutions to allow logarithmic growth at infinity, then the kernel becomes three-dimensional, and it is spanned by
\begin{equation}\label{elemkernel}
Z_0(y)=\frac{2}{\gamma-1}\Gamma(y)+y\cdot\nabla_y \Gamma(y), \quad  Z_i(y)  =  \partial_{y_i}  \Gamma(y), \  i=1,2.
\end{equation}
Further details on the radial scaling mode $Z_0$ can be found in \cite{MR5036920}*{Section 5.2}. Here, we note that $Z_0(y)=\mathcal{O}\left(\log\left(2+|y|\right)\right)$ and $Z_i(y)=\mathcal{O}\left(\frac{1}{1+|y|}\right),\,i=1,2$. 

The following lemma holds.
\begin{lemma}\label{lemat}
Given $\sigma\in(0,1)$ and $\beta\in (0,1)$, consider the norms defined in \eqref{norma} and the functions $Z_i,\,i=0,1,2,$ given by \eqref{elemkernel}. Then, there exists a constant $C>0$ such that for each $h$ with $\|h\|_{2+\sigma} <+\infty$, equation \eqref{Linearisedinner} has a solution $\phi =  \mathcal{T} [ h]$ that defines a linear operator of $h$ and satisfies
\begin{equation*}
\begin{aligned}
&(1+|y|) |\nabla \phi (y)|  + | \phi (y)| \\
&\le  \,  C \bigg[ \,  \bigg|\int_{\R^2} h Z_0\bigg|  +    (1+|y|) \sum_{i=1}^2 \bigg|\int_{\R^2} h Z_i\bigg|+  (1+|y|)^{-\sigma} \|h\|_{2+\sigma}   \,\bigg].
\end{aligned}
\end{equation*}
Moreover, if 
$\|h\|_{2+\sigma,\beta} <+\infty$, it further holds
\begin{equation*}
\begin{aligned}
& (1+|y|^{2+\beta})  [D^2_y \phi]_{B_{\frac{1+|y|}{8}}(y),\beta}  +\left(1+|y|^2\right)  \left|D^2_y \phi (y)\right| \\  &
\,  \le  \,   C \bigg [ \,  \bigg|\int_{\R^2} h Z_0\bigg|  +    (1+|y|) \sum_{i=1}^2 \bigg|\int_{\R^2} h Z_i\bigg|  +  (1+|y|)^{-\sigma} \|h\|_{2+\sigma,\beta}  \,\bigg]. 
\end{aligned}
\end{equation*}
\end{lemma}

\subsection{Projected Linearised Inner Problem}
In this section, we fix $\delta_{1}>0$ sufficiently small, take scaling parameters $\mu_i,\,i=1,\dots,N$ as in Proposition \ref{propscaling}, and define 
\[\rho_i \defeq \frac{2\delta_1}{\ve\mu_i|\log\ve|^2}.\]
For certain solvability scalars $d_{ij},\,j=0,1,2,$ and a fixed $i\in\{1,\dots,N\},$ we  consider the projected linearised problem
\begin{equation}\label{problemlinearprojected}
\Delta_{y}\phi_{i}+\gamma\Gamma^{\gamma-1}_{+}\phi_{i}+\bar{B}_i[\phi_{i}]+h(y)=\sum_{j=0}^{2}d_{ij}\gamma\Gamma^{\gamma-1}_{+}Z_j\quad\mbox{in}\quad B_{\rho_i}.
\end{equation}
For convenience, we recall that the definition of the functions $Z_j, j=0,1,2$ can be found in \eqref{elemkernel}, while 
 \begin{equation}\label{defofBinner}
\bar{B}_i[\phi_{i}]=b_i(y)\phi_{i}+\widetilde{B}_0[\phi_{i}],
\end{equation}
 with $b_i(y)=\mathcal{O}\left(\ve\mu_i|\log\ve|\big(\Gamma(\frac{y}{2})\big)^{\gamma-2}_{+}\right)$ and $\widetilde{B}_{0}=\ve^2\mu_i^2B_{0}(\ve\mu_i y)$; see \eqref{oper} and \eqref{rescaledLx}-\eqref{defofbi}. 

Moreover, for $\sigma\in(0,1)$,\, $\beta\in(0,1)$, and any open set $A \subset \R^2$, we modify the norms in \eqref{norma} taking all suprema over $A$, namely
\begin{equation}\label{norm1}
\begin{aligned}
\|h\|_{2+\sigma,A}&=\sup_{y\in A}\left(1+|y|\right)^{2+\sigma}|h(y)|,\\ \|h\|_{2+\sigma,\beta,A}&=\sup_{y\in A}\left(1+|y|\right)^{2+\sigma+\beta}[h]_{B_{\frac{1+|y|}{8}}(y)\cap A,\beta}\,+\|h\|_{2+\sigma,A},
\end{aligned}
\end{equation}
while for $\phi_{i}\in C^{2,\beta}(A)$ we define the norm
 \begin{equation}\label{starnorm}
 \|\phi_{i}\|_{*,\sigma,A} =  \|  D^2_y\phi_{i} \|_{2+\sigma,\beta,A}   + \|  D_y\phi_{i} \|_{1+\sigma,A}+ \|\phi_{i} \|_{\sigma,A}.
 \end{equation}
With a slight abuse of notation, we will suppress the subscript in \eqref{norm1} and \eqref{starnorm} when $A=\R^2$. 

We establish the following result.
\begin{prop}\label{prop1} 
Let $\sigma,\beta\in(0,1)$, and consider the norms defined in \eqref{norm1} and \eqref{starnorm}. For the vortex points and full scaling parameters satisfying the assumptions of Proposition \ref{propscaling}, let $\bar B_i$ be the operator defined in \eqref{defofBinner} and set $\rho_i=\frac{2\delta_1}{\ve\mu_i|\log\ve|^2}$. Then, there exist constants $C>0$ and $\ve_0>0$, uniform in $i$ and the admissible parameters, such that for every $0<\ve<\ve_0$ and $h$ satisfying $\|h\|_{2+\sigma,\beta,B_{\rho_i}}<+\infty$, the projected problem \eqref{problemlinearprojected} has a solution $\phi_i=T_i[h]$ for certain scalars $d_{ij}[h]$, $j=0,1,2,$ with both depending linearly on $h$, such that
\[
\|\phi_i\|_{*,\sigma,B_{\rho_i}}\le C\|h\|_{2+\sigma,\beta,B_{\rho_i}}.
\]
Moreover, the projection coefficients $d_{ij}$ take the form
\begin{align*}
d_{i0}[h]&=\gamma_0\int_{B_{\rho_i}}hZ_0+\mathcal O\left(\frac{\log(2+\rho_i)}{1+\rho_i^\sigma}\right)\|h\|_{2+\sigma,\beta,B_{\rho_i}},\\
d_{ij}[h]&=\gamma_j\int_{B_{\rho_i}}hZ_j+\mathcal O\left(\ve\mu_i|\log\ve|\right)\|h\|_{2+\sigma,\beta,B_{\rho_i}},\quad j=1,2,
\end{align*}
where $\gamma_j^{-1}=\int_{\R^2}\gamma\Gamma_+^{\gamma-1}Z_j^2$, $j=0,1,2$.
\end{prop}
\subsection{Linear Outer Theory}
We consider the differential operator $L_x$ defined in \eqref{operL}, and study the elliptic equation  
\begin{equation}
\label{outerprob}
L_x[\psi] + g(x) = 0  \inn \R^2 ,
\end{equation}
where $g(x)$ is a bounded function satisfying the decay condition
\[\|g\|_{\bar{\nu}}\, :=\, \sup_{x\in \R^2}\left(1+|x|^{\bar{\nu}}\right)|g(x)|< +\infty,\]
for some $\bar{\nu} >2$. 
    
 The following Proposition holds.
\begin{prop}\label{propouter}
There exists a solution $\psi(x)$ to \eqref{outerprob} of class $C^{1,\beta}_{\rm loc}(\R^2)$ for any $\beta\in(0,1)$, which defines a linear operator $\psi= {\mathcal T}^o (g) $ of $g(x)$
and satisfies the bound
\begin{equation*}
|\psi(x)| \,\le \,C\| g\|_{\bar{\nu}} (1+ |x|^2) \quad \forall x \in \R^2,
\end{equation*}
for a constant $C>0$.
\end{prop}
\begin{proof}
See \cite{MR4417384}*{Proposition 7.1}.
\end{proof}

\section{Solvability of the Inner-Outer System}\label{Section projected}
With the outer solution operator already constructed in Section \ref{subsecouter}, we now aim to utilise the linear solvability theories developed in Section \ref{Section 7} to construct, for each $i=1,\dots,N$, a solution $\phi_{in,i}$ of the Inner problem
\[\Delta_y\phi_{in,i}+\gamma\Gamma_+^{\gamma-1}\phi_{in,i}+\bar B_i[\phi_{in,i}]+H_i(\phi_{in},\phi_{out})=0\quad\mbox{in }B_{\rho_i}.\]
We recall the notation $\phi_{in}=(\phi_{in,1},\dots,\phi_{in,N})$ and write $\phi_{out}=\phi_{out}[\phi_{in};P,\mu_1,\dots,\mu_N]$ for the solution of \eqref{poisreduced} associated with prescribed admissible inner corrections, vortex points and full scaling parameters. This correction satisfies \eqref{outermap} and \eqref{boundphiout}. In addition, $\rho_i=\frac{2\delta_1}{\ve\mu_i|\log\ve|^2}$, while the definitions of $\bar B_i$ and $H_i$ appear in \eqref{defofBinner} and \eqref{defHinner}. 

As a starting point, we decompose
\[\phi_{in,i}=\phi_{i,1}+\phi_{i,2},\qquad i=1,\dots,N,\]
so that for the functions $Z_j$ in \eqref{elemkernel} and coefficients $d_{ij}$, $j=0,1,2$, supplied by Proposition \ref{prop1}, the function $\phi_{i,1}$ solves the auxiliary projected problem
\begin{equation}\label{problemforphi1}
\Delta_y\phi_{i,1}+\gamma\Gamma_+^{\gamma-1}\phi_{i,1}+\bar B_i[\phi_{i,1}]+\bar B_i[\phi_{i,2}]+H_i(\phi_{in},\phi_{out})
=\sum_{j=0}^2d_{ij}\gamma\Gamma_+^{\gamma-1}Z_j\quad\mbox{in }B_{\rho_i},
\end{equation}
whereas $\phi_{i,2}$ is required to satisfy
\begin{equation}\label{problemforphi2}
\Delta_y\phi_{i,2}+\gamma\Gamma_+^{\gamma-1}\phi_{i,2}+d_{i0}\gamma\Gamma_+^{\gamma-1}Z_0=0\quad\mbox{in }\R^2.
\end{equation}
Adding these equations, we see that the Inner-Outer system in \eqref{inprobexp} and \eqref{poisreduced} admits a solution of the form
\[(\phi_{in},\phi_{out}),\qquad \phi_{in}=(\phi_{1,1}+\phi_{1,2},\dots,\phi_{N,1}+\phi_{N,2}),\]
provided that we also satisfy
\begin{equation}\label{redconds}
d_{ij}\bigl[H_i(\phi_{in},\phi_{out})+\bar B_i[\phi_{i,2}]\bigr]=0,\qquad i=1,\dots,N,\quad j=1,2.
\end{equation}
We postpone the analysis of \eqref{redconds} to Section \ref{reduced}, where we adjust the locations of the vortex points $P_1,\dots,P_N$ in \eqref{formpoints} to achieve their solvability.

We now examine the equation for $\phi_{i,2}$ in \eqref{problemforphi2}. Writing $r=|y|$, we define $\hat\phi_2(r)$ as the unique bounded regular solution of the problem
\begin{equation}\label{radialhatphi}
\left\{\begin{aligned}
&\hat\phi_2''(r)+\frac1r\hat\phi_2'(r)+\gamma\Gamma_+^{\gamma-1}(r)\hat\phi_2(r)
+\gamma\Gamma_+^{\gamma-1}(r)Z_0(r)=0,\qquad r>0,\\
&\hat\phi_2'(0)=0,\quad \,\|\hat\phi _2\|_{L^{\infty}(\R^2)} < \infty.
\end{aligned}\right.
\end{equation}
The radial construction in the proof of Lemma \ref{lemat} gives existence and uniqueness of a solution, with $\hat{\phi}_{2}(r
)=\hat{\phi}_2(0)+\mathcal{O}(r^2)$ and $\hat{\phi}'_2(r)=\mathcal{O}(r)$ as $r\to 0$. Since $\gamma>3$, the resulting radial function belongs to $C^3(\R^2)$ and is independent of $\ve$, the vortex points, and the full scaling parameters. We then set
\begin{equation*}
\phi_{i,2}(y)=d_{i0}\hat\phi_2(y),\qquad i=1,\dots,N,
\end{equation*}
which solves \eqref{problemforphi2}.

Since $\Gamma_+^{\gamma-1}$ vanishes for $r\ge1$, \eqref{radialhatphi} and boundedness give
\[\hat\phi_2(r)\equiv \hat{\phi}_2(1),\qquad r\ge1,\]
which together with the regularity on $\overline{B_1}$ also yields
\begin{equation}\label{estimateforphi2}
\begin{aligned}
|\hat\phi_2(y)|&\le C,&
|D_y^k\hat\phi_2(y)|&\le\frac{C}{(1+|y|)^k},\\
|\phi_{i,2}(y)|&\le C|d_{i0}|,&
|D_y^k\phi_{i,2}(y)|&\le\frac{C|d_{i0}|}{(1+|y|)^k},
\end{aligned}
\end{equation}
for all $y\in\R^2$, $k=1,2,3$ and $i=1,\dots,N$, for a universal $C>0$. In particular, the derivatives in \eqref{estimateforphi2} vanish outside $B_1.$

The Proposition below is the main result of this section. 
\begin{prop}\label{propsec8}
Let $\sigma,\beta\in(0,1)$ and $0<\sigma^*<\frac{\sigma}{2}$. For $\delta_1>0$ sufficiently small, every admissible configuration $P=(P_1,\dots,P_N)$ and all sufficiently small $\ve>0$, there exist positive scaling parameters $\mu_i=\mu_i^0+\mu_i^*$, functions $\phi_{out}$, $\mathbf\Phi_1=(\phi_{1,1},\dots,\phi_{N,1})$, $\mathbf\Phi_2=(\phi_{1,2},\dots,\phi_{N,2})$, and coefficients $d_{ij}$, $j=0,1,2$, solving \eqref{poisreduced}, \eqref{problemforphi1} and \eqref{problemforphi2}.

Here $\mu_i^0$ is defined by \eqref{scalingmain}, $\phi_{i,2}=d_{i0}\hat\phi_2$ with $\hat\phi_2$ the bounded regular solution of \eqref{radialhatphi}, and
\begin{equation}\label{scalingpertexplicit}
-\kappa_i\nu'(1)\log\Big(1+\frac{\mu_i^*}{\mu_i^0}\Big)=\phi_{out}(P_i),\qquad i=1,\dots,N.
\end{equation}
Moreover, there exists $C>0$ such that
\begin{equation}\label{boundsforsoln}
\begin{aligned}
&\|(1+|\cdot|^2)^{-1}\phi_{out}\|_{L^\infty(\R^2)}+\|\phi_{out}\|_{C^{1,\beta}(B_1((r_0,0)))}\le C\ve^{1+\sigma^*},\qquad \|\phi_{i,1}\|_{*,\sigma,B_{\rho_i}}\le C\ve\mu_i|\log\ve|,\\
&|d_{i0}|\le C(\ve\mu_i)^{1+\sigma}|\log\ve|^{1+2\sigma}, \qquad \Big|\log\Big(1+\frac{\mu_i^*}{\mu_i^0}\Big)\Big|+\frac{|\mu_i^*|}{\mu_i^0}\le C\ve^{1+\sigma^*},
\end{aligned}
\end{equation}
where $\rho_i=\frac{2\delta_1}{\ve\mu_i|\log\ve|^2}$ and the norm $\|\cdot\|_{*,\sigma,B_{\rho_i}}$ is defined in \eqref{starnorm}. 

The inner functions $\phi_{in,i}=\phi_{i,1}+\phi_{i,2}$ satisfy $\|\phi_{in,i}\|_{*,\frac{\sigma}{2},B_{\rho_i}}\le C\ve\mu_i|\log\ve|$ and verify a posteriori \eqref{propertiesforphi1}, while the   remaining coefficients for $j=1,2$ satisfy $|d_{ij}|\le C\ve\mu_i|\log\ve|$. Finally, \eqref{boundsforsoln} establishes \eqref{scalingpert} with $\tilde\sigma=\sigma^*$, and the constructed solutions and scales depend continuously on $P_j$.
\end{prop}
\begin{proof}
We keep the points $P=(P_1,\dots,P_N)$ fixed and first record the dependence of the leading scales $\mu_i^0$ on the full scales $\mu_1,\dots,\mu_N$. For $C_0>0$ to be chosen below independently of $\ve$ and the admissible configuration, we consider a vector of positive scales $(\mu_1,\dots,\mu_N)$ such that
\begin{equation}\label{scalesbox}
\max_{1\le i\le N}\Big|\log\mu_i-\big(\frac{1}{\kappa_i}\sum\limits_{j\ne i}\kappa_j\big)\log|\log\ve|\Big|\le C_0.
\end{equation}
Notice that before imposing \eqref{scalingpertexplicit}, the identity in \eqref{scalingmain} defines leading-order scales $\mu_i^0>0$ satisfying
\begin{equation}\label{leadinglogscalebounds}
\log\mu_i^0=\big(\frac{1}{\kappa_i}\sum\limits_{j\ne i}\kappa_j\big)\log|\log\ve|+\mathcal O(1),\qquad \max_i\sum_{j=1}^N\left|\mu_j\partial_{\mu_j}\log\mu_i^0\right|\le C\frac{\log|\log\ve|}{|\log\ve|}.
\end{equation}
The first estimate follows from the proof of Proposition \ref{propscaling}, uniformly for the prescribed full scales in \eqref{scalesbox}. We establish the second estimate by differentiating with respect to the full scales, keeping the points fixed. By \eqref{defGamma}, \eqref{defofgammahat} and \eqref{Rescaledh1}, differentiation with $z$ fixed gives
\begin{equation}\label{mujderivatives}
\mu_j\partial_{\mu_j}\widehat\Gamma_{\ve\mu_j}(z)=0,\qquad |\mu_j\partial_{\mu_j}H_1(z)|\le C\frac{(\ve\mu_j)^4}{|z|^3},\qquad \ve\mu_j\le|z|<\delta<1.
\end{equation}
Then, differentiating \eqref{Definitionofg} and using \eqref{oper} and Proposition \ref{propapprox} we also obtain
\[
\|\mu_j\partial_{\mu_j}g\|_{L^\infty(\R^2)}+\ve\min_{\ell}\mu_{\ell}\,\|\nabla_x(\mu_j\partial_{\mu_j}g)\|_{L^\infty(\R^2)}\le C.
\]
Due to \eqref{Definitionofg} and \eqref{etazero}, these sources are supported in the fixed set $\operatorname{supp}\eta_0$, hence \eqref{preciseH2epsilon}, Proposition \ref{propouter} and the interior elliptic estimates used in the proof of Proposition \ref{propscaling} yield
\begin{equation}\label{H2scalebounds}
\|\mu_j\partial_{\mu_j}H_{2\ve}\|_{C^1(B_2((r_0,0)))}+\frac{\|D_x^2(\mu_j\partial_{\mu_j}H_{2\ve})\|_{L^\infty(B_2((r_0,0)))}}{|\log\ve|}\le C.
\end{equation}
In addition, the normalisation in \eqref{preciseH2epsilon} holds for every choice of the full scales, thus $\mu_j\partial_{\mu_j}H_{2\ve}((r_0,0))=0$, while the fundamental theorem of calculus together with the $C^1$ bound in \eqref{H2scalebounds} and \eqref{formpoints}  gives
\begin{equation}\label{dermujforH2e}
\begin{aligned}
|\mu_j\partial_{\mu_j}H_{2\ve}(P_i)|
&=\Big|\int_0^1\nabla_x\left(\mu_j\partial_{\mu_j}H_{2\ve}\right)\bigl((r_0,0)+t(P_i-(r_0,0))\bigr)\cdot(P_i-(r_0,0))\,\mathrm dt\Big|\\
&\le C|P_i-(r_0,0)|\\
&\le C\frac{\log|\log\ve|}{|\log\ve|}.
\end{aligned}
\end{equation}
Since the matrices $A_j$ in \eqref{var} and the coefficients $c_{1,j},c_{2,j}$ in \eqref{constantsj} depend only on the points $P_j$, differentiating \eqref{scalingmain} we get
\[\begin{aligned}
-\kappa_i\nu'(1)\mu_j\partial_{\mu_j}\log\mu_i^0&=\frac{\kappa_jR_j^3}{2h(h^2+R_j^2)^{\frac32}}\mu_j\partial_{\mu_j}H_1\big(A_j^{-1}(P_i-P_j)\big)+\mu_j\partial_{\mu_j}H_{2\ve}(P_i), \quad j\neq i\\
-\kappa_i\nu'(1)\mu_i\partial_{\mu_i}\log\mu_i^{0}&= \mu_i\partial_{\mu_i}H_{2\ve}(P_i),
\end{aligned}\]
while using \eqref{mujderivatives}, \eqref{dermujforH2e}, the uniform separation in \eqref{unifdistance} and \eqref{scalesbox}, we conclude that
\[
\max_i\sum_{j=1}^N|\mu_j\partial_{\mu_j}\log\mu_i^0|
\le C\bigg(\frac{\log|\log\ve|}{|\log\ve|}+\sum_{j=1}^N(\ve\mu_j)^4|\log\ve|^3\bigg)
\le C\frac{\log|\log\ve|}{|\log\ve|},
\]
which proves the second estimate in \eqref{leadinglogscalebounds}, uniformly in the admissible parameters.

Now, we define \[\bar\mu_i=|\log\ve|^{\frac{1}{\kappa_i}\sum\limits_{j\ne i}\kappa_j},\,i=1,\dots,N,\] 
where \eqref{leadinglogscalebounds} yields
\begin{equation}\label{uniformleadingremainder}
\max_i\Big|\log\Big(\frac{\mu_i^0}{\bar\mu_i}\Big)\Big|\le C_*
\end{equation}
for all sufficiently small $\ve$, with $C_*>0$ independent of $C_0$, the admissible points $P_j$ and the scales $\mu_j$. We then choose $C_0\ge2(C_*+1)$ and restrict the full scaling parameters $\mu_i$ to
\begin{equation}\label{admissiblefullscales}
\max_i\Big|\log\Big(\frac{\mu_i}{\bar\mu_i}\Big)\Big|\le C_0.
\end{equation}
We recast the coupled Inner-Outer system in \eqref{poisreduced}, \eqref{problemforphi1} and \eqref{problemforphi2} as a fixed point problem of the form
\begin{equation}\label{fixedfinal}
\bigl(\mathbf\Phi_1,(d_{10},\dots,d_{N0})\bigr)=\widetilde{\mathcal F}\bigl(\mathbf\Phi_1,(d_{10},\dots,d_{N0})\bigr),
\end{equation}
where the components of $\widetilde{\mathcal F}$ are defined by
\begin{equation*}
\widetilde\phi_{i,1}=T_i[h_i],\qquad \widetilde d_{i0}=d_{i0}[h_i],\qquad \mbox{with }\,\, h_i=H_i(\mathbf\Phi_1+\mathbf\Phi_2,\phi_{out})+\bar B_i[\phi_{i,2}], \qquad i=1,\dots,N.
\end{equation*}
Here $T_i$ denotes the global operator constructed in Proposition \ref{prop1} (see \eqref{fp}) acting on the global source $h_i$ before restriction to $B_{\rho_i}$, while the definitions of $\phi_{out}$, $H_i$ and $\bar B_i$ are given in \eqref{outermap}, \eqref{defHinner} and \eqref{defbarBinner}, respectively. In addition, $\phi_{i,2}=d_{i0}\hat\phi_2$, with $d_{i0}$ the coefficient of the trial pair, while \eqref{coeffsforortho} gives
\begin{equation}\label{projcoeffinfixedpoint}
\widetilde d_{i0}=\gamma_0\int_{\R^2}\bigl(H_i(\mathbf\Phi_1+\mathbf\Phi_2,\phi_{out})+\bar B_i[\phi_{i,2}]+\bar B_i[\widetilde\phi_{i,1}]\bigr)Z_0,\qquad \gamma_0^{-1}=\int_{\R^2}\gamma\Gamma_+^{\gamma-1}Z_0^2.
\end{equation}
The operators and sources above will be evaluated at the full scales selected below through \eqref{scalefixedpoint}. We then seek a solution of \eqref{fixedfinal} in the ball
\begin{equation}\label{contractionball}
\begin{aligned}
\mathcal B\defeq\Big\{\bigl(\mathbf\Phi_1,(d_{10},\dots,d_{N0})\bigr)\in\mathcal X_*^N\times\R^N:&
\ \|\phi_{i,1}\|_{*,\sigma}\le\widehat C\ve\bar\mu_i|\log\ve|,\\
&\,\,\,|d_{i0}|\le\widehat C(\ve\bar\mu_i)^{1+\sigma}|\log\ve|^{1+2\sigma},\quad i=1,\dots,N\Big\},
\end{aligned}
\end{equation}
for some $\widehat C>0$ to be chosen later, where $\mathcal X_*\defeq\big\{\phi\in C^{2,\beta}_{\mathrm{loc}}(\R^2):\ \|\phi\|_{*,\sigma}<\infty\big\}$.

For each element in \eqref{contractionball} and prescribed full scales $\mu_i$, we consider the outer solution $\phi_{out}$ given by \eqref{outermap}. Since $\hat\phi_2$ is constant for $|y|\ge1$, we have
\begin{equation}\label{commphi2}
L_x[\tilde\eta_j\phi_{j,2}]-\tilde\eta_jL_x[\phi_{j,2}]
=d_{j0}\hat\phi_2(1)L_x\tilde\eta_j,
\end{equation}
therefore using \eqref{outeroperatorphiin}, \eqref{cutoffgluing} and Proposition \ref{propouter} we obtain
\begin{equation}\label{outerbounddecomposed}
\begin{aligned}
\left\|(1+|\cdot|^2)^{-1}\phi_{out}\right\|_{L^\infty(\R^2)}+\|\phi_{out}\|_{C^{1,\beta}(B_1((r_0,0)))}
&\le C\sum_{j=1}^N\left((\ve\bar\mu_j)^\sigma|\log\ve|^{4+2\sigma}\|\phi_{j,1}\|_{*,\sigma}+|\log\ve|^4|d_{j0}|\right)\\
&\le C\widehat C\ve^{1+\sigma^*}.
\end{aligned}
\end{equation}
This linear outer map is defined throughout \eqref{admissiblefullscales}, before imposing \eqref{scalingpertexplicit}. With the points and trial pair fixed, \eqref{commphi2} is independent of the full scales, thus using the identity $\mu_j\partial_{\mu_j}y_j=-y_j$ and differentiating gives
\[
\left|\mu_j\partial_{\mu_j}\left(L_x[\tilde\eta_j\phi_{in,j}]-\tilde\eta_jL_x[\phi_{in,j}]\right)\right|
\le C|\log\ve|^4\left(|\phi_{j,1}|+|y_j||D_y\phi_{j,1}|+|y_j|^2|D_y^2\phi_{j,1}|\right)
\]
on its cutoff annulus. Proposition \ref{propouter} and local elliptic estimates then yield
\begin{equation}\label{outerparameterbounds}
\left\|(1+|\cdot|^2)^{-1}\mu_j\partial_{\mu_j}\phi_{out}\right\|_{L^\infty(\R^2)}+\|\mu_j\partial_{\mu_j}\phi_{out}\|_{C^{1,\beta}(B_1((r_0,0)))}
\le C(\ve\bar\mu_j)^\sigma|\log\ve|^{4+2\sigma}\|\phi_{j,1}\|_{*,\sigma}\le C\widehat C\ve^{1+\sigma^*}.
\end{equation}
At this point, we require that the full scales $\mu_i=\mu_i^{0}+\mu_i^{*}$ satisfy the conditions
\begin{equation}\label{scalingpertexplicitinproof}
-\kappa_i\nu'(1)\log\Big(1+\frac{\mu_i^*}{\mu_i^0}\Big)
=\phi_{out}(P_i),\qquad i=1,\dots,N.
\end{equation}
In particular, since both $\mu_i^0$ and $\phi_{out}$ depend on the full vector $(\mu_1,\dots,\mu_N)$, these conditions are equivalent to the fixed point equations
\begin{equation}\label{scalefixedpoint}
\log\mu_i=\log\mu_i^0(P,\mu_1,\dots,\mu_N)
-\frac{\phi_{out}[\mathbf\Phi_1+\mathbf\Phi_2;P,\mu_1,\dots,\mu_N](P_i)}
{\kappa_i\nu'(1)}, \quad i=1,\dots,N.
\end{equation}
For fixed $\widehat C>0$ in \eqref{contractionball}, \eqref{outerbounddecomposed} gives $\max\limits_i\left|\frac{\phi_{out}(P_i)}{\kappa_i\nu'(1)}\right|\le1$ for sufficiently small $\ve$, thus  \eqref{uniformleadingremainder} shows that the right-hand side of \eqref{scalefixedpoint} differs from $\log\bar\mu_i$ by at most $C_*+1\le\frac{C_0}{2}$ and therefore preserves \eqref{admissiblefullscales}.

Moreover, \eqref{leadinglogscalebounds} and \eqref{outerparameterbounds} bound its Lipschitz constant in the maximum norm of $(\log\mu_1,\dots,\log\mu_N)$ by
\[C\frac{\log|\log\ve|}{|\log\ve|}+C\widehat C\ve^{1+\sigma^*}<\frac12\]
for sufficiently small $\ve$. As such, for each trial pair in \eqref{contractionball}, Banach's fixed point Theorem gives a unique positive vector $(\mu_1,\dots,\mu_N)$ satisfying \eqref{scalefixedpoint} within the admissible range \eqref{admissiblefullscales}. Defining $\mu_i^*=\mu_i-\mu_i^0$ at the selected full scales, \eqref{scalingpertexplicitinproof}, \eqref{outerbounddecomposed} and the inequality $|e^s-1|\le C|s|$ for $|s|\le1$ yield
\begin{equation*}
\Big|\log\Big(1+\frac{\mu_i^*}{\mu_i^0}\Big)\Big|
+\frac{|\mu_i^*|}{\mu_i^0}\le C\widehat C\ve^{1+\sigma^*}.
\end{equation*}

Returning to the fixed point problem in \eqref{fixedfinal}, we first need to show that $\widetilde{\mathcal F}(\mathcal B)\subset\mathcal B$. Assuming that $\bigl(\mathbf\Phi_1,(d_{10},\dots,d_{N0})\bigr)\in\mathcal B$, we know from \eqref{defofbi}, \eqref{defHinner}  and Proposition \ref{errorprop1} that $\operatorname{supp}_y b_i\cup\operatorname{supp}_y H_i\subset B_{\frac32}$, since $\operatorname{sup}|\varphi|=o(1)$ in the region $\bigcup\limits_{i=1}^{N} \operatorname{supp}\chi_{i},$ uniformly for fixed $\widehat{C}>0.$  

Consequently, extending the functions $H_i$ and $b_i$ by zero outside $B_{\rho_i}$, we find
\[\begin{aligned}
|H_i(\mathbf\Phi_1+\mathbf\Phi_2,\phi_{out})|
&\le C\left(\Gamma\left(\frac y2\right)\right)_+^{\gamma-2}\bigl(|\phi_{i,1}|^2+|\phi_{i,2}|^2+|\phi_{out}|^2\bigr)+|\widetilde E_i|+\bigl(\gamma\Gamma_+^{\gamma-1}+|b_i|\bigr)|\phi_{out}|\\
&\le C\ve\mu_i|\log\ve|\left(\Gamma\left(\frac y2\right)\right)_+^{\gamma-1},
\end{aligned}\]
which further yields
\begin{equation}\label{EstHfixedpoint}
\|H_i(\mathbf\Phi_1+\mathbf\Phi_2,\phi_{out})\|_{2+\sigma,\beta}\le C\ve\mu_i|\log\ve|.
\end{equation}
Since $D_y\phi_{i,2}$ and $D_y^2\phi_{i,2}$ vanish outside $B_1$ and $b_i$ is supported in $B_{\frac32}$, we also have $\operatorname{supp}\bar B_i[\phi_{i,2}]\subset B_{\frac32}$, hence using \eqref{oper}, \eqref{defofbi} and \eqref{estimateforphi2}, we find
\[\begin{aligned}
|\bar B_i[\phi_{i,2}]|
&\le C\ve\mu_i\bigl(|y||D_y^2\phi_{i,2}|+|D_y\phi_{i,2}|\bigr)+C\ve\mu_i|\log\ve|\left(\Gamma\left(\frac y2\right)\right)_+^{\gamma-2}|\phi_{i,2}|\\
&\le C\ve\mu_i|\log\ve||d_{i0}|\left(\Gamma\left(\frac y2\right)\right)_+^{\gamma-2},
\end{aligned}\]
and consequently
\begin{equation}\label{geometricnormbounds}
\|\bar B_i[\phi_{i,2}]\|_{2+\sigma,\beta}\le C\ve\mu_i|\log\ve||d_{i0}|\le C\ve\mu_i|\log\ve|.
\end{equation}
Combining \eqref{EstHfixedpoint}, \eqref{geometricnormbounds} and employing Proposition \ref{prop1}, we readily get
\begin{equation}\label{firsterr71}
\|\widetilde\phi_{i,1}\|_{*,\sigma}
=\bigl\|T_i[H_i(\mathbf\Phi_1+\mathbf\Phi_2,\phi_{out})+\bar B_i[\phi_{i,2}]]\bigr\|_{*,\sigma}
\le C_1\ve\mu_i|\log\ve|,
\end{equation}
for some $C_1>0$ independent of $\widehat{C}$.

Regarding $\widetilde d_{i0}$ given by \eqref{projcoeffinfixedpoint}, we first estimate the two terms involving $\bar B_i$. The preceding pointwise estimate gives
\begin{equation}\label{estimateBphi2Zo}
\Big|\int_{\R^2}\bar B_i[\phi_{i,2}]Z_0\Big|
\le C\ve\mu_i|\log\ve||d_{i0}|\int_{B_{\frac32}}\left(\Gamma\left(\frac y2\right)\right)_+^{\gamma-2}|Z_0|\,\dd y
\le C\ve\mu_i|\log\ve||d_{i0}|,
\end{equation}
while for the term involving the updated inner correction $\widetilde\phi_{i,1}=T_i[h_i]$, similar calculations as above give
\begin{equation}\label{estimateBphi1Zo}
\begin{aligned}
\Big|\int_{\R^2}\bar B_i[\widetilde\phi_{i,1}]Z_0\Big|
&\le C\ve\mu_i\int_{B_{2\rho_i}}\bigl(|y||D_y^2\widetilde\phi_{i,1}|+|D_y\widetilde\phi_{i,1}|\bigr)|Z_0|\,\mathrm dy+C\ve\mu_i|\log\ve|\int_{\R^2}\left(\Gamma\left(\frac y2\right)\right)_+^{\gamma-2}|\widetilde\phi_{i,1}Z_0|\,\mathrm dy\\
&\le C\left(\ve\mu_i(1+\rho_i)^{1-\sigma}\log(2+\rho_i)+\ve\mu_i|\log\ve|\right)\|\widetilde\phi_{i,1}\|_{*,\sigma}\\
&\le C(\ve\mu_i)^\sigma|\log\ve|^{2\sigma-1}\|h_i\|_{2+\sigma,\beta}.
\end{aligned}
\end{equation}
It remains to estimate the integral involving $H_i$, where \eqref{defHinner} provides the exact identity
\begin{equation}\label{estissue}
H_i(\mathbf\Phi_1+\mathbf\Phi_2,\phi_{out})
=\mathcal N_i(\varphi)+\widetilde E_i+\gamma\Gamma_+^{\gamma-1}\phi_{out}+b_i\phi_{out},
\end{equation}
with $\varphi$ defined in \eqref{gluingansatz}. Since $\mathcal N_i(\varphi)$ is supported in $B_{\frac32}$, its quadratic estimate yields
\begin{equation*}
\begin{aligned}
\Big|\int_{\R^2}\mathcal N_i(\varphi)Z_0\Big|
&\le C\int_{B_2}\left(\Gamma\left(\frac y2\right)\right)_+^{\gamma-2}
\bigl(|\phi_{i,1}|^2+|\phi_{i,2}|^2+|\phi_{out}|^2\bigr)|Z_0|\,\mathrm dy\\
&\le C\left(\|\phi_{i,1}\|_{*,\sigma}^2+|d_{i0}|^2+\|\phi_{out}\|_{L^\infty(B_1((r_0,0)))}^2\right)\\
&\le C\ve^2\mu_i^2|\log\ve|^2,
\end{aligned}
\end{equation*}
where we used \eqref{estimateforphi2}, \eqref{contractionball} and \eqref{outerbounddecomposed}.

Likewise, \eqref{defofbi} gives
\[
\Big|\int_{\R^2}b_i\phi_{out}Z_0\Big|
\le C\ve\mu_i|\log\ve|\|\phi_{out}\|_{L^\infty(B_1((r_0,0)))}
\le C\ve^2\mu_i^2|\log\ve|^2.
\]
Next, invoking again Proposition \ref{errorprop1}, the expansion of $\widetilde E_i$ in $B_{\frac32}$ reads
\[
\begin{aligned}
\widetilde E_i
&=\kappa_i\frac{3R_ih^2+R_i^3}{2h(h^2+R_i^2)^{\frac32}}\ve\mu_i y_1\Gamma_+^\gamma
+\kappa_i\gamma\Gamma_+^{\gamma-1}\Bigg[\nu'(1)\log\Big(1+\frac{\mu_i^*}{\mu_i^0}\Big)
+\bigl(c_{1,i}\ve\mu_i y_1+c_{2,i}\ve^2\mu_i^2|y|^2\bigr)\Gamma(y)\\
&\quad+\frac{R_i^3}{2h(h^2+R_i^2)^{\frac32}}H_1(\ve\mu_i y)
+\ve\mu_i|\log\ve|\bigl(y_1\widehat{\mathcal A}_{1,i}+y_2\widehat{\mathcal A}_{2,i}\bigr)\Bigg]+\mathcal O\Big(\ve^2\mu_i^2|\log\ve|^2\left(\Gamma\left(\frac y2\right)\right)_+^{\gamma-2}\Big).
\end{aligned}
\]
Since $\Gamma$ and $Z_0$ are radial functions, the terms linear in $y_1,y_2$ have zero integral against $Z_0$, while the term involving $H_1$ also has zero integral as  its angular dependence is $\cos(3\theta)$. 

Moreover, $\phi_{out}$ admits the expansion
\[
\phi_{out}(P_i+A_i\ve\mu_i y)
=\phi_{out}(P_i)+\mathcal O\left(\ve\mu_i|y|\|\nabla_x\phi_{out}\|_{L^\infty(B_1((r_0,0)))}\right).
\]
Collecting the above and integrating \eqref{estissue} against $Z_0$ we find
\begin{equation}\label{finalestHiZ_0}
\begin{aligned}
\int_{\R^2}H_i(\mathbf\Phi_1+\mathbf\Phi_2,\phi_{out})Z_0
&=\Big[\kappa_i\nu'(1)\log\Big(1+\frac{\mu_i^*}{\mu_i^0}\Big)+\phi_{out}(P_i)\Big]\int_{\R^2}\gamma\Gamma_+^{\gamma-1}Z_0\\
&\quad+\mathcal O\left(\ve^2\mu_i^2|\log\ve|^2+\ve\mu_i\|\nabla_x\phi_{out}\|_{L^\infty(B_1((r_0,0)))}\right)\\
&=\mathcal O\left(\ve^2\mu_i^2|\log\ve|^2\right),
\end{aligned}
\end{equation}
where the terms inside the bracket cancel by \eqref{scalingpertexplicitinproof} and the last equality follows from \eqref{outerbounddecomposed}.

As already alluded to in Remark \ref{decoupling}, this cancellation is essential for closing the fixed point argument for the coupled Inner–Outer system, as it removes the constant outer contribution $\phi_{out}(P_i)$ from the radial projection and enables us to obtain the bound on $d_{i0}$ in \eqref{contractionball}. Indeed, the estimate $|\phi_{out}(P_i)|\le C\ve^{1+\sigma^*}$ alone does not give the required bound on the radial coefficient, since $0<\sigma^*<\frac{\sigma}{2}$ and
\[\int_{\R^2}\gamma\Gamma_+^{\gamma-1}Z_0=-\frac{4\pi\nu'(1)}{\gamma-1}\ne0.\]
Combining \eqref{projcoeffinfixedpoint}, \eqref{estimateBphi2Zo}, \eqref{estimateBphi1Zo} and \eqref{finalestHiZ_0}, we infer that
\begin{equation}\label{contr2}
\begin{aligned}
|\widetilde d_{i0}|
&\le C\ve\mu_i|\log\ve||d_{i0}|+C(\ve\mu_i)^\sigma|\log\ve|^{2\sigma-1}\|h_i\|_{2+\sigma,\beta}+C\ve^2\mu_i^2|\log\ve|^2\\
&\le C_2(\ve\mu_i)^{1+\sigma}|\log\ve|^{1+2\sigma},
\end{aligned}
\end{equation}
for some $C_2>0$ independent of $\widehat{C}$.

Collecting \eqref{firsterr71}, \eqref{contr2}, and using \eqref{admissiblefullscales}, one can choose $\widehat C\ge 2\max\{e^{C_0}C_1,e^{(1+\sigma)C_0}C_2\}$ to obtain $\widetilde{\mathcal F}(\mathcal B)\subset\mathcal B$ for all sufficiently small $\ve>0$.

It remains to establish that $\widetilde{\mathcal F}$ is a contraction in $\mathcal B$. We take $\bigl(\mathbf\Phi_1^\ell,(d_{10}^\ell,\dots,d_{N0}^\ell)\bigr)\in\mathcal B$, $\,\,\ell=1,2$, and denote their associated scales, outer corrections, operators and sources by $\mu_i^\ell$, $\phi_{out}^\ell$, $\bar B_i^\ell$ and $h_i^\ell$, respectively.

By the linearity of \eqref{outermap} at fixed scales, \eqref{admissiblefullscales}, \eqref{outerparameterbounds} and \eqref{scalefixedpoint}, we have
\begin{equation}\label{phioutcontraction}
\begin{aligned}
&\max_i\Big|\log\Big(\frac{\mu_i^1}{\mu_i^2}\Big)\Big|
+\big\|(1+|\cdot|^{2})^{-1}(\phi_{out}^1-\phi_{out}^2)\big\|_{L^{\infty}(\R^2)}
+\|\phi_{out}^1-\phi_{out}^2\|_{C^{1,\beta}(B_1((r_0,0)))}\\
&\le C\sum_{j=1}^N\Big((\ve\bar\mu_j)^\sigma|\log\ve|^{4+2\sigma}\|\phi_{j,1}^1-\phi_{j,1}^2\|_{*,\sigma}
+|\log\ve|^4|d_{j0}^1-d_{j0}^2|\Big).
\end{aligned}
\end{equation}

For fixed \(i\in\{1,\dots,N\}\), we now regard \(H_i^\ell\), \(\ell=1,2\), as functions of the same variable \(y\in\mathbb R^2\), using \(x=P_i+A_i\varepsilon\mu_i^\ell y\) in the \(\ell\)-th expression. We know that \(\operatorname{supp}_y H_i^\ell\subset B_{\frac32}(0)\) for both \(\ell\), so their difference is supported in this fixed ball. Then, we use \eqref{scalingpertexplicitinproof} in each nonlinear argument to write
\[
\nu'(1)\log\bigg(1+\frac{\mu_i^{*,\ell}}{\mu_i^{0,\ell}}\bigg)
+\frac{\phi_{out}^\ell(P_i+A_i\ve\mu_i^\ell y)}{\kappa_i}
=\frac{\phi_{out}^\ell(P_i+A_i\ve\mu_i^\ell y)-\phi_{out}^\ell(P_i)}{\kappa_i},\qquad \ell=1,2.
\]
Using Taylor's formula for $s_+^\gamma$ together with \eqref{exppsiminusfinal}, \eqref{defofbi} and \eqref{H2scalebounds}, one can verify that
\[\begin{aligned}
|H_i^1-H_i^2|
&\le C\ve\bar\mu_i|\log\ve|\left(\Gamma\left(\frac y2\right)\right)_+^{\gamma-2}
\bigg(|\phi_{i,1}^1-\phi_{i,1}^2|+|\phi_{i,2}^1-\phi_{i,2}^2|
+\max_j\Big|\log\Big(\frac{\mu_j^1}{\mu_j^2}\Big)\Big|\bigg)\\
&\quad+C\ve\bar\mu_i\left(\Gamma\left(\frac y2\right)\right)_+^{\gamma-1}
\|\phi_{out}^1-\phi_{out}^2\|_{C^{1,\beta}(B_1((r_0,0)))}.
\end{aligned}\]
As a result, \eqref{H2scalebounds} and the identity  $\phi_{i,2}^\ell=d_{i0}^\ell\hat\phi_2$ give
\begin{equation}\label{Esth1minush2}
\begin{aligned}
\|H_i^1-H_i^2\|_{2+\sigma,\beta}
+\Big|\int_{\R^2}(H_i^1-H_i^2)Z_0\Big|&\le C\ve\bar\mu_i|\log\ve|
\bigg(\|\phi_{i,1}^1-\phi_{i,1}^2\|_{*,\sigma}+|d_{i0}^1-d_{i0}^2|
+\max_j\Big|\log\Big(\frac{\mu_j^1}{\mu_j^2}\Big)\Big|\bigg)\\
&\quad+C\ve\bar\mu_i\|\phi_{out}^1-\phi_{out}^2\|_{C^{1,\beta}(B_1((r_0,0)))}.
\end{aligned}
\end{equation}

Moreover, \eqref{defofbi} and \eqref{estimateforphi2} yield
\[
\begin{aligned}
|\bar B_i^1[\phi_{i,2}^1-\phi_{i,2}^2]|
&\le C\ve\bar\mu_i\bigl(|y||D_y^2(\phi_{i,2}^1-\phi_{i,2}^2)|+|D_y(\phi_{i,2}^1-\phi_{i,2}^2)|\bigr)+C\ve\bar\mu_i|\log\ve|\left(\Gamma\left(\frac y2\right)\right)_+^{\gamma-2}|\phi_{i,2}^1-\phi_{i,2}^2|\\
&\le C\ve\bar\mu_i|\log\ve||d_{i0}^1-d_{i0}^2|\left(\Gamma\left(\frac y2\right)\right)_+^{\gamma-2}.
\end{aligned}
\]
Since this expression vanishes outside $B_{\frac32}$, we obtain
\begin{equation}\label{estforBfinal}
\|\bar B_i^1[\phi_{i,2}^1-\phi_{i,2}^2]\|_{2+\sigma,\beta}
+\Big|\int_{\R^2}\bar B_i^1[\phi_{i,2}^1-\phi_{i,2}^2]Z_0\Big|
\le C\ve\bar\mu_i|\log\ve||d_{i0}^1-d_{i0}^2|.
\end{equation}
We continue by estimating the difference between $\bar{B}_{i}^1$ and $\bar{B}_{i}^2$. Keeping $P_j$ and $y$ fixed, the coefficient extension in Proposition \ref{prop1} and the identity $\ve\mu_i\rho_i=\frac{2\delta_1}{|\log\ve|^2}$ preserve the bounds in \eqref{oper} under logarithmic scale differentiation. Due to this, \eqref{defofbi}, \eqref{leadinglogscalebounds} and \eqref{H2scalebounds} give
\[
\max_{j=1,\dots,N}\|\mu_j\partial_{\mu_j}b_i\|_{C^{0,\beta}(B_2)}\le C
\]
uniformly in \eqref{admissiblefullscales}, before imposing \eqref{scalingpertexplicitinproof}. Since $b_i^1-b_i^2$ is supported in $B_{\frac32}$ and $\bar B_i^1-\bar B_i^2$ vanishes outside $B_{2\max\{\rho_i^1,\rho_i^2\}}$, the mean value theorem and \eqref{estimateforphi2},  \eqref{contractionball} and \eqref{firsterr71} yield
\[
\begin{aligned}
&\big\|(\bar B_i^1-\bar B_i^2)[\phi_{i,2}^2+\widetilde\phi_{i,1}^2]\big\|_{2+\sigma,\beta}
+\Big|\int_{B_{2\max\{\rho_i^1,\rho_i^2\}}}(\bar B_i^1-\bar B_i^2)[\phi_{i,2}^2+\widetilde\phi_{i,1}^2]Z_0\,\mathrm dy\Big|\\
&\le C\left[\|\widetilde\phi_{i,1}^2\|_{*,\sigma}
+|d_{i0}^2|\right]
\max_j\Big|\log\Big(\frac{\mu_j^1}{\mu_j^2}\Big)\Big|\\
&\le C\ve\bar\mu_i|\log\ve|\max_j\Big|\log\Big(\frac{\mu_j^1}{\mu_j^2}\Big)\Big|.
\end{aligned}
\]
In addition, since both projected problems are constructed using the same underlying inverse operator and projections in Proposition \ref{prop1} (see \eqref{coeffsforortho} and \eqref{fp}), subtracting their equations and using the identity \[\bar B_i^1[\widetilde\phi_{i,1}^1]-\bar B_i^2[\widetilde\phi_{i,1}^2]
=\bar B_i^1[\widetilde\phi_{i,1}^1-\widetilde\phi_{i,1}^2]
+(\bar B_i^1-\bar B_i^2)[\widetilde\phi_{i,1}^2]\]
allows us to write
\[\widetilde\phi_{i,1}^1-\widetilde\phi_{i,1}^2
=T_i^1\left[H_i^1-H_i^2+\bar B_i^1[\phi_{i,2}^1-\phi_{i,2}^2]
+(\bar B_i^1-\bar B_i^2)[\phi_{i,2}^2+\widetilde\phi_{i,1}^2]\right],
\]
while $\widetilde d_{i0}^1-\widetilde d_{i0}^2$ is given by the same formula, with $T_i^1$ replaced by the coefficient functional $d_{i0}^1[\cdot]$ in \eqref{coeffsforortho}.

Applying Proposition \ref{prop1}, \eqref{coeffsforortho} and \eqref{estimateBphi1Zo} to this identity, together with \eqref{Esth1minush2}, \eqref{estforBfinal} and the operator difference estimate above, we deduce
\begin{equation}\label{finalestimates2}
\begin{aligned}
\|\widetilde\phi_{i,1}^1-\widetilde\phi_{i,1}^2\|_{*,\sigma}
&\le C\ve\bar\mu_i|\log\ve|
\bigg(\|\phi_{i,1}^1-\phi_{i,1}^2\|_{*,\sigma}+|d_{i0}^1-d_{i0}^2|
+\max_j\Big|\log\Big(\frac{\mu_j^1}{\mu_j^2}\Big)\Big|\bigg)\\
&\quad+C\ve\bar\mu_i\|\phi_{out}^1-\phi_{out}^2\|_{C^{1,\beta}(B_1((r_0,0)))},
\end{aligned}
\end{equation}
as well as
\begin{equation}\label{finalestimates1}
\begin{aligned}
|\widetilde d_{i0}^1-\widetilde d_{i0}^2|
&\le C\ve\bar\mu_i|\log\ve|
\bigg(\|\phi_{i,1}^1-\phi_{i,1}^2\|_{*,\sigma}+|d_{i0}^1-d_{i0}^2|
+\max_j\Big|\log\Big(\frac{\mu_j^1}{\mu_j^2}\Big)\Big|\bigg)\\
&\quad+C\ve\bar\mu_i\|\phi_{out}^1-\phi_{out}^2\|_{C^{1,\beta}(B_1((r_0,0)))}.
\end{aligned}
\end{equation}
To conclude, substituting \eqref{phioutcontraction} into \eqref{finalestimates2} and \eqref{finalestimates1} and dividing by the component bounds in \eqref{contractionball} we get
\[\begin{aligned}
&\max_i\bigg(\frac{\|\widetilde\phi_{i,1}^1-\widetilde\phi_{i,1}^2\|_{*,\sigma}}{\ve\bar\mu_i|\log\ve|}
+\frac{|\widetilde d_{i0}^1-\widetilde d_{i0}^2|}{(\ve\bar\mu_i)^{1+\sigma}|\log\ve|^{1+2\sigma}}\bigg)\\
&\le C\ve^{1-\sigma}|\log\ve|^{\tilde\zeta}
\max_i\bigg(\frac{\|\phi_{i,1}^1-\phi_{i,1}^2\|_{*,\sigma}}{\ve\bar\mu_i|\log\ve|}+\frac{|d_{i0}^1-d_{i0}^2|}{(\ve\bar\mu_i)^{1+\sigma}|\log\ve|^{1+2\sigma}}\bigg)
\end{aligned}\]
for some fixed $\tilde\zeta>0$,  thus $\widetilde{\mathcal F}$ is a contraction in $\mathcal B$ for all sufficiently small $\ve>0$, and \eqref{fixedfinal} has a unique fixed point.
\end{proof}

\section{Balancing conditions and adjustment of the vortex locations}\label{reduced}
In Section \ref{Section projected}, we have constructed a solution $(\phi_{in},\phi_{out})=(\phi_{in,1},\dots,\phi_{in,N},\phi_{out})$ to the coupled system
\begin{equation*}
\Delta_y\phi_{in,i}+\gamma\Gamma_+^{\gamma-1}\phi_{in,i}+\bar B_i[\phi_{in,i}]+H_i(\phi_{in},\phi_{out})=\sum_{j=1}^2d_{ij}\gamma\Gamma_+^{\gamma-1}Z_j\quad\mbox{in }B_{\rho_i},\qquad i=1,\dots,N,
\end{equation*}
and
\[
L_x[\phi_{out}]+\sum_{i=1}^N\left(L_x[\tilde\eta_i\phi_{in,i}]-\tilde\eta_iL_x[\phi_{in,i}]\right)=0\quad\mbox{in }\R^2.
\]
For convenience, we recall that $\rho_i=\frac{2\delta_1}{\ve\mu_i|\log\ve|^2}$, $\bar B_i$ and $H_i$ are defined in \eqref{defbarBinner}-\eqref{defHinner}, and the kernel elements $Z_j$, $j=1,2$, are given in \eqref{elemkernel}. The solution and the full scales are selected in Proposition \ref{propsec8} and satisfy \eqref{boundsforsoln}.

To conclude the proof of Theorem \ref{maintheorem}, it remains to obtain an exact solution to the Inner-Outer system in \eqref{inprobexp} and \eqref{poisreduced} by solving
\begin{equation}\label{zerofunctionals}
d_{ij}\left[H_i(\phi_{in},\phi_{out})+\bar B_i[\phi_{i,2}]\right]=0,\qquad i=1,\dots,N,\quad j=1,2,
\end{equation}
where we recall that
\[d_{ij}\left[H_i(\phi_{in},\phi_{out})+\bar B_i[\phi_{i,2}]\right]=\gamma_j\int_{\R^2}\left(H_i(\phi_{in},\phi_{out})+\bar B_i[\phi_{i,2}]+\bar B_i[\phi_{i,1}]\right)Z_j,\quad \gamma_j^{-1}=\int_{\R^2}\gamma\Gamma_+^{\gamma-1}Z_j^2.\]
These conditions will be achieved by a suitable choice of the points in \eqref{formpoints}. In what follows, $\mathcal M(P)$ denotes a continuous function for each fixed $\ve>0$, uniformly bounded for configurations $P=(P_1,\dots,P_N)$ satisfying \eqref{formpoints} and \eqref{unifdistance}. Its expression may vary from line to line, or within the same line, and all estimates below are uniform in $i\in\{1,\dots,N\}$.

Similarly to Proposition \ref{propsec8}, we first note that
\[\Big|\int_{\R^2}\bar B_i[\phi_{i,1}]Z_j\Big|\le C\ve\mu_i|\log\ve|\|\phi_{i,1}\|_{*,\sigma},\qquad
\Big|\int_{\R^2}\bar B_i[\phi_{i,2}]Z_j\Big|\le C\ve\mu_i|\log\ve||d_{i0}|,\qquad j=1,2,\]
while \eqref{defHinner}, \eqref{nol}, \eqref{defofbi} and \eqref{boundsforsoln} give
\begin{equation*}
\begin{aligned}
\bigg|\frac{d_{ij}}{\gamma_j}-\int_{B_{\rho_i}}\widetilde E_iZ_j\bigg|
&\le C\left(\ve^2\mu_i^2|\log\ve|^2+\ve\mu_i|\log\ve|\|\phi_{out}\|_{C^1(B_1((r_0,0)))}+\ve\mu_i|\log\ve||d_{i0}|\right)\\
&\le C\ve^2\mu_i^2|\log\ve|^2.
\end{aligned}
\end{equation*}
Consequently, the conditions in \eqref{zerofunctionals} can be written as
\begin{equation}\label{conderror}
\int_{B_{\rho_i}}\widetilde E_iZ_j=\ve^2\mu_i^2|\log\ve|^2\mathcal M(P),\qquad i=1,\dots,N,\quad j=1,2.
\end{equation}
In light of \eqref{exppsiminusfinal} and Proposition \ref{errorprop1}, the integral against $Z_1$ takes the form
\begin{equation*}
\begin{aligned}
\int_{B_{\rho_i}}\widetilde E_iZ_1\,\mathrm dy
=\ve\mu_i\kappa_i\Bigg[&\frac{3R_i h^2+R_i^3}{2h(h^2+R_i^2)^{\frac32}}\int_{B_{\rho_i}}y_1\Gamma_+^\gamma Z_1\,\mathrm dy+c_{1,i}\int_{B_{\rho_i}}y_1\gamma\Gamma_+^{\gamma-1}\Gamma Z_1\,\mathrm dy\\
&+\big(\mathcal F_{1,i}(P)+\log|\log\ve|\mathcal M(P)\big)\int_{B_{\rho_i}}\gamma\Gamma_+^{\gamma-1}y_1Z_1\,\mathrm dy+\ve\mu_i|\log\ve|^2\mathcal M(P)\Bigg],
\end{aligned}
\end{equation*}
where
\[\mathcal F_{1,i}(P)=|\log\ve|\Bigg(-c_{1,i}\nu'(1)-\frac{\alpha R_i h}{\kappa_i\sqrt{h^2+R_i^2}}\Bigg)+\sum_{j\ne i}\frac{\kappa_j}{\kappa_i}\nu'(1)\frac{[A_j^{-1}(P_i-P_j)]_1}{|A_j^{-1}(P_i-P_j)|^2}.\]
Likewise, we compute
\begin{equation}\label{projectionZ2}
\int_{B_{\rho_i}}\widetilde E_iZ_2\,\mathrm dy=\ve\mu_i\kappa_i\left[\big(\mathcal F_{2,i}(P)+\mathcal M(P)\big)\int_{B_{\rho_i}}\gamma\Gamma_+^{\gamma-1}y_2Z_2\,\mathrm dy+\ve\mu_i|\log\ve|^2\mathcal M(P)\right],
\end{equation}
with
\[
\mathcal F_{2,i}(P)=\sum_{j\ne i}\frac{\kappa_j}{\kappa_i}\nu'(1)\frac{[A_j^{-1}(P_i-P_j)]_2}{|A_j^{-1}(P_i-P_j)|^2}.
\]
Using \eqref{formpoints} and \eqref{rescaledmainpoints}, we write
\[P_i=(r_0+\tilde r,0)+\frac{\widehat P_i}{|\log\ve|},\qquad \widetilde P_i=\left(\frac{\sqrt{h^2+r_0^2}}h\hat a_i,\hat b_i\right),\]
and combine \eqref{conderror}, \eqref{projectionZ2} and \eqref{expinverse} with the identity $\int_{\R^2}\gamma\Gamma_+^{\gamma-1}y_jZ_j=2\pi\nu'(1),\,j=1,2,$ to obtain
\begin{equation}\label{Reduced1}
\begin{aligned}
-\frac{\nu'(1)R_i h}{2(h^2+R_i^2)^{\frac32}}-\frac{\alpha R_i h}{\kappa_i\sqrt{h^2+R_i^2}}+\sum_{j\ne i}\nu'(1)\frac{\kappa_j}{\kappa_i}\frac{[\widetilde P_i-\widetilde P_j]_1}{|\widetilde P_i-\widetilde P_j|^2}&=\frac{\log|\log\ve|}{|\log\ve|}\mathcal M(P),\\
\sum_{j\ne i}\nu'(1)\frac{\kappa_j}{\kappa_i}\frac{[\widetilde P_i-\widetilde P_j]_2}{|\widetilde P_i-\widetilde P_j|^2}&=\frac{\log|\log\ve|}{|\log\ve|}\mathcal M(P).
\end{aligned}
\end{equation}
Identifying $\widetilde P_i=(\tilde a_i,\tilde b_i)\in\R^2$ with $\widetilde P_i=\tilde a_i+\mathrm{i}\tilde b_i\in\C$  (where $\mathrm{i}^2=-1$) and subtracting $\mathrm{i}$ times the second equation in \eqref{Reduced1} from the first yields a perturbation of the limit problem \eqref{balancingcond}, which in complex coordinates reads
\begin{equation*}
\sum_{j\ne i}\frac{\kappa_j}{\widetilde P_i-\widetilde P_j}=\frac{\kappa_ihr_0}{2(h^2+r_0^2)^{\frac32}}+\frac{\alpha hr_0}{\nu'(1)\sqrt{h^2+r_0^2}},\qquad i=1,\dots,N.
\end{equation*}
We then define
\begin{equation}\label{deflhsrhs}
\mathbf F_i(\widetilde P)=\sum_{j\ne i}\frac{\kappa_j}{\widetilde P_i-\widetilde P_j},\qquad U_i(r)=\frac{\kappa_ihr}{2(h^2+r^2)^{\frac32}}+\frac{\alpha hr}{\nu'(1)\sqrt{h^2+r^2}},
\end{equation}
so that \eqref{balancingcond} shows that $\mathbf F_i(\widetilde P^b)=U_i(r_0)$, while its derivative is given by
\begin{equation}\label{diffF}
\mathrm d\mathbf F_{\widetilde P^b}=\begin{pmatrix}
-\sum\limits_{j\ne1}\kappa_jT_{1j}&\kappa_2T_{12}&\cdots&\kappa_NT_{1N}\\
\kappa_1T_{21}&-\sum\limits_{j\ne2}\kappa_jT_{2j}&\cdots&\kappa_NT_{2N}\\
\vdots&\vdots&\ddots&\vdots\\
\kappa_1T_{N1}&\kappa_2T_{N2}&\cdots&-\sum\limits_{j\ne N}\kappa_jT_{Nj}
\end{pmatrix},\qquad T_{ij}=T_{ji}=\frac{1}{(\widetilde P_i^b-\widetilde P_j^b)^2}.
\end{equation}
Denoting \[\hat e=(1,\dots,1),\qquad \hat\kappa=(\kappa_1,\dots,\kappa_N),\] 
and using that the linearisation commutes with reflection, we deduce from the nondegeneracy assumption and the pairwise cancellation in \eqref{diffF} that
\[\operatorname{Ker}_{\C}(\mathrm d\mathbf F_{\widetilde P^b})=\operatorname{span}_{\C}\{\hat e\},\qquad
\operatorname{Ker}_{\C}\left((\mathrm d\mathbf F_{\widetilde P^b})^T\right)=\operatorname{span}_{\C}\{\hat\kappa\}.\]
In order to find a solution of \eqref{Reduced1} near $\widetilde P^b$, for $q_i=(q_{i,1},q_{i,2})\in\R^2$ we define
\[\widetilde P_i=\widetilde P_i^b+\tilde q_i,\qquad \tilde q_i=\frac{\sqrt{h^2+r_0^2}}h q_{i,1}+\mathrm{i}q_{i,2}, \quad i=1,\dots,N,\]
and in what follows we carry out the reduction in the reflection-symmetric subspace prescribed by \eqref{symmetryforpoints}, requiring $\tilde q_j=\overline{\tilde q_i}$ and $\kappa_j=\kappa_i$ whenever $\widetilde P_j^b=\overline{\widetilde P_i^b}$, and fix the remaining translation by imposing $\sum_{i=1}^{N}\tilde q_i=0$. Since the linear solution operators and coefficient extensions are chosen to commute with reflection, uniqueness in Proposition \ref{propsec8} gives equal scales at paired points and reflected corrections.

We equivalently write the system in \eqref{zerofunctionals} as
\begin{equation}\label{expreduced1}
\mathbf F_i(\widetilde P)=U_i(R_i)+\hat\sigma_i,\qquad
\hat\sigma_i=\mathbf F_i(\widetilde P)-U_i(R_i)-\frac{\frac{d_{i1}}{\gamma_1}-\mathrm{i}\frac{d_{i2}}{\gamma_2}}{2\pi[\nu'(1)]^2\ve\mu_i|\log\ve|},
\end{equation}
while using the fixed rotational speed in \eqref{uniformspeed} we further express \eqref{expreduced1} as
\begin{equation}\label{exp2reduced}
\mathrm d\mathbf F_{\widetilde P^b}(\tilde{\mathbf q})=\mathcal J(\tilde r,\tilde{\mathbf q})+\frac{\tilde r h}{2(h^2+r_0^2)^{\frac52}}\left((h^2-2r_0^2)\hat\kappa-h^2\frac{\sum_{i=1}^{N}\kappa_i^2}{\sum_{i=1}^{N}\kappa_i}\hat e\right),
\end{equation}
with
\[\mathcal J_i=U_i(R_i)-U_i(r_0)-U_i'(r_0)\tilde r+\hat\sigma_i-\left(\mathbf F_i(\widetilde P^b+\tilde{\mathbf q})-\mathbf F_i(\widetilde P^b)-[\mathrm d\mathbf F_{\widetilde P^b}(\tilde{\mathbf q})]_i\right).\]
We note that in the subspace considered above, it holds $\mathcal J(\tilde r,\tilde{\mathbf q})\cdot\hat\kappa\in\R$.

Due to the continuous dependence on $P_j$ in Proposition \ref{propsec8}, one finds that the exact coefficients in \eqref{expreduced1} and $\mathcal J=(\mathcal J_1,\dots,\mathcal J_N)$ are both continuous in $(\tilde r,\tilde{\mathbf q})$, while using the expansion
\begin{equation*}
|\log\ve|A_j^{-1}(P_i-P_j)=\widetilde P_i-\widetilde P_j+\frac{r_0\tilde r}{h^2+r_0^2}\operatorname{Re}(\widetilde P_i-\widetilde P_j)+\mathcal O\Big(\tilde r^2+\frac1{|\log\ve|}\Big),
\end{equation*}
we obtain
\begin{equation}\label{exactprojectionexpansion}
\begin{aligned}
\frac{\frac{d_{i1}}{\gamma_1}-\mathrm{i}\frac{d_{i2}}{\gamma_2}}{2\pi(\nu'(1))^2\,\ve\mu_i|\log\ve|}
&=\mathbf F_i(\widetilde P)-U_i(R_i)+\frac{r_0\tilde r}{h^2+r_0^2}\Big[\mathrm d\mathbf F_{\widetilde P}(\operatorname{Re}\widetilde P)\Big]_i
+\mathcal O\bigg(\frac{\log|\log\ve|}{|\log\ve|}+\tilde r^2\bigg).
\end{aligned}
\end{equation}
Since $R_i=r_0+\tilde r+\mathcal{O}(|\log\ve|^{-1})$ by \eqref{formpoints}, a Taylor expansion in \eqref{deflhsrhs}, together with \eqref{expreduced1} and \eqref{exactprojectionexpansion}, gives
\begin{equation}\label{propK}
\left|\mathcal J(\tilde r,\tilde{\mathbf q})+\frac{r_0\tilde r}{h^2+r_0^2}\,\mathrm d\mathbf F_{\widetilde P^b}(\operatorname{Re}\widetilde P^b)\right|
\le C\left(\frac{\log|\log\ve|}{|\log\ve|}+|\tilde r|^2+|\tilde{\mathbf q}|^2\right).
\end{equation}
To rewrite \eqref{exp2reduced} with this linear contribution on the left-hand side, we introduce the real linear map
\begin{equation}\label{reducedlinearoperator}
\mathcal L(\tilde{\mathbf q},\tilde r)
\defeq \mathrm d\mathbf F_{\widetilde P^b}\Big(\tilde{\mathbf q}+\frac{r_0\tilde r}{h^2+r_0^2}\operatorname{Re}\widetilde P^b\Big)
-\tilde r\bigl(U_1'(r_0),\dots,U_N'(r_0)\bigr),
\end{equation}
defined for reflection-symmetric $\tilde{\mathbf q}\in\C^N$ satisfying $\sum_{i=1}^{N}\tilde q_i=0$ and $\tilde r\in\R$, with values in the reflection-symmetric subspace of $\C^N$. Then, multiplying the $i$-th equation in \eqref{exp2reduced} by $\kappa_i$ and summing over $i$, we obtain the necessary solvability condition
\[\sum_{i=1}^N\kappa_i\mathcal J_i(\tilde r,\tilde{\mathbf q})
-\frac{\tilde r hr_0^2}{(h^2+r_0^2)^{\frac52}}\sum_{i=1}^N\kappa_i^2=0.\]
The same calculation shows that $\mathcal L$ is invertible. Indeed, if $\mathcal L(\tilde{\mathbf q},\tilde r)=0,$ then it holds 
\[0=\sum_{i=1}^N\kappa_i[\mathcal L(\tilde{\mathbf q},\tilde r)]_i
=-\tilde r\sum_{i=1}^N\kappa_iU_i'(r_0)
=\frac{\tilde r hr_0^2}{(h^2+r_0^2)^{\frac52}}\sum_{i=1}^N\kappa_i^2,\]
which implies that $\tilde r=0$, while the nondegeneracy together with $\sum_{i=1}^{N}\tilde q_i=0$ yields $\tilde{\mathbf q}=0$. Since $\mathcal L$ is injective and both its domain and codomain have real dimension $N$, we deduce that $\mathcal L$ is invertible with inverse independent of $\ve$.

Now, using \eqref{reducedlinearoperator} one finds
\[\mathcal L(\tilde{\mathbf q},\tilde r)-\frac{r_0\tilde r}{h^2+r_0^2}\,\mathrm d\mathbf F_{\widetilde P^b}(\operatorname{Re}\widetilde P^b)
=\mathrm d\mathbf F_{\widetilde P^b}(\tilde{\mathbf q})-\tilde r\bigl(U_1'(r_0),\dots,U_N'(r_0)\bigr),\]
thus \eqref{exp2reduced} is equivalent to the fixed point problem
\begin{equation}\label{reducedfixedpoint}
(\tilde{\mathbf q},\tilde r)
=\mathcal L^{-1}\Big[\mathcal J(\tilde r,\tilde{\mathbf q})+\frac{r_0\tilde r}{h^2+r_0^2}\,\mathrm d\mathbf F_{\widetilde P^b}(\operatorname{Re}\widetilde P^b)\Big].
\end{equation} 
Employing Proposition \ref{propsec8} and the definition of $\mathcal J$ once more, the right-hand side of \eqref{reducedfixedpoint} depends continuously on $(\tilde{\mathbf q},\tilde r)$, thus we seek a fixed point among reflection-symmetric vectors $\tilde{\mathbf q}\in\C^N$ and $\tilde r\in\R$ satisfying
\[\sum_{i=1}^N\tilde q_i=0,\qquad
|(\tilde{\mathbf q},\tilde r)|
\le K\frac{\log|\log\ve|}{|\log\ve|},\]
for some $K>0$ to be fixed below. After decreasing the fixed separation constant $d$ in \eqref{unifdistance} if necessary, all configurations in this ball maintain the uniform separation for all sufficiently small $\ve$, while \eqref{propK} and the boundedness of $\mathcal{L}^{-1}$ give
\[\begin{aligned}
\Big|\mathcal L^{-1}\Big[\mathcal J(\tilde r,\tilde{\mathbf q})+\frac{r_0\tilde r}{h^2+r_0^2}\,\mathrm d\mathbf F_{\widetilde P^b}(\operatorname{Re}\widetilde P^b)\Big]\Big|&\le C_1\frac{\log|\log\ve|}{|\log\ve|}+C_2K^2\left(\frac{\log|\log\ve|}{|\log\ve|}\right)^2\\
&\le K\frac{\log|\log\ve|}{|\log\ve|},
\end{aligned}\]
provided $K>2C_1$ is fixed and $\ve$ is sufficiently small. Taking the fixed constant in \eqref{formpoints} sufficiently large, Brouwer's fixed point theorem provides a solution $(\tilde{\mathbf q},\tilde{r}_{*})$ of \eqref{reducedfixedpoint} satisfying the desired bound.

Collecting the above and setting 
\[q_i=\left(\frac{h}{\sqrt{h^2+r_0^2}}\operatorname{Re}\tilde q_i,
\operatorname{Im}\tilde q_i\right),\qquad i=1,\dots,N,\] the points
\begin{equation}\label{finalconstructedpoints}
P_i=\left(r_0+\tilde r_*+\frac{h}{\sqrt{h^2+r_0^2}}\frac{\operatorname{Re}(\widetilde P_i^b+\tilde q_{i})}{|\log\ve|},\;\frac{\operatorname{Im}(\widetilde P_i^b+\tilde q_{i})}{|\log\ve|}\right),\qquad i=1,\dots,N,
\end{equation}
satisfy \eqref{pointstheorem} and \eqref{zerofunctionals}.

To conclude, it remains to establish the support and convergence properties in Theorem \ref{maintheorem}. For clarity, we use the notation $x=(x',x_3)\in\R^3,$ with $x'=(x_1,x_2)\in\R^2.$ We first note that \eqref{nonlinearity}, \eqref{exppsiminusfinal} and \eqref{boundsforsoln} show that the stationary scalar vorticity $W$ satisfies
\[\ve^2\mu_i^2\,W(P_i+A_i\ve\mu_i y)=\kappa_i\big(\Gamma(y)\big)^{\gamma}_{+}+\mathcal{O}(\ve\mu_i|\log\ve|),\qquad |y|\le \frac 32,\]
uniformly in $i\in\{1,\dots,N\}$. In addition, the negativity argument in Proposition \ref{errorprop1} together with \eqref{boundsforsoln} gives
\begin{equation}\label{finalsupport}
\operatorname{supp} W \subset \bigcup_{i=1}^{N}\left(P_i+A_i\ve\mu_i \overline{B}_{\frac 32}\right) \subset \bigcup_{i=1}^{N} B_{\ve^{1-\zeta}}(P_i)
\end{equation}
for every fixed $\zeta\in(0,1)$ and sufficiently small $\ve,$ where the latter balls are pairwise disjoint due to \eqref{scalingboundsinprop} and \eqref{unifdistance}.  

Due to the identities 
\[\det{A_i}=\frac{h}{\sqrt{h^2+R_i^2}},\qquad \int_{\R^2}\Gamma^{\gamma}_{+}=-2\pi\nu'(1),\]
the change of variables $x'=P_i+A_i\ve\mu_i y$ then yields
\begin{equation}\label{corecirculation}
\int_{B_{\ve^{1-\zeta}}(P_i)} W(x')\,\dd x'=-\frac{2\pi h\nu'(1)\kappa_i}{\sqrt{h^2+R_i^2}}+\mathcal{O}(\ve\mu_i|\log\ve|)=c_{*}\kappa_i+\mathcal{O}\Big(\frac{\log|\log\ve|}{|\log\ve|}\Big),
\end{equation}
where we used \eqref{finalconstructedpoints} and the definition of $c_{*}$ in \eqref{circulationconstant}, while the same rescaling and support inclusion also give
\begin{equation}\label{absolutemassandcenter}
\int_{B_{\ve^{1-\zeta}}(P_i)}|W(x')|\,\dd x' \leq C, \qquad \int_{B_{\ve^{1-\zeta}}(P_i)}|x'-P_i||W(x')|\,\dd x'\leq C \ve\mu_i.
\end{equation}
Taking into account \eqref{rotansatz} and the three-dimensional helical reconstruction in \eqref{omegaHelical}, we have
\[\vec\omega_{\ve}(x,|\log\ve|^{-1}\,\tau)=\frac{1}{h}W(R_{-\theta}x')(-x_2,x_1,h)^{T}, \qquad \theta=\frac{x_3}{h}-\alpha\tau,\]
thus \eqref{finalsupport} shows \eqref{remark1prop} for all $(x_3,\tau)\in\R^2,$ while \eqref{circulationfull} follows from the identity $W(x')=\omega_{\ve,3}(x',0,0)$.

For fixed $\tau\in\R$ and any $\widetilde{\mathbf{F}}\in C_{c} (\R^3;\R^3)$, the change of variables $x'\mapsto R_{\theta}x'$ gives
\[\begin{aligned}
\int_{\R^3}\vec{\omega}_{\ve}(x,|\log\ve|^{-1}\tau)\cdot\widetilde{\mathbf{F}}(x)\,\dd x&=\int_{\R}\int_{\R^2}W(x')\,\widetilde{\mathbf{F}}\big(R_{\theta}x',x_3\big)\cdot\frac{\dd}{\dd x_3}(R_{\theta}x',x_3)\,\dd x'\, \dd x_3\\
&=c_{*}\sum_{i=1}^{N}\kappa_i\int_{\R}\widetilde{\mathbf{F}}\big(R_{\theta}P_i,x_3\big)\cdot\frac{\dd}{\dd x_3}(R_{\theta}P_i,x_3)\,\dd x_3 + o(1)\\
&=c_{*}\sum_{i=1}^{N}\kappa_i\int_{\mathcal{G}_{i}^{\ve}(\tau)}\widetilde{\mathbf{F}}\cdot \mathbf{T}_{\mathcal{G}_{i}^{\ve}(\tau)}\,\dd s + o(1),
\end{aligned}\]
with $o(1)\to 0$ as $\ve\to 0$ by the uniform continuity of $\widetilde{\mathbf{F}}$ in a fixed compact set, the shrinking core supports in \eqref{finalsupport}, \eqref{corecirculation}, and the bounds in \eqref{absolutemassandcenter}. In the last equality, $\mathcal{G}_{i}^{\ve}(\tau)$ is the helix in \eqref{helicesdefshorter}, $\dd s$ denotes arclength along $\mathcal{G}_i^{\ve}(\tau)$ and $\mathbf{T}_{\mathcal{G}_{i}^{\ve}(\tau)}$ is oriented by increasing $x_3$, so that  $\mathbf{T}_{\mathcal{G}_{i}^{\ve}(\tau)}\,\dd s= \frac{\dd}{\dd x_3}\big(R_{\theta}P_i,x_3\big)\,\dd x_3$.

In a similar manner, using \eqref{finalconstructedpoints} we obtain
\[\int_{\R^3}\vec{\omega}_{\ve}(x,|\log\ve|^{-1}\,\tau)\cdot\widetilde{\mathbf{F}}(x)\,\dd x =c_{*}\left(\sum_{i=1}^{N}\kappa_i\right)\int_{\mathcal{G}_0(\tau)}\widetilde{\mathbf F}\cdot\mathbf{T}_{\mathcal{G}_0(\tau)}\,\dd s +o(1),\]
where $\mathcal{G}_0(\tau)$ denotes the limiting helix in \eqref{limitinghelix}. 

Finally, for fixed $\gamma>3$, we know that $s\mapsto s^{\gamma}_{+}$ belongs to $C_{\mathrm{loc}}^{\lfloor\gamma\rfloor,\gamma-\lfloor\gamma\rfloor}(\R)$ if $\gamma\notin \Z$ and to $C_{\mathrm{loc}}^{\gamma-1,1}(\R)$ if $\gamma\in\Z,$ where $\lfloor\cdot\rfloor$ denotes the integer part function. For each fixed $\ve>0$, combining \eqref{nonlinearity}, \eqref{finalsupport} and \eqref{truesol} we find
\[W(x')=F\Big(x',\Psi(x')-\frac{\alpha}{2}|\log\ve|\,|x'|^2\Big)\in
\begin{cases}
C^{\lfloor\gamma\rfloor,\gamma-\lfloor\gamma\rfloor}(\R^2),
&\gamma \notin \Z,\\[2pt]
C^{\gamma-1,1}(\R^2),
&\gamma\in\Z,
\end{cases}
\]
while \eqref{rotansatz}, \eqref{omegaHelical} and the periodicity in $x_3$ and $t$ further yield
\[\vec\omega_\ve(x,t)\in
\begin{cases}
C^{\lfloor\gamma\rfloor,\gamma-\lfloor\gamma\rfloor}(\R^3\times\R;\R^3),
&\gamma\notin\Z,\\[2pt]
C^{\gamma-1,1}(\R^3\times\R;\R^3),
&\gamma\in\Z.
\end{cases}
\]
Accordingly, for every $\beta\in(0,1)$, the stationary stream function satisfies
\[\Psi(x')\in
\begin{cases}
C_{\mathrm{loc}}^{\lfloor\gamma\rfloor+2,\gamma-\lfloor\gamma\rfloor}(\R^2),
&\gamma\notin\Z,\\[2pt]
C_{\mathrm{loc}}^{\gamma+1,\beta}(\R^2),
&\gamma\in\Z.
\end{cases}\]
Since the associated no-swirl velocity field is recovered by
\[\mathbf u_\ve(x,t)
=\frac{h}{h^2+|x'|^2}
\nabla_x\left[\Psi\left(R_{\alpha|\log\ve|t-\frac{x_3}{h}}x'\right)\right]
\times\begin{pmatrix}-x_2\\x_1\\h\end{pmatrix},
\qquad (x,t)\in\R^3\times\R,\]
we also infer
\[\mathbf u_\ve(x,t)\in
\begin{cases}
C_{\mathrm{loc}}^{\lfloor\gamma\rfloor+1,\gamma-\lfloor\gamma\rfloor}(\R^3\times\R;\R^3),
&\gamma\notin\Z,\\[2pt]
C_{\mathrm{loc}}^{\gamma,\beta}(\R^3\times\R;\R^3),
&\gamma\in\Z.
\end{cases}\]
This concludes the proof of Theorem \ref{maintheorem}.
\section{Appendix}
In this section, we provide the proofs of the Inner linear solvability results stated in Section \ref{Section 7}. 
\begin{proof}[Proof of Lemma \ref{lemat}]
Let $y=re^{i\theta}, \, r=|y|$. We decompose $\phi$ and $h$ in Fourier series as 
\[ h(y) =  \sum_{k\in\Z}  h_k (r) e^{ik\theta} , \quad \phi(y)= \sum_{k\in\Z}\phi_{k} (r) e^{ik\theta},\]
to obtain the ODEs  
\begin{equation}\label{01}
L_k [\phi_{k}] + h_k(r) = 0 , \quad r\in(0,\infty), \quad k \in \Z,
\end{equation}
with
\[L_k\defeq \partial_{rr} + \frac 1r \partial_{r}- \frac{k^2}{r^2}+  \gamma\Gamma^{\gamma-1}_{+}.\]
For $k=0$, we have that $z_0(r) = \frac{2}{ \gamma-1}\Gamma(r)+r\Gamma'(r)$ satisfies $L_0[z_0]=0,$ hence we choose $\phi_0$ as the solution of \eqref{01} that is regular at the origin and bounded at infinity. This solution exists uniquely and depends linearly on $h_0$, as any regular particular solution has at most logarithmic growth at infinity, and adding a suitable multiple of the homogeneous solution \[z_0(r)=\frac{2\nu'(1)}{\gamma-1}\log r +\nu'(1),\quad r\geq 1\]
removes its logarithmic term. The Wronskian identity $[r(z_0\phi_0'-z_0'\phi_0)]'=-rh_0(r)z_0(r)$ together with $(r\phi_0')'=-rh_0$ for $r\geq 1$ and $\int_{0}^{\infty}h_0(r)z_0(r)r\,\dd r=\frac{1}{2\pi}\int_{\R^2}h(y)Z_0(y)\,\dd y$ then yields\[\phi_{0}(r)=\frac{\gamma-1}{4\pi\nu'(1)}\int_{\R^2}hZ_0\,\dd y -\int_{r}^{\infty}\frac{\dd s}{s}\int_{s}^{\infty}t h_0(t)\,\dd t, \quad r\geq 1.\]
Using this formula and the regular radial equation on $[0,1],$ we deduce 
\[|\phi_{0}(r)| \, \le \,   C\Big[ \,  \Big| \int_{\R^2} h(y)Z_0(y)\, \dd y      \Big|   \, +\,  (1+r)^{-\sigma} \|h\|_{2+ \sigma} \Big ].\]
For $|k|=1$, one can verify that  $z_k(r)=-\Gamma'(r)$ satisfies $L_k[z_k]=0$, thus a regular solution reads
\[\phi_{k}  (r) =    - \Gamma'(r)\int_0^r  \frac {\dd s}{ s\Gamma'(s)^2}  \int_0^s h_k(\rho) \Gamma'(\rho) \rho\, \dd\rho,\]
with
\[|\phi_{k}(r)| \, \le \,   C\Big[ \,  (1+ r) \sum_{i=1}^2 \Big| \int_{\R^2} h(y)Z_i(y)\, \dd y      \Big|   \, +\,  (1+r)^{-\sigma} \|h\|_{ 2+\sigma} \Big ].\]
For $k=2$, let $z_2$ be the regular solution of $L_2[z_2]=0$ normalised by $\lim\limits_{r\to0}\frac{z_2(r)}{r^2}=1$. Since $z_1=-\Gamma'>0$ for $r>0$, the identity $[r(z_1z_2'-z_1'z_2)]'=\frac{3}{r}z_1z_2$ rules out a first zero of $z_2$ and also shows that $\frac{z_2}{z_1}$ is strictly increasing. As $z_1=-\frac{\nu'(1)}{r}$ for $r\geq1$, the coefficient of $r^2$ in the exterior solution $z_2$ is positive, hence there exist $c_1,c_2>0$ such that
\[c_1r^2\leq z_2(r)\leq c_2r^2,\qquad r>0.\]
Then, we introduce the positive function
\[\tilde\phi_2(r)=z_2(r)\int_r^\infty\frac{\dd s}{sz_2(s)^2}\int_0^s\frac{z_2(t)t}{(1+t)^{2+\sigma}}\,\dd t,\]
which satisfies, after adjusting $c_1,c_2$ if necessary,
\[-L_2[\tilde\phi_2]=\frac{1}{(1+r)^{2+\sigma}},\qquad c_1\frac{r^2}{(1+r)^{2+\sigma}}\leq\tilde\phi_2(r)\leq c_2(1+r)^{-\sigma}.\]
Moreover, since $\tilde\phi_2(r)=\mathcal O(r^2|\log r|)$ as $r\to0$, it extends continuously by $\tilde\phi_2(0)=0$. Thus, for every $|k|\geq2$, we have
\[-L_k[\tilde\phi_2]=\frac{1}{(1+r)^{2+\sigma}}+\frac{k^2-4}{r^2}\tilde\phi_2\geq\frac{ck^2}{(1+r)^{2+\sigma}},\]
for some $c>0$ independent of $k$.

In fact, since $|h_k(r)|\leq\frac{\|h\|_{2+\sigma}}{(1+r)^{2+\sigma}}$, we have that the function $\frac{M}{k^2}\|h\|_{2+\sigma}\,\tilde\phi_2$ is a positive supersolution for a sufficiently large constant $M>0$ independent of $k$. Solving on bounded annuli with zero boundary conditions and applying the comparison principle to the real and imaginary parts (justified by $z_1>0$ and $-L_k[z_1]=\frac{k^2-1}{r^2}z_1>0$), by passing to the limit we obtain a solution that satisfies
\[|\phi_k(r)|\leq\frac{C}{k^2}(1+r)^{-\sigma}\|h\|_{2+\sigma},\qquad |k|\geq2.\]
Combining the above estimates for each mode, the series
\[ \phi(y)=\sum_{k\in\Z}\phi_k(r)e^{ik\theta}\]
converges locally uniformly, and its sum solves \eqref{Linearisedinner} in the sense of distributions. This construction defines a linear operator of $h$ satisfying
\begin{equation}\label{cot}
\begin{aligned}
|\phi(y)|
\le C\bigg[\Big|\int_{\R^2}hZ_0\Big|+(1+|y|)\sum_{i=1}^2\Big|\int_{\R^2}hZ_i\Big|+(1+|y|)^{-\sigma}\|h\|_{2+\sigma}\bigg].
\end{aligned}
\end{equation}
In particular, interior elliptic regularity gives $\phi\in W^{2,p}_{\mathrm{loc}}(\R^2)$ for every $1<p<\infty$,
while under the further assumption $\|h\|_{2+\sigma,\beta}<+\infty$, Schauder estimates yield $\phi\in C^{2,\beta}_{\mathrm{loc}}(\R^2)$. 

We next establish the corresponding weighted estimates for the first and second derivatives of $\phi$. We let $R=|y|>2$, fix the direction $\vec e=\frac{y}{|y|}\in\mathbb S^1$ so that $y=R\vec e$, and consider the rescaled functions
\[\phi_R(\ell)=R^\sigma\phi(y+R\ell),\qquad h_R(\ell)=R^{2+\sigma}h(y+R\ell).\]
In the neighbourhood $U_y\defeq\{y+R\ell:|\ell|<\frac12\}$, one finds
\[\Delta_\ell\phi_R(\ell)+R^2\gamma\Gamma_+^{\gamma-1}(y+R\ell)\phi_R(\ell)+h_R(\ell)=0,\]
where the potential term vanishes since $|y+R\ell|>\frac R2>1$.

Setting
\[m_i\defeq\Big|\int_{\R^2}hZ_i\Big|,\qquad i=0,1,2,\]
we use \eqref{cot}, the fact that $\|h_R\|_{L^\infty(B_{\frac12}(0))}\le C\|h\|_{2+\sigma}$, and an interior $W^{2,p}$ estimate with $p>2$ to find
\begin{equation}\label{holder1}
\|\nabla_\ell\phi_R\|_{L^\infty(B_{\frac14}(0))}+\|\phi_R\|_{L^\infty(B_{\frac12}(0))}
\le C\Big[m_0R^\sigma+\sum_{i=1}^2m_iR^{1+\sigma}+\|h\|_{2+\sigma}\Big].
\end{equation}
In addition, since $[h_R]_{B_{\frac12}(0),\beta}\le C\|h\|_{2+\sigma,\beta}$, a combination of interior Schauder estimates and \eqref{holder1} gives
\begin{equation}\label{holder2}
\|D_\ell^2\phi_R\|_{L^\infty(B_{\frac14}(0))}+[D_\ell^2\phi_R]_{B_{\frac14}(0),\beta}
\le C\Big[m_0R^\sigma+\sum_{i=1}^2m_iR^{1+\sigma}+\|h\|_{2+\sigma,\beta}\Big].
\end{equation}
Collecting together \eqref{cot}-\eqref{holder2} and rescaling back to the original variable gives the asserted estimates for $|y|>2$, since $B_{\frac{1+R}{8}}(y)\subset B_{\frac R4}(y)$. The estimates for $|y|\le2$ follow from the usual interior elliptic estimates on $B_4(0)$ and \eqref{cot}.
\end{proof}
\begin{proof}[Proof of Proposition \ref{prop1}]
As a preliminary step, we consider a standard linear extension operator $h\mapsto\tilde h$ to the whole $\R^2$, so that $\tilde h=h$ in $B_{\rho_i},$  $\operatorname{supp}\tilde h\subset B_{2\rho_i}$, and $\|\tilde h\|_{2+\sigma,\beta}\le C\|h\|_{2+\sigma,\beta,B_{\rho_i}}$, with $C>0$  independent of $\rho_i$. 
In addition, we extend the coefficients of $\widetilde B_0$ to $C^2(\R^2)$ functions supported in $B_{2\rho_i}$ by dilation of a fixed bounded linear extension from $B_1$ to $\R^2,$ with both the source and coefficient extensions chosen to commute with reflection. Accordingly, we also extend $b_i$ by zero, noting that \eqref{defofbi} gives $\operatorname{supp}b_i \subset B_{\frac 32}$ and $\|b_i\|_{C^{0,\beta}}\leq C\ve\mu_i|\log\ve|.$

In this regard, we study the auxiliary projected problem in $\R^2$ given by
\begin{equation}\label{001}
\Delta_y\phi_i+\gamma\Gamma_+^{\gamma-1}\phi_i+\bar B_i[\phi_i]+\tilde h(y)
=\sum_{j=0}^2d_{ij}\gamma\Gamma_+^{\gamma-1}Z_j\quad\mbox{in }\R^2,
\end{equation}
where $d_{ij}=d_{ij}[h,\phi_i]$ are the scalars defined as
\begin{equation}\label{coeffsforortho}
d_{ij}=\gamma_j\int_{\R^2}\bigl(\tilde h+\bar B_i[\phi_i]\bigr)Z_j,\qquad
\gamma_j^{-1}=\int_{\R^2}\gamma\Gamma_+^{\gamma-1}Z_j^2,\qquad j=0,1,2.
\end{equation}
Using the coefficient expansion in \eqref{oper} and the estimate for $b_i$, we have
\[
|\bar B_i[\phi_i]|\le C\ve\mu_i\bigl(|y||D_y^2\phi_i|+|D_y\phi_i|\bigr)
+C\ve\mu_i|\log\ve|\left(\Gamma\left(\frac y2\right)\right)_+^{\gamma-2}|\phi_i|
\quad\mbox{in }B_{2\rho_i},
\]
while $\bar B_i[\phi_i]$ vanishes outside this ball. For $j=1,2$, since $|Z_j(y)|\le C(1+|y|)^{-1}$, the definition of the norm in \eqref{starnorm} gives
\[\begin{aligned}
\Big|\int_{\R^2}\bar B_i[\phi_i]Z_j\Big|
&\le C\ve\mu_i\|\phi_i\|_{*,\sigma}
\bigg(\int_{B_{2\rho_i}}\frac{\mathrm dy}{(1+|y|)^{2+\sigma}}+|\log\ve|\bigg)\\
&\le C\ve\mu_i|\log\ve|\|\phi_i\|_{*,\sigma}.
\end{aligned}\]
On the other hand, since $|Z_0(y)|\le C\log(2+|y|)$, for $j=0$ we obtain
\[
\begin{aligned}
\Big|\int_{\R^2}\bar B_i[\phi_i]Z_0\Big|
&\le C\ve\mu_i\int_{B_{2\rho_i}}\bigl(|y||D_y^2\phi_i|+|D_y\phi_i|\bigr)|Z_0|\,\mathrm dy +C\ve\mu_i|\log\ve|\int_{\R^2}\left(\Gamma\left(\frac y2\right)\right)_+^{\gamma-2}|\phi_iZ_0|\,\mathrm dy\\
&\le C(\ve\mu_i)^\sigma|\log\ve|^{2\sigma-1}\|\phi_i\|_{*,\sigma},
\end{aligned}
\]
where we used $0<\sigma<1$, the definition of $\rho_i$ and \eqref{scalingboundsinprop}.

In addition, for the source extension $\tilde h$, we observe that for $j=1,2$ it holds
\[\int_{\R^2\setminus B_{\rho_i}}\tilde hZ_0
=\mathcal O\bigg(\frac{\log(2+\rho_i)}{1+\rho_i^\sigma}\bigg)\|h\|_{2+\sigma,\beta,B_{\rho_i}},
\qquad
\int_{\R^2\setminus B_{\rho_i}}\tilde hZ_j
=\mathcal O\bigg(\frac{1}{1+\rho_i^{1+\sigma}}\bigg)\|h\|_{2+\sigma,\beta,B_{\rho_i}}.\]
Moreover, the coefficient bounds above, including their scaled H\"older bounds, and \eqref{starnorm} yield
\[\|\bar B_i[\phi_i]\|_{2+\sigma,\beta}
\le C\left(\ve\mu_i\rho_i+\ve\mu_i|\log\ve|\right)\|\phi_i\|_{*,\sigma}
\le\frac{C\delta_1}{|\log\ve|^2}\|\phi_i\|_{*,\sigma}\]
for sufficiently small $\ve>0$, where all norms are taken over $\R^2$. Since $\ve\mu_i\rho_i=\frac{2\delta_1}{|\log\ve|^2}$ is independent of the full scales, the same extension also preserves the coefficient bounds of $\widetilde{B}_0$ after logarithmic scale differentiation with the points fixed.

In order to solve \eqref{001}, we work in the Banach space $\mathcal X\defeq\{\phi_i\in C^{2,\beta}_{\mathrm {loc}}(\R^2):\|\phi_i\|_{*,\sigma}<+\infty\}$ and use the linear operator $\mathcal T$ of Lemma \ref{lemat} to reformulate the equation as
\begin{equation}\label{fp}
\phi_i=\mathcal A[\phi_i]+\mathcal H,\qquad \phi_i\in\mathcal X,
\end{equation}
where
\[\mathcal A[\phi_i]
=\mathcal T\Big[\bar B_i[\phi_i]-\sum_{j=0}^2d_{ij}[0,\phi_i]\gamma\Gamma_+^{\gamma-1}Z_j\Big],
\qquad
\mathcal H
=\mathcal T\Big[\tilde h-\sum_{j=0}^2d_{ij}[h,0]\gamma\Gamma_+^{\gamma-1}Z_j\Big].\]
Employing \eqref{coeffsforortho} and Lemma \ref{lemat}, one can verify that
\[\|\mathcal A[\phi_i]\|_{*,\sigma}\le\frac{C\delta_1}{|\log\ve|^2}\|\phi_i\|_{*,\sigma},
\qquad
\|\mathcal H\|_{*,\sigma}\le C\|h\|_{2+\sigma,\beta,B_{\rho_i}},\]
hence the contraction mapping Theorem in $\mathcal X$ yields a unique solution to \eqref{fp} for all sufficiently small $\ve>0$, which defines a linear operator of $h$ and satisfies
\[\|\phi_i\|_{*,\sigma}\le C\|h\|_{2+\sigma,\beta,B_{\rho_i}}.\]
We then set $T_i[h]=\phi_i\big|_{B_{\rho_i}},$ which solves the projected problem with the required estimate. For a prescribed global source $h$ with $\|h\|_{2+\sigma,\beta}<+\infty,$ the same construction applies directly with $\tilde h=h,$ yielding a global solution with 
\[\|\phi_i\|_{*,\sigma}\leq C \|h\|_{2+\sigma,\beta}.\]
\end{proof}
\section*{Acknowledgements}
During the preparation of this manuscript, the authors used ChatGPT (OpenAI) to improve the clarity of the exposition. The authors reviewed and edited all suggested changes and take full responsibility for the content of the manuscript.

\bibliographystyle{plain}
\bibliography{references}
\end{document}